\documentclass[12pt]{amsart}
\usepackage{
  a4wide, 
  amsmath,
  amsfonts,
  amssymb,
  graphicx, 
  times,
  color,
  hyphenat,
  pinlabel,
  stmaryrd, mathabx 
}
\input xy
\xyoption{all}
\newcommand*{\medcup}{\mathbin{\scalebox{0.82}{\ensuremath{\bigcup}}}}%
\newcommand{\sll}{{\mathfrak{sl}_{n+1}}}
\newcommand{\sln}{{\mathfrak{sl}_{n}}}

\newcommand{\und}[1]{{\underline{#1}}}

\newcommand{\llambda}{{\bar{\lambda}}}
\newcommand{\lmu}{{\bar{\mu}}}

\newcommand{\n}{\noindent}

\DeclareMathOperator{\amod}{mod}
\DeclareMathOperator{\fmod}{fmod}
\DeclareMathOperator{\prmod}{pmod}

\DeclareMathOperator{\Fun}{Fun}

\DeclareMathOperator{\End}{End}
\DeclareMathOperator{\Ext}{Ext}
\DeclareMathOperator{\ext}{ext}
\DeclareMathOperator{\Cxt}{Coext}
\DeclareMathOperator{\CExt}{CoExt}

\DeclareMathOperator{\gdim}{gdim}

\DeclareMathOperator{\seq}{Seq}

\newcommand{\figins}[3] 
{\raisebox{#1pt}{\includegraphics[height=#2 in]{figs/#3}}}

\newtheorem{thm}{Theorem}[section]
\newtheorem{lem}[thm]{Lemma}
\newtheorem{cor}[thm]{Corollary}
\newtheorem{prop}[thm]{Proposition}

\newtheorem{conj}[thm]{Conjecture}

\theoremstyle{definition}
\newtheorem{defn}[thm]{Definition}

\newcommand{\bN}{\mathbb{N}}
\newcommand{\bZ}{\mathbb{Z}}
\newcommand{\bQ}{\mathbb{Q}}

\newcommand{\cD}{\mathcal{D}}

\newcommand{\cG}{\mathcal{G}}

\newcommand{\cQ}{\mathcal{Q}}

\newcommand{\cS}{\mathcal{S}}
\newcommand{\cT}{\mathcal{T}}
\newcommand{\cU}{\mathcal{U}}
\newcommand{\cV}{\mathcal{V}}

\DeclareMathOperator{\Hom}{Hom}

\DeclareMathOperator{\Ind}{Ind}
\DeclareMathOperator{\Res}{Res}
\DeclareMathOperator{\res}{res}

\DeclareMathOperator{\Kar}{Kar}

\newcommand{\xra}[1]{\xrightarrow{#1}}

\newcommand{\ket}[1]{|\,#1\,\rangle}

\catcode`\@=11
\long\def\@makecaption#1#2{%
    \vskip 10pt
    \setbox\@tempboxa\hbox{%
\small{#1: }\ignorespaces #2}%
    \ifdim \wd\@tempboxa >\captionwidth {%
        \rightskip=\@captionmargin\leftskip=\@captionmargin
        \unhbox\@tempboxa\par}%
      \else
        \hbox to\hsize{\hfil\box\@tempboxa\hfil}%
    \fi}
\newdimen\@captionmargin\@captionmargin=2\parindent
\newdimen\captionwidth\captionwidth=\hsize
\catcode`\@=12
\title{KLR algebras and the branching rule I: The categorical Gelfand-Tsetlin basis in type $A_n$}
\author{Pedro Vaz}
\address{Institut de Recherche en Math\'ematique et Physique\\
Universit\'e Catholique de Louvain\\ 
Chemin du Cyclotron 2\\ 
1348 Louvain-la-Neuve\\ 
Belgium}
\email{pedro.vaz@uclouvain.be}
\begin{document}
%
%
\newdimen\captionwidth\captionwidth=\hsize
%
%
\begin{abstract}
We define a quotient of the category of finitely generated  
modules over the cyclotomic Khovanov-Lauda-Rouquier algebra for type $A_n$  
and show it has a module category structure over a direct sum of certain 
cyclotomic Khovanov-Lauda-Rouquier algebras of type $A_{n-1}$, 
this way categorifying the branching rules for $\sll\supset \sln$. 
Using this we provide an  
elementary proof of Khovanov-Lauda's cyclotomic conjecture. 
We show that continuing recursively gives the Gelfand-Tsetlin basis for type $A_n$.   
As an application we prove a conjecture of Mackaay, Sto\v si\'c and Vaz 
concerning categorical Weyl modules. 
\end{abstract}
\maketitle

%
%
\pagestyle{myheadings}
\markboth{\em\small Pedro Vaz}{\em\small KLR algebras and the branching rule I: Gelfand-Tsetlin basis }
%
%
%
\section{Introduction}\label{sec:intro}

Let $A$ and $B$ be associative algebras, $M$ a (left) $B$-module and 
$f\colon A\to B$ 
a map of algebras. 
Then $A$ acts on $M$ through $f$ by the formula $a.m=f(a)m$ for $a$ in $A$ 
and $m$ in $M$. 
This procedure turns each (left) module over $B$ into a (left) module over $A$. 
It is well known that this operation defines a functor between the categories of 
modules over the respective algebras. 
Each homomorphism of algebras $f\colon A\to B$ 
gives rise to a functor of restriction 
between their categories of representations
\begin{equation*}
\res_B^A\colon B-\amod\to A-\amod
\end{equation*}
defined by $M\mapsto f M = f B \otimes_B M$ for 
left $A$-modules $N$ and $B$-modules $M$. 
Here $fM$ means the structure of $A$-module on the $M$ as defined above: $a.fm:=f(a)m$.

In general an irreducible object $M$ in $B-\amod$ is not sent to an irreducible  
over $A$ but we can restrict to categories of modules which are totally reducible:  
any module is isomorphic to a direct sum of irreducible modules, determined up to isomorphism, which is unique up to 
permutation of its summands.  
In this case $\res_B^A(M)$ decompose as a direct sum of 
irreducibles over $A$ and the obtention of such a decomposition  
gives the \emph{branching rule of $B$ with respect to $A$}.

\medskip

The study of the branching rules has its roots in group theory and were first obtained in a systematic way in 
the study of the representations of the classical groups. 
They were subsequently extended to categories of representations of other types of algebras like for example 
associative algebras, Lie algebras,  
Hopf algebras and quantum groups.
Besides being a useful tool in the study of the representations of the objects under consideration, 
the branching rules have been extensively studied in theoretical physics where they have found 
important applications in the study of systems through reduction of its group of symmetry to one of its subgroups
(see~\cite{molev,IN}, the review~\cite{KT} and the references therein).

\medskip  

Let us consider the case of Kac-Moody algebras associated to finite quivers. 
For each embedding $\Gamma_2\hookrightarrow\Gamma_1$ of quivers there is an embedding 
of the Kac-Moody algebras $A_{\Gamma_2}\hookrightarrow A_{\Gamma_1}$ 
associated to $\Gamma_1$ and $\Gamma_2$. 
If we restrict to the categories of integrable representations  
then every irreducible integral representation $V(A_{\Gamma_1})$ of $A_{\Gamma_1}$ 
is isomorphic as a representation of $A_{\Gamma_2}$ to a direct sum of irreducibles~\cite{GK}.   
In some cases a general procedure exists to obtain the branching rules for this 
embedding, but in the general case has to work out the result case by case  
(see~\cite{BKW,molev,IN} for a general treatment of the branching rule for classical Lie algebras). 
These results extend to the quantum version of Kac-Moody algebras, which is the case 
we are interested in. 
Further application of the branching rule eventually gets us to a direct sum of irreducibles over the one-dimensional Kac-Moody algebra 
(one-dimensional spaces therefore). 
Including this collection of spaces back into $V(A_{\Gamma_1})$ defines a distinguished basis which is 
an example of a canonical basis and is 
called the Gelfand-Tsetlin basis after~\cite{GT1} 
(see also~\cite{cherednik}). 
Besides its interest to representation theory, 
Gelfand-Tsetlin have been applied to problems in mathematical physics~\cite{KT}.

\medskip

In a remarkable series of papers~\cite{KL1,KL2,Rouq1} M.~Khovanov, A.~Lauda and independently R. Rouquier introduced a family 
of Hecke algebras associated to a quiver (see also~\cite{Rouq2}). 
These \emph{quiver Hecke} algebras, which became known as \emph{KLR algebras}, 
have shown to have a rich representation theory 
(\cite{BK,HMM,Klesh-Ram,Klesh-Ram2})  
but more immediate to us in this paper is the fact that 
the KLR algebra associated to a quiver $\Gamma$ categorifies the lower 
half of the quantum version of the Kac-Moody algebra associated to $\Gamma$, 
which means that the latter is isomorphic to the Grothendieck ring of the former.  
For each dominant integral weight $\lambda$ the KLR algebra $R_\Gamma$ admits a quotient, denoted $R_\Gamma^\lambda$, 
which is called a cyclotomic quotient after~\cite{BK}, 
and whose Grothendieck group is isomorphic to the integral representation $V_\lambda$ of $A(\Gamma)$:  
the category of graded modules over $R_\Gamma^\lambda$, finite in each degree,  
admits a categorical action of $A(\Gamma)$ which descends to the Grothendieck group yielding  
a representation which is isomorphic to $V_\lambda$~\cite{KK,W1}.  

\medskip

In this paper we concentrate on the case where $\Gamma_1$ is  
the Dynkin diagram of type $A_{n}$ and $\Gamma_2$ is the Dynkin diagram obtained from $\Gamma_1$ 
by removing the vertex labeled $n$ 
(and the corresponding edge),  
and investigate the consequences for the corresponding KLR algebras and its cyclotomic quotients. 
The inclusion of quivers $A_{n-1}\hookrightarrow A_n$ determines an inclusion 
$R_{n-1}\hookrightarrow R_{\n}$
between the respective KLR algebras.  
This gives rise to functors of restriction and induction between their categories of representations 
which turn out to descend to the usual inclusion and projection maps between the corresponding 
(one-half) quantum Kac-Moody algebras. 
This approach needs to be modified to work with cyclotomic quotients. In this case 
there is a projection of the cyclotomic KLR algebra $R_{A_n}^\lambda$ to a 
direct sum of cyclotomic KLR algebras $\oplus_{\mu\in\tau(\lambda)} R_{A_{n-1}}^\mu$ with the set $\tau(\lambda)$  
being determined combinatorially from $\lambda$.  
We obtain a functor 
\begin{equation*}
\Pi\colon R_{A_n}^\lambda-\fmod\to\bigoplus_{\mu\in \tau(\lambda)} R_{A_n-1}^\mu-\fmod 
\end{equation*}
between their categories of graded, finite dimensional modules 
which is full, essentially bijective, and commutes with the categorical action of the Kac-Moody algebra 
given by ${A_n}$. 
Continuing recursively we end up in the category of one-dimensional modules over a collection of 
one-dimen\-sio\-nal algebras $R_{A_0}^{\mu'}$ 
which are labeled by certain sequences of partitions 
$(\lambda,\lambda^{(n-1)},\dotsc ,\lambda^{(1)})$,
each $\lambda^{(i)}$ being a partition with exactly $i$ parts.   
There is a categorical action of the Kac-Moody algebra 
of ${A_n}$ on the functors $R_{A_{n+1}}^\lambda-\prmod\to R^{\mu}_{A_0}-\prmod \cong \Bbbk-\prmod$ 
which 
descends to an action on the Grothendieck group, which means that these functors categorify 
the elements of the Gelfand-Tsetlin basis.  
These functors can be interpreted as the preimages under $\Pi$ of the one-dimensional modules over $R_{A_0}^{\mu'}$, giving 
a realization of the Gelfand-Tsetlin basis in terms of some special objects in the category of modules over 
$R_{A_n}^\lambda$. 

\medskip

One consequence of the categorical branching rules is that we can use it to 
provide an easy proof of Khovanov and Lauda's cyclotomic conjecture from~\cite{KL1}. 
As another application of the categorical branching rules we prove a conjecture in~\cite{MSV} 
about categorical Weyl modules for the $q$-Schur algebra. 
Namely we prove that the cyclotomic KLR algebra is isomorphic to a certain endomorphism algebra 
constructed in~\cite{MSV} 
as part of the $q$-Schur categorification 
to give a conjectural categorification of the Weyl module $W_\lambda$.  
As a consequence we obtain that the aforementioned endomorphism algebra indeed categorifies 
$W_\lambda$, this way proving a second conjecture in~\cite{MSV}. 

\medskip

This paper was motivated by an attempt to lift the recursive formulas 
for link polynomials in~\cite{Jaeg} and~\cite{VW} 
to statements between the corresponding link homology theories  
(see~\cite{VW,wagner} for further explanations and~\cite{vaz} for developments).
This is the first output of the program outlined in~\cite{vaz}. 
%
%
%
%
%
%
We have tried to make this paper reasonably self-contained with the exception of 
Section~\ref{sec:appl} where we assume familiarity
with~\cite{MSV}.

%
\section{Quantum $\sll$, the branching rule and the Gelfand-Tsetlin basis}\label{sec:basics}

In this section we review the basics about quantum $\sll$, its irreducible representations,  
the branching rule for $\sll\supset \sln$, and the Gelfand-Tsetlin basis.
We also fix notation and recollect some results that will be used in this paper.

\subsection{Quantum $\sll$ and its irreducible representations}\label{ssec:sln}

We denote the weight lattice $\Lambda^{n+1}$ and the root lattice $X^{n+1}$. 
Let $\alpha_1,\dotsc ,\alpha_n$ be the simple roots and 
$\alpha_1^\vee,\dotsc ,\alpha_n^\vee$ the coroots. 
Any weight $\llambda$ can be written as $\llambda = (\llambda_1,\dotsc , \llambda_n)$
where $\llambda_i= \alpha_i^\vee(\llambda)$. Denote the set of dominant integral weights by 
\begin{equation*}
\Lambda_+^{n+1} = \bigl\{ \llambda\in\Lambda^{n+1}\vert \alpha_i^\vee(\llambda)\in\bZ_{\geq 0}\text{ for all }i=1,\dotsc ,n \bigr\}. 
\end{equation*}
Let also 
\begin{equation*}
a_{ij}=\alpha_i^\vee(\alpha_j) =
\begin{cases}
2  &\text{if }j=i\\
-1 &\text{if }j=i\pm1\\
0 & \text{else }
\end{cases}
\end{equation*}
be the entries of the Cartan matrix of $\sll$. 

\medskip

The {\em quantum special linear algebra} 
$U_q(\sll)$ is 
the associative unital $\bQ(q)$-algebra generated by the Chevalley generators $F_i$, $E_i$ and 
$K_i^{\pm 1}$, for $1,\ldots, n$, 
subject to the relations
\begin{gather*}
K_iK_j=K_jK_i\mspace{30mu} K_iK_i^{-1}=K_i^{-1}K_i=1
\\
K_iF_jK_i^{-1} = q^{-a_{ij}}F_j \mspace{40mu} K_iE_jK_i^{-1} = q^{a_{ij}}E_j
\\
E_iF_j - F_jE_i = \delta_{i,j}\dfrac{K_i - K_i^{-1} }{q-q^{-1}}
\\
F_i^2F_j - (q+q^{-1})F_iF_jF_i + F_jF_i^2 = 0
\qquad\text{if}\quad |i-j|=1
\\
E_i^2E_j - (q+q^{-1})E_iE_jE_i + E_jE_i^2 = 0
\qquad\text{if}\quad |i-j|=1
\\
F_iF_j = F_jF_i , \mspace{10mu} E_iE_j = E_jE_i\qquad\text{if}\quad |i-j|>1.
\end{gather*} 

\medskip

For $\und{i}=(i_1,\dotsc , i_k)$ we define $F_{\und{i}} = F_{i_k}\dotsm F_{i_1}$ 
and $E_{\und{i}} = E_{i_k}\dotsm E_{i_1}$. 
The reason for this convention will be clear later when we introduce the diagrammatics. 

\medskip

\n The \emph{lower half} $U^-(\sll)\subset U_q(\sll)$ quantum algebra is the subalgebra generated by the $F_i$s 
(analogously for the \emph{upper half} $U^+(\sll)$).

\medskip

Recall that a subspace $V_\lmu$  of a 
finite dimensional $U_q(\sll)$-module $V$ is called a \emph{weight space} if
\begin{equation*}
K_i^{\pm 1}v = q^{\pm \lmu_i}v
\end{equation*}
for all $v\in V_\lmu$ and that $V$ is called a \emph{weight module} if
\begin{equation*}
V = \bigoplus\limits_{\lmu\in\Lambda^{n+1}}V_{\lmu} . 
\end{equation*}

A weight module $V$ is called a \emph{highest weight module with highest weight $\llambda$} if 
there exists a nonzero weight vector vector $v_\llambda\in V_\llambda$ such that $E_iv_\llambda=0$ for $i=1,\dotsc ,n$. 
For each $\llambda\in\Lambda_+^{n+1}$ there exists a unique irreducible highest 
weight module with highest weight $\llambda$. 
In the sequel we will drop the $U_q$ and write $\sll$-module instead of  $U_q(\sll)$-module.

\medskip

Let $\phi$ be the anti-involution on $U_q(\sll)$ defined by
\begin{equation*}
\phi(K_i^{\pm 1})=K_i^{\mp 1}
\mspace{20mu}
\phi(F_i)=q^{-1}K_iE_i
\mspace{20mu}
\phi(E_i)=q^{-1}K_i^{-1}F_i .
\end{equation*}
The \emph{$q$-Shapovalov form} $\langle-,-\rangle$ is the unique nondegenerate symmetric bilinear form 
on the highest weight module $V(\llambda)$ satisfying 
\begin{align*}
\langle v_\llambda,  v_\llambda \rangle &= 1
\\
\langle uv,v'\rangle &=  \langle v,\phi(u)v'\rangle \text{\quad for all $u\in U_q(\sll)$ and $v,v'\in V(\lambda)$}
\\
f\langle v,v'\rangle &=  \langle\bar{f}v,v'\rangle = \langle v,fv'\rangle \text{\quad for any $f\in\bQ(q)$ and $v,v'\in V(\lambda)$}.
\end{align*}

\medskip


\subsection{The $q$-Schur algebra}\label{ssec:schur}

In this subsection we give a brief review the $q$-Schur algebra $S_q(n,d)$ following the exposition in~\cite{MSV} 
(see~\cite{MSV} and the references therein for more details). 
The Schur algebra appears naturally in the context of (polynomial) representations of ${U}_q(\mathfrak{gl}_n)$,
which is the starting point of this subsection.
The root and weight lattices are very easy to describe for quantum $\mathfrak{gl}_n$.
Let 
$\epsilon_i=(0,\ldots,1,\ldots,0)\in \bZ^n$, with $1$ being on the $i$th 
coordinate for $i=1,\ldots,n$.
Let also
$\widetilde{\alpha}_i=\epsilon_i-\epsilon_{i+1}\in\bZ^{n}$ 
and $(\epsilon_i,\epsilon_j)=\delta_{i,j}$ be the Euclidean inner product on $\bZ^n$ 
(in this basis the $\mathfrak{sl}_{n}$ roots can be expressed by $\alpha_i = \widetilde{\alpha}_i-\widetilde{\alpha}_{i+1}$).

\medskip

The quantum general linear algebra ${U}_q(\mathfrak{gl}_n)$ is the 
associative unital $\bQ(q)$-algebra generated by $K_i,K_i^{-1}$, for $i=1,\ldots, n$, 
and $F_i$, $E_{i}$, for $i=1,\ldots, n-1$, subject to the relations
\begin{gather*}
K_iK_j=K_jK_i\quad K_iK_i^{-1}=K_i^{-1}K_i=1
\\[1.5ex]
E_iF_{j} - F_{j}E_i = \delta_{i,j}\dfrac{K_iK_{i+1}^{-1}-K_i^{-1}K_{i+1}}{q-q^{-1}}
\\[1.5ex]
K_iF_{j}=q^{-(\epsilon_i,\alpha_j)}F_{j}K_i
\mspace{50mu} 
K_iE_{j}=q^{ (\epsilon_i,\alpha_j)}E_{ j}K_i
\end{gather*}
and
\begin{gather*}
\left.\begin{aligned}
F_{i}^2F_{j}-(q+q^{-1})F_{i}F_{j}F_{i}+F_{j}F_{i}^2 &= 0 
\\
E_{i}^2E_{j}-(q+q^{-1})E_{i}E_{j}E_{i}+E_{j}E_{i}^2 &= 0 
\end{aligned}\;\right\}
\qquad\text{if}\quad |i-j|=1
\\[1.5ex]
\left.\begin{aligned}
F_{i}F_{j}-F_{j}F_{i}=0
\\
E_{i}E_{j}-E_{j}E_{i}=0 
\end{aligned}\;\right\}
\qquad\text{if}\quad |i-j|>1. 
\end{gather*} 

\medskip

%

%
%
%


In the Beilinson-Lusztig-MacPherson~\cite{B-L-M} 
idempotented version of quantum groups, 
the Cartan subalgebras are 
``replaced'' by algebras generated by 
orthogonal idempotents corresponding to the weights. 
To understand their definition, recall that $K_i$ acts as $q^{\lambda_i}$ on the 
$\lambda$-weight space of any weight representation. 
The idempotented version  
of ${U}_q(\mathfrak{gl}_n)$  
can be obtained 
from ${U}_q(\mathfrak{gl}_n)$ by adjoining orthogonal idempotents 
$1_{\lambda}$, for $\lambda\in\bZ^{n}$, and adding the relations
\begin{align*}
1_{\lambda}1_{\nu} &= \delta_{\lambda,\nu}1_{\mu} 
\\
F_{i}1_{\lambda} &= 1_{\lambda - {\widetilde{\alpha}_i}}F_{i}
\\
E_{i}1_{\lambda} &= 1_{\lambda + {\widetilde{\alpha}_i}}E_{i}
\\
K_i1_{\lambda} &= q^{\lambda_i}1_{\lambda}.
\end{align*}

\n The \emph{idempotent quantum $\mathfrak{gl}_n$} is then defined by 
\begin{equation*}
\dot{U}(\mathfrak{gl}_n)\cong \underset{\lambda,\mu\in\bZ^{n}}{\oplus} 
1_{\lambda}{U}_q(\mathfrak{gl}_n)1_{\mu}  . 
\end{equation*}
Note that $\dot{U}(\mathfrak{gl}_n)$  
is not unital anymore because $1=\sum\limits_{\lambda\in\bZ^{n}}1_{\lambda}$ would be 
an infinite sum. 
In this setting the $q$-Schur algebra occurs naturally as a quotient
of idempotent ${U}_q(\mathfrak{gl}_n)$, which happens to be very easy to describe.
Let  
\begin{equation*}
\Lambda(n,d) =
\bigl\{ \lambda\in\bN^n\colon\sum_i\lambda_i = d \bigr\}
\end{equation*}
be a weight (sub)lattice and the highest weights be elements in
\begin{equation*}
\Lambda^+(n,d) =
\{ \lambda\in\Lambda(n,d)\colon\lambda_1 \geq \lambda_2\geq\dotsc \geq \lambda_n \}. 
\end{equation*}

The $q$-Schur algebra $S_q(n,d)$ can be defined as the quotient of 
idempotented quantum $\mathfrak{gl}_n$ 
by the ideal generated by all idempotents 
$1_{\lambda}$ such that $\lambda\not\in\Lambda(n,d)$. 
Thus we have   
a finite presentation of $S_q(n,d)$ as 
the associative unital 
$\bQ(q)$-algebra generated by $1_{\lambda}$, for $\lambda\in\Lambda(n,d)$, 
and $F_{i}$, $E_{i}$, for $i=1,\ldots,n-1$, subject to the relations
\begin{align*}
1_{\lambda}1_{\mu} &= \delta_{\lambda,\mu}1_{\lambda} 
\\[0.5ex]
\sum_{\lambda\in\Lambda(n,d)}1_{\lambda} &= 1
\\[0.5ex]
F_{i}1_{\lambda} &= 1_{\lambda - \widetilde{\alpha}_i}F_{i} 
\\[0.5ex]
E_{i}1_{\lambda} &= 1_{\lambda + \widetilde{\alpha}_i}E_{i} 
\\[0.5ex]
E_iE_{-j}-E_{-j}E_i &= \delta_{ij}\sum\limits_{\lambda\in\Lambda(n,d)}
[\lambda_i - \lambda_{i+1}]1_{\lambda}
\end{align*}
We use the convention that $1_{\mu}X1_{\nu}=0$, if $\mu$ 
or $\nu$ is not contained in $\Lambda(n,d)$. 

\medskip 

The irreducibles $W_{\lambda}$, 
for $\lambda\in\Lambda^+(n,d)$, can be 
constructed as subquotients of $S_q(n,d)$, called \emph{Weyl modules}. Let $<$ denote  
the lexicographic order on $\Lambda(n,d)$.  
For any $\lambda\in\Lambda^+(n,d)$, we have 
$$W_{\lambda}\cong 1_{\lambda}S_q(n,d)/[\mu>\lambda].$$
Here $[\mu>\lambda]$ is the ideal generated by all elements of the form 
$1_{\mu}x1_{\lambda}$, for some $x\in S_q(n,d)$ and $\mu>\lambda$.

\subsection{Branching rules}\label{ssec:brules}

Recall that a partition with $m$ parts is a sequence of nonnegative integers
$(\lambda_1,\dotsc ,\lambda_{m})$ with $\lambda_1\geq\dotsm\geq\lambda_{m}$. 
Partitions are in bijection with Young diagrams. 
We follow the convention where Young diagrams are left justified and lines are enumerated from top to bottom. 
The bijection sends $\lambda$ to the Young diagram with 
$\lambda_i$ boxes in the $i$th line. From now on we denote them by the same symbols.

There is a well known relation between integral dominant weights of $\sll$ and 
partitions with $n+1$ parts. 
For each such partition $\lambda$ there is an integral dominant weight $\llambda\in\Lambda^{n+1}_+$  
defined by 
\begin{equation*}
\llambda_i := \lambda_i - \lambda_{i+1} . 
\end{equation*}
If we want to use partitions to describe 
the finite dimensional irreducibles of $\sll$ 
we can write $V^\sll_\lambda$ to denote the irreducible $\sll$-module $V_\llambda$ without any ambiguity.  
Of course there are several partitions giving the same element of $\Lambda_+^{n+1}$, 
but there is only one if we fix the value of $\lambda_{n+1}=0$. 

\medskip

For a partition $\lambda$ with $n+1$ parts 
denote by $\tau(\lambda)$ the set of all partitions $\mu$ with $n$ parts satisfying
\begin{equation}\label{eq:inbetw}
\lambda_i \leq \mu_i \leq \lambda_{i+1} . 
\end{equation}

We denote by $V^{\mathfrak{sl}_m}_{\lambda}$ the irreducible finite dimensional representation of $\mathfrak{sl}_m$ of 
highest weight $\lambda$.  
For the embedding $\sln\hookrightarrow \sll$ corresponding to 
adding one vertex to the Dynkin diagram of $\sln$ 
the branching rule~\cite{IN} says that   
\begin{equation}
\label{eq:brulsalg}
V^{\sll}_{\lambda} 
\cong 
\bigoplus\limits_{\mu\in\tau(\lambda)} V_{\mu}^{\sln} 
\end{equation}
is an isomorphism of $\sln$-modules. 
This decomposition is multiplicity free that is, each of the $V_{\lmu}^\sln=V_{\mu}^\sln$ occurs 
at most once in the sum~\eqref{eq:brulsalg}.

Define $\tau_{k}(\lambda)$ as the set of all the Young diagrams obtained from $\lambda$ by the removal of $k$ boxes, 
no more than one from each column.  
The inbetweenness condition~\eqref{eq:inbetw} is 
the same as requiring that $\mu$ is in exactly one of the $\tau_k(\lambda)$ for some $k$.  
In other words,  
\begin{equation*}
\tau(\lambda) = \bigoplus\limits_{k\geq 0}\tau_k(\lambda) .
\end{equation*}

Let $\bar\tau_k(\lambda)$ be the set of all the $\lmu\in\Lambda_+^{n}$ for $\mu\in\tau_k(\lambda)$  
and for $1\leq i_1\leq\dotsm \leq i_k\leq n+1$
denote by $\lmu(i_1\dotsm i_k)\in\bar\tau_k(\lambda)$ the weight obtained by removal of 
exactly one box 
from each of the $i_1,\dotsc ,i_k$th lines of $\lambda$, in the order given.   

For practical purposes the set $\bar\tau_k(\lambda)$ is best described using maps  
$\Lambda^{n+1}\to\Lambda^{n+1}$ and $\Lambda^{n+1}\to\Lambda^{n{}}$. 
For $p_\Lambda\colon\Lambda^{n+1}\to\Lambda^{n}$ the projection 
\begin{equation}\label{eq:projW}
p_\Lambda(\lmu_1,\dotsc ,\lmu_n) = (\lmu_1,\dotsc ,\lmu_{n-1})
\end{equation}
and $(i_1,\dotsc ,i_k)$ as above
we define $\xi'_{i_1\dotsm i_k}\colon\Lambda^{n+1}\to\Lambda^{n+1}$  
and
$\xi_{i_1\dotsm i_k}\colon\Lambda^{n+1}\to\Lambda^{n{}}$  
by  
\begin{align*}
\xi'_{i}(\llambda) &= (\llambda_1,\dotsc, \llambda_{i-1}+1, \llambda_{ i }-1,\dotsc , \llambda_{n-1},\llambda_n)
\\
\xi'_{i_1\dotsm i_k}(\llambda) &= \xi'_{i_k}\dotsm\xi'_{i_1}(\llambda) 
\end{align*} 
and $\xi_i(\llambda) = p_\Lambda\xi'_i$.  
We say that $\xi_i$ is \emph{$\lambda$-dominant} if $\xi_i(\llambda)$ is in $\Lambda_+^n$ 
whenever $\llambda$ is in $\Lambda_+^{n+1}$ and that $\xi_{i_1\dotsm i_k}$ is $\lambda$-dominant 
if for each $j\leq k$ 
the map $\xi_{i_1\dotsm i_j}$ is $\lambda$-dominant.

We have  
\begin{equation*}
\lmu(i_1\dotsm i_k)=\xi_{i_k}\dotsm\xi_{i_1}(\llambda) . 
\end{equation*}  
Keeping this notation in mind we denote by $\xi_i(\lambda)$ 
the partition obtained from $\lambda$ by the removal of one box from its $i$th line 
with $\xi_{i_1\dotsm i_k}(\lambda)$ meaning the one obtained by the removal of $k$ boxes, one for each 
$i_r$th line. 
We see that $\lambda$-positivity of $\xi_{i_1\dotsm i_k}$ is equivalent of the requirement 
that no two boxes are removed from the same column of $\lambda$. 
For later use we denote by $\cD_\lambda^k$ the set of all $\lambda$-dominant $\xi_{i_1\dotsm i_k}$s
and define $\cD_\lambda=\medcup_{k\geq 0}\cD_\lambda^k$. 
This way $\tau(\lambda)$ can be also seen as 
the set of all the $\xi_{i_1\dotsm i_j}(\lambda)$ with $\xi_{i_1\dotsm i_j}$ in $\cD_\lambda$.

\subsection{The Gelfand-Tsetlin basis}\label{ssec:GT}

We can 
reapply the branching rule~\eqref{eq:brulsalg} recursively until 
we end up with a direct sum of 1-dimensional spaces 
corresponding to a final decomposition of each irreducible of $\mathfrak{sl}_2$ 
into 1-dimensional $\bQ(q)$-vector spaces.  

\medskip

We say a sequence $(\mu^{(n+1)},\dotsc ,\mu^{(1)})$  of partitions, where $\mu^{(j)}$ has $j$ parts,   
is a \emph{Gelfand-Tsetlin pattern for $\mu^{(n+1)}$} if each consecutive pair $(\mu^{(j)},\mu^{(j-1)})$ 
satisfy the inbetweenness condition~\eqref{eq:inbetw}. 
Denote by $\cS(\lambda)$ the set of all the Gelfand-Tsetlin patterns for $\lambda$.  

\medskip

The Gelfand-Tsetlin patterns for $\lambda$ are the paths followed in the sequence of weight lattices 
\begin{equation*}
\Lambda_+^{n+1} \to\Lambda_+^{n} \to\dotsm\to \Lambda_+^{1}=\bN_0 
\end{equation*}
in going from 
$V^\sll_\lambda$ to each of the 1-dimensional spaces occurring at the end. 
Since the decomposition~\eqref{eq:brulsalg} is multiplicity free there is a 1-1 correspondence 
between $\cS(\lambda)$ and the set of all these 1-dimensional spaces. 
Let $V_{\cS(\lambda)}$ be the $\bQ(q)$-linear spanned by $\cS(\lambda)$. 
We write $\ket{s}$ for a Gelfand-Tsetlin pattern $s$ seen as an element of $V_{\cS(\lambda)}$. 
It turns out that it is isomorphic to  $V_\lambda^{\sll}$ not only as a vector space but as 
$\sll$-modules. 

The $\sll$-action on $V_{\cS(\lambda)}$ can be obtained through a procedure which, in some sense, 
is the reverse of the direct sum decomposition~\eqref{eq:brulsalg} using the branching rule.   
Whilst the generators $\{E_i,F_i\}_{i\in\{1,\dotsc ,n-1\}}$ preserve the weight spaces 
$V_\mu^\sln$
on the right-hand side of~\eqref{eq:brulsalg},   
the generators $E_n$ and $F_n$ move between the different $V_\mu^\sln$: 
let $\phi$ be the isomorphism $V_\lambda^\sll\to\oplus_{\mu\in\tau(\lambda)} V_{\mu}^{\sln}$ 
in~\eqref{eq:brulsalg}. 
Then the $\sln$-action on $\oplus_{\mu\in\tau(\lambda)} V_{\mu}^{\sln}$ extends to an $\sll$-action 
if we define 
\begin{equation*}
E_nv := \phi E_n\phi^{-1} v
\mspace{25mu}\text{and}\mspace{25mu}
F_nv := \phi F_n\phi^{-1} v
\end{equation*}
for $v\in\oplus_{\mu\in\tau(\lambda)} V_{\mu}^{\sln}$. 
This is a consequence of $\phi E_i\phi^{-1} v = E_i v$ and $\phi F_i\phi^{-1} v = F_i v$ 
for all $v\in\oplus_{\mu\in\tau(\lambda)} V_{\mu}^{\sln}$. 
We can continue this procedure until we get the desired 1-dimensional spaces 
and regard the $\sll$-action on them as an action on $V_{\cS(\lambda)}$. 
The basis of $V_{\cS(\lambda)}$ given by the Gelfand-Tsetlin patterns  
is called the Gelfand-Tsetlin basis for $V_\lambda^\sll$.  
This basis was first defined by I.~M.~Gelfand and M.~L.~Tsetlin in~\cite{GT1} for the Lie algebra $gl(n)$. 
The explicit form of action of the generators of the Lie algebra 
$gl(n)$ on the Gelfand-Tsetlin basis can be found for example in~\cite{molev,zelobenko}.

%
\section{KLR algebras and their cyclotomic quotients}\label{sec:KLRalgebras}

In this section we describe the quiver Hecke algebras which were introduced by Khovanov and Lauda in~\cite{KL1} 
and independently by Rouquier in~\cite{Rouq1}. We concentrate on the particular case of type $A_n$.    
The KLR algebra $R_{n+1}$ associated to the quiver $A_n$ 
is the algebra generated by $\Bbbk$-linear combinations of isotopy classes of 
braid-like planar diagrams where each strand is labeled by a simple root of $\sll$. 
Strands can intersect transversely to form crossings and they can also carry dots. 
Multiplication is given by concatenation of diagrams and the collection of such diagrams is subject 
to relations~\eqref{eq:R2}-\eqref{eq:dotslide} below  
(for the sake of simplicity we write $i$ instead $\alpha_i$ when labeling a strand). 
We read diagrams from bottom to top by convention and therefore the diagram 
for the product $a.b$ is the diagram obtained by stacking the diagram for $a$ 
on the top of the one for $b$. 
\begin{equation}\label{eq:R2}
\labellist
\tiny \hair 2pt
\pinlabel $i$ at -5 0
\pinlabel $j$ at 61 -2
\endlabellist
\figins{-23}{0.7}{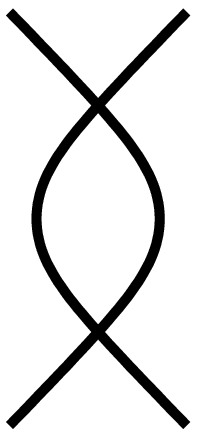}\ \
=\ \
\begin{cases}
\mspace{35mu}0 
&\mspace{20mu}\text{ if }i = j  
\\[1ex]
\mspace{20mu}
\labellist
\tiny \hair 2pt
\pinlabel $i$ at -5 0
\pinlabel $j$ at 61 -2
\endlabellist 
\figins{-23}{0.7}{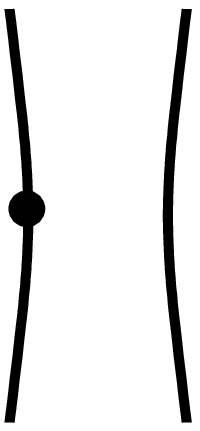}
\quad +\quad 
\labellist
\tiny \hair 2pt
\pinlabel $i$ at -5 0
\pinlabel $j$ at 61 -2
\endlabellist 
\figins{-23}{0.7}{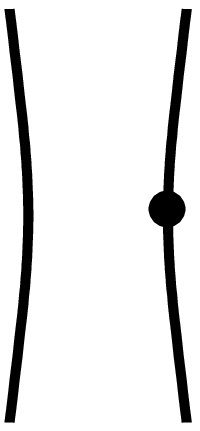}
&\mspace{20mu}\text{ if }j = i \pm 1 
\\[6ex]
\mspace{20mu}
\labellist
\tiny \hair 2pt
\pinlabel $i$ at -5 0
\pinlabel $j$ at 61 -2
\endlabellist 
\figins{-23}{0.7}{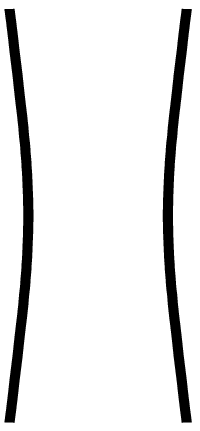}
&\mspace{20mu}\text{ else }
\end{cases}
\end{equation}

\medskip
%
\begin{equation}\label{eq:R3}
\labellist
\tiny \hair 2pt
\pinlabel $i$ at  -5 0
\pinlabel $j$ at  72 -2
\pinlabel $k$ at 132 -2
\endlabellist
\figins{-23}{0.7}{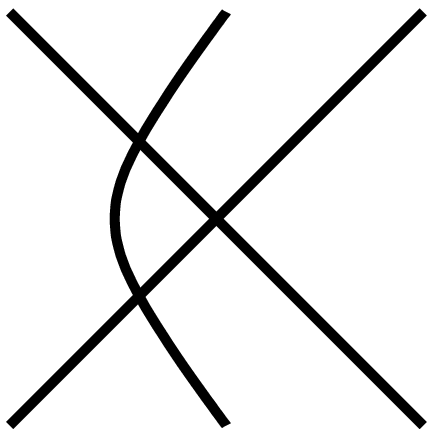}
\ \ - \ \
\labellist
\tiny \hair 2pt
\pinlabel $i$ at  -5 0
\pinlabel $j$ at  52 -2
\pinlabel $k$ at 132 -2
\endlabellist
\figins{-23}{0.7}{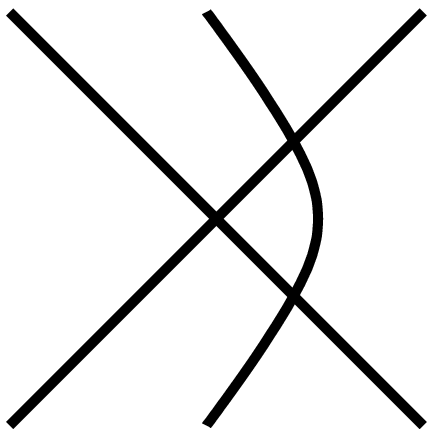}
\ \ = \ \ \ 
\begin{cases}
\mspace{20mu} 
\labellist
\tiny \hair 2pt
\pinlabel $i$ at  -5 0
\pinlabel $j$ at  68 -2
\pinlabel $k$ at 124 -2
\endlabellist
\figins{-23}{0.7}{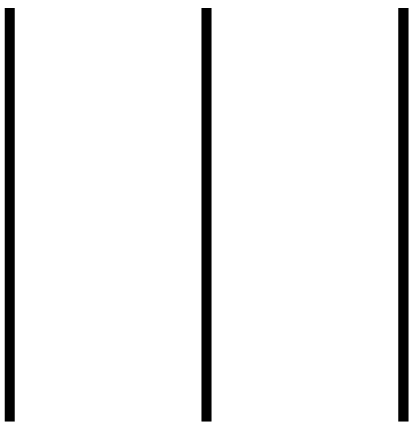} & \text{ if }i=k=j\pm 1
\\[8ex]
\mspace{55mu}
0 & \text{ else}
\end{cases}
\end{equation}

\medskip
%
\begin{align}
\label{eq:dotslide}
\labellist
\tiny \hair 2pt
\pinlabel $i$ at -5   0
\pinlabel $j$ at 97  -2
\endlabellist
\figins{-20}{0.60}{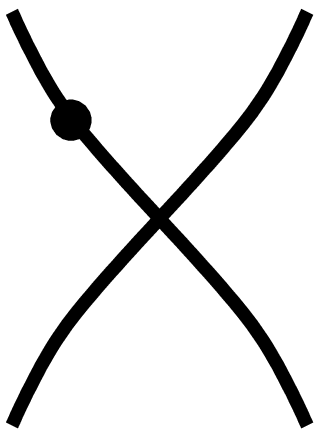}
\ \ - \ \
\labellist
\tiny \hair 2pt
\pinlabel $i$ at -5   0
\pinlabel $j$ at 97  -2
\endlabellist
\figins{-20}{0.60}{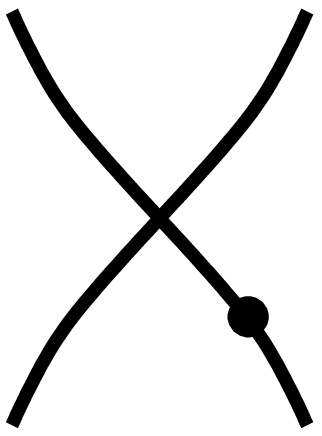}
\ \ = \ \ 
\delta_{ij}\ \ \   
\labellist
\tiny \hair 2pt
\pinlabel $i$ at -5   0
\pinlabel $j$ at 97  -2
\endlabellist
\figins{-20}{0.60}{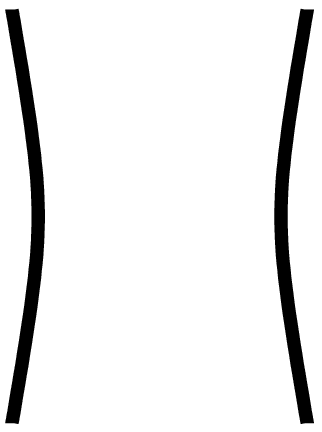}
\quad = \quad  
\labellist
\tiny \hair 2pt
\pinlabel $i$ at -5   0
\pinlabel $j$ at 97  -2
\endlabellist
\figins{-20}{0.60}{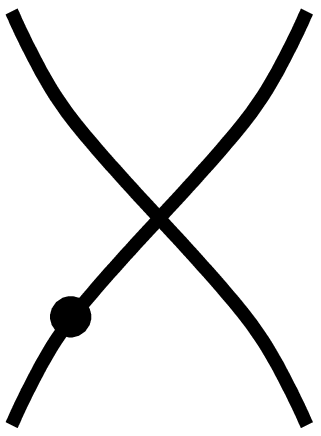}
\ \ - \ \
\labellist
\tiny \hair 2pt
\pinlabel $i$ at -5   0
\pinlabel $j$ at 97  -2
\endlabellist
\figins{-20}{0.60}{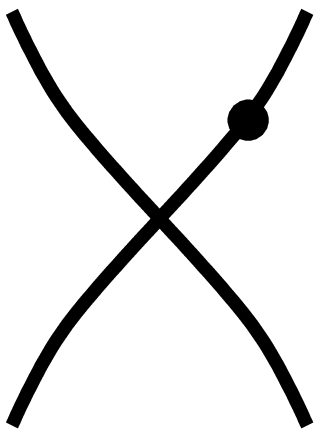}
\vspace*{4ex}
\end{align}

\medskip

Algebra $R_{n+1}$ is graded with the degrees given by 
\begin{align}
\deg\biggl(\  
\labellist
\tiny \hair 2pt
\pinlabel $i$ at -5   0
\pinlabel $j$ at 97  -2
\endlabellist
\figins{-10}{0.40}{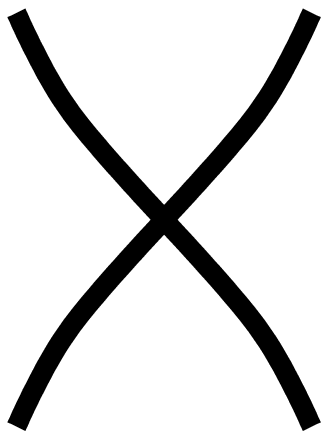}\ 
\biggr) = -a_{ij}
\mspace{60mu}
\deg\biggl(\  
\labellist
\tiny \hair 2pt
\pinlabel $i$ at -5   0
\endlabellist
\figins{-10}{0.40}{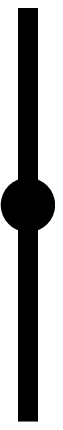}\ 
\biggr) = a_{ii} .
\vspace*{4ex}
\end{align}

\medskip

\n The following useful relation follows from~\eqref{eq:dotslide} and will be used in the sequel.
\begin{align}
\label{eq:dotslides}
\labellist
\tiny \hair 2pt
\pinlabel $i$ at -5 0  \pinlabel $i$ at 97 0 \pinlabel $d$ at 0 90
\endlabellist
\figins{-20}{0.60}{dotslide-lu}
\ \ - \ \
\labellist
\tiny \hair 2pt
\pinlabel $i$ at -5 0 \pinlabel $i$ at 97 0 \pinlabel $d$ at 90 35
\endlabellist
\figins{-20}{0.60}{dotslide-rd}
\ \ = \ \ 
\sum_{\ell_1+\ell_2 = d-1}\ \ \   
\labellist
\tiny \hair 2pt
\pinlabel $i$ at -5 0 \pinlabel $i$ at 97 0 
\pinlabel $\ell_1$ at 32 70 \pinlabel $\ell_2$ at 103 70
\endlabellist
\figins{-20}{0.60}{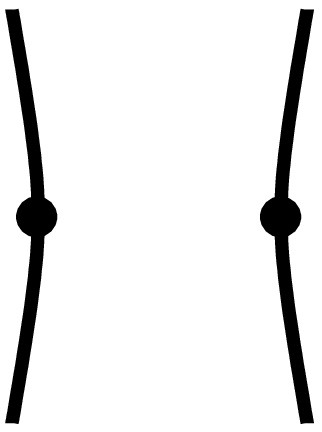}
\quad = \quad  
\labellist
\tiny \hair 2pt
\pinlabel $i$ at -5 0 \pinlabel $i$ at 97 0 \pinlabel $d$ at 0 35
\endlabellist
\figins{-20}{0.60}{dotslide-ld}
\ \ - \ \
\labellist
\tiny \hair 2pt
\pinlabel $i$ at -5 0 \pinlabel $i$ at 97 0 \pinlabel $d$ at 90 90
\endlabellist
\figins{-20}{0.60}{dotslide-ru}
\vspace{4ex}
\end{align}

\medskip

Let $\beta = \sum\limits_{j=1}^{n}\beta_i \alpha_i$ and let $R_{n+1}(\beta)$ be the subalgebra generated by all diagrams 
of $R_{n+1}$ containing exactly $\beta_i$ strands labeled $i$.
We have  
\begin{equation*}
R_{n+1} = \sum\limits_{\beta\in\Lambda_+^{n}}R_{n+1}(\beta) .
\end{equation*}
We also denote by
\begin{equation*}
R_{n+1}(k\alpha_n) = \bigoplus\limits_{\beta'\in\Lambda_+^{n-1}} R_{n+1}(\beta' + k\alpha_n) 
\end{equation*}
the subalgebra of $R_{n+1}$ containing exactly $k$ strands labeled $n$. 
With this notation we have 
\begin{equation}\label{eq:decKLR1}
R_{n+1} = \bigoplus\limits_{k\geq 0}R_{n+1}(k\alpha_n)  .
\end{equation}
For a sequence $\und{i}=(i_1,\dotsc , i_k)$ with $i_j$ corresponding to the simple root $\alpha_{i_j}$ we write 
$1_{\und{i}}$ for the idempotent formed by $k$ vertical strands with labels 
in the order given by $\und{i}$,
\begin{equation*}
1_{\und{i}}\ = 
\labellist
\pinlabel $\dotsc$ at  145 65   
\tiny \hair 2pt
\pinlabel $i_1$   at   2 12
\pinlabel $i_2$   at  46 12
\pinlabel $i_3$   at  88 12
\pinlabel $i_k$   at 198 12
\endlabellist
\mspace{15mu}
\figins{-34}{0.9}{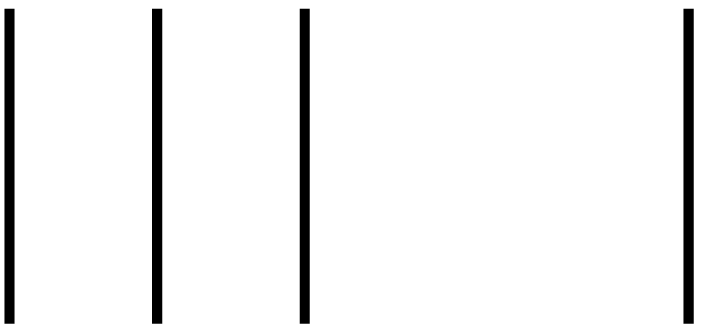} 
\end{equation*}
%
We write $1_{*\ell *}$ for $1_{\und{i}}$ if the sequence of 
labels $\und{i}=\und{j}\und{\ell}\und{j}'$ can be written as a concatenation of sequences 
and we are only interested in the $\und{\ell}$ part.
We also write $x_{r,\und{i}}$ for the diagram consisting of a dot on the $r$th strand of $1_{\und{i}}$,
\begin{equation*}
x_{r,\und{i}}\ = 
\labellist
\pinlabel $\dotsc$ at   45 50   
\pinlabel $\dotsc$ at  135 50   
\tiny \hair 2pt
\pinlabel $i_1$   at   7 -10
\pinlabel $i_r$   at  92 -10
\pinlabel $i_k$   at 178 -10
\endlabellist
\mspace{15mu}
\figins{-19}{0.6}{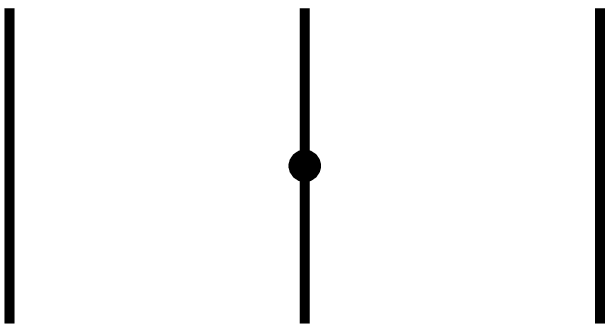}
\end{equation*}

\medskip

For $\beta$ as above we denote by $\seq(\beta)$ the set of all sequences $\und{i}$ of simple roots 
in which $i_j$ appears exactly $\beta_j$ times. The identity of $R_{n+1}(\beta)$ is then given by 
\begin{equation*}
1_{R_{n+1}(\beta)} =
\sum_{\und{i}\:\in\:\seq(\beta)}
\labellist
\pinlabel $\dotsc$ at  145 65   
\tiny \hair 2pt
\pinlabel $i_1$   at   2 12
\pinlabel $i_2$   at  46 12
\pinlabel $i_3$   at  88 12
\pinlabel $i_k$   at 198 12
\endlabellist
\mspace{15mu}
\figins{-34}{0.9}{idan} 
\end{equation*}

\medskip

We have
\begin{equation*}
R_{n+1}(\beta) = \bigoplus\limits_{\und{i},\und{j}\:\in\:\seq(\beta)}1_{\und{i}}\ R_{n+1}1_{\und{j}}. 
\end{equation*}
If $e\in R_{n+1}$ is an idempotent then there is a (right) projective module $_{e}P=eR_{n+1}$. 
For $e=1_{\und{i}}$ this is the projective spanned by all diagrams whose labels  
end up in the sequence $\und{i}$.  We can define the left projective $P_e$ in a similar way. 

Denote by $R_{n+1}-\amod$ and $R_{n+1}-\prmod$ the categories of 
graded finitely generated right $R_{n+1}$-modules and of 
graded finitely generated projective right $R_{n+1}$-modules respectively. 
For idempotents $e$, $e'$ we have 
\begin{equation*}
\Hom_{R_{n+1}-\amod}\bigl({}_{e}P, ~{}_{e'\!}P \bigr) = e'R_{n+1}e  . 
\end{equation*}

\medskip

For a graded algebra $A$ we denote by $K'_0(A)$ the Grothendieck group of finitely generated graded projective 
$A$-modules and write $K_0(A)$ for $\bQ(q)\otimes_{\bZ[q,q^{-1}]} K'_0(A)$.  

\smallskip

There is a pair of functors on $R_{n+1}-\amod$ which descend to the Grothendieck group giving it 
the structure of a \emph{twisted bialgebra} (see~\cite{KL1} for the details). 

\begin{thm}[Khovanov-Lauda~\cite{KL1}, Rouquier~\cite{Rouq1}] 
The Grothendieck group $K_0(R_{n+1})$ is isomorphic to the lower half 
$U^-(\sll)$ through the map that takes $[{}_{\und{i}}P]$ to $F_{\und{i}}$. 
\end{thm}
\n This is an isomorphism of twisted bialgebras  
but we do not pursue this direction in this paper.

\medskip

\subsection{Categorical inclusion and projection for KLR algebras}

Let $\Gamma_n$ and $\Gamma_{n-1}$ the Dynkin diagrams associated to $\sll$ and $\sln$ 
respectively and consider the inclusion $\Gamma_{n-1}\hookrightarrow\Gamma_n$ 
that adds a vertex at the end of $\Gamma_{n-1}$ and the corresponding edge:  
\begin{equation*}
\labellist
\pinlabel $\dotsc$  at 188 60.5    
\tiny \hair 2pt 
\pinlabel $1$   at   7 36  
\pinlabel $2$   at  78 36 
\pinlabel $n-2$ at 287 35 
\pinlabel $n-1$ at 364 35 
\endlabellist
\figins{-19}{0.65}{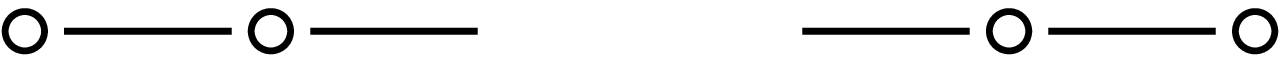}
\mspace{35mu}\hookrightarrow\mspace{35mu}  
\labellist
\pinlabel $\dotsc$  at 188 60.5    
\tiny \hair 2pt 
\pinlabel $1$   at   7 36  
\pinlabel $2$   at  78 36 
\pinlabel $n-2$ at 287 35 
\pinlabel $n-1$ at 364 35 
\pinlabel $n$   at 435 35 
\endlabellist
\figins{-19}{0.65}{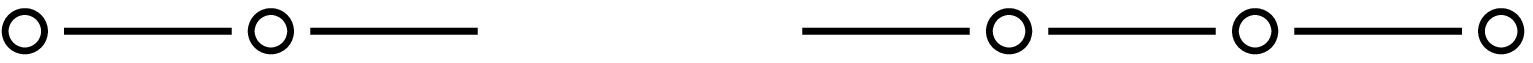} 
\end{equation*}
This induces an inclusion of KLR algebras 
\begin{equation*}
\imath\colon R_n\hookrightarrow R_{n+1}
\mspace{50mu}
x\mapsto x
\end{equation*}
which coincides with the obvious map coming from the decomposition
\begin{equation}\label{eq:decKLR2}
R_{n+1} = \bigoplus\limits_{k\geq 0}R_{n+1}(k\alpha_n) 
\cong 
R_n + \bigoplus\limits_{k\geq 1}R_{n+1}(k\alpha_n). 
\end{equation} 
The functors of inclusion and restriction induced by $\imath$ 
\begin{align*}
\Ind_\imath \colon R_n-\amod \to R_{n+1}-\amod
\mspace{60mu}
\Res_\imath \colon R_{n+1}-\amod \to R_{n}-\amod
\end{align*}
are biadjoint, take projectives to projectives 
and descend to the natural inclusion and projection maps between the Grothendieck groups. 

To see this we notice that the dual construction takes the 
projection $\rho\colon R_{n+1}\to R_n$ in the decomposition~\eqref{eq:decKLR2} to form the 
functor of restriction of scalars and its left and right adjoints, 
the functors of extension of scalars and coextension of scalars by $\rho$ respectively. 
Recall that $\rho$ endows $R_{n}$ with a structure of $(R_{n+1},R_{n})$-bimodule, where 
the structure of left $R_{n+1}$-module is given by $r.b=\rho(r)b$ for $b\in R_n$, $r\in R_{n+1}$.  
The same procedure can be used to give $R_n$ a structure of $(R_{n},R_{n+1})$-bimodule.
We use the notation $_{n+1}(R_n)_n$ and $_{n}(R_n)_{n+1}$ for $R_n$ 
seen as a $(R_{n+1},R_{n})$-bimodule and  $(R_{n},R_{n+1})$-bimodule respectively.
Then we have the functors 
\begin{align*}
\Res_{\rho} &\colon\mspace{12mu} R_n-\amod\ \to R_{n+1}-\amod 
\\[1ex]
\Ext_{\rho}, \CExt_{\rho} &\colon R_{n+1}-\amod \to R_{n}-\amod 
\end{align*}
with 
\begin{equation*}
\Ext_{\rho}(M)= M \bigotimes\limits_{R_{n+1}}{}_{n+1}(R_n)_n  ,
\mspace{50mu}
\CExt_{\rho}(M)= \Hom_{R_{n+1}-\amod}\bigl( {}_n(R_{n})_{n+1}, M\bigr) 
\end{equation*}
for a right $R_{n+1}$-module $M$. 
We have that the functors $\Ext_{\rho}$ and $\CExt_{\rho}$ coincide and we also have isomorphisms of functors 
$\Ext_\rho\cong \Res_\imath$ and  $\Res_\rho \cong\Ind_\imath$.

\subsection{Factoring idempotents}\label{ssec:factidemp}%

In this section we give some properties of $R_{n+1}-\prmod$ that will be used in the sequel. 
For $\nu = \sum\limits_{j=1}^{n-1}\nu_i \alpha_i \in\Lambda_+^{n}$ 
and for an ordered sequence $i_1,\dotsc ,i_k$
we define 
\begin{equation}
\label{eq:nui}
\nu(i)=\sum\limits_{j<i}\nu_j\alpha_j + \sum\limits_{j\geq i}(\nu_j-1)\alpha_j , 
\end{equation}
and $\nu(i_1\dotsm i_k)$ 
as the result of iteration of~\eqref{eq:nui} from $i=i_k$ to $i=i_1$. 

For each $i\in\{1,\dotsm ,n\}$ 
let $p_i$ be the idempotent $1_{i,i+1,\dotsc, n}\in R_{n+1}$  
and for $1 _{\und{j}}\in R_{n+1}(\nu(i))$
let  
$e'(p_i,\und{j})\in R_{n+1}(\nu +\alpha_n)$ be the idempotent obtained by horizontal 
composition of the 
diagram for $p_i$  
at the left of the one for $1 _{\und{j}}$,  
\begin{equation*}
e'(p_i,\und{j}) \ =\  \ 
\labellist
\pinlabel $\dotsc$ at   90 61   
\pinlabel $\dotsc$ at  232 61   
\tiny \hair 2pt
\pinlabel $i$    at   2  10
\pinlabel $i+1$  at  46   9 
\pinlabel $n$    at 130   9
\pinlabel $j_{1}$ at 176 10
\pinlabel $j_{m}$ at 288 10
\endlabellist
\figins{-28}{0.8}{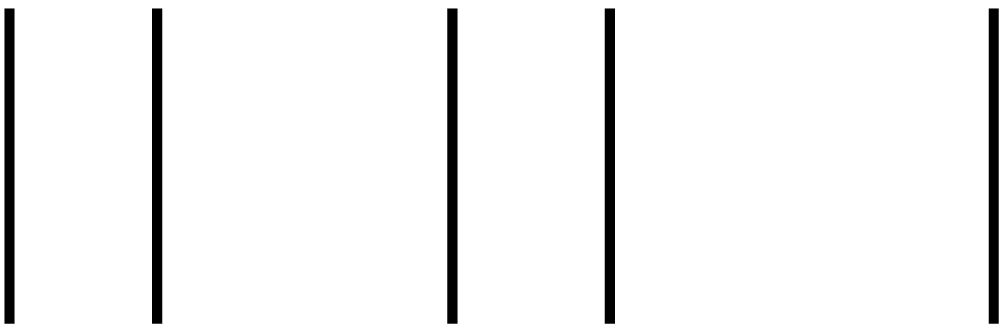} 
\end{equation*}
This generalizes easily to $R_{n+1}(k\alpha_n)$.
In this case we denote by $p_{i_1\dotsm i_k}\in R_{n+1}(k\alpha_n)$ the idempotent 
$1_{i_1,i_1+1,\dotsc ,n, i_2\dotsc ,n,\dotsc ,i_k,\dotsc ,n}$. 
The idempotent 
 $\tilde{e}'(p_{i_1\dotsm i_k},\und{j})\in R_{n+1}(\nu+k\alpha_n)$ for
$1_{\und{j}}\in R_{n+1}(\nu(i_1\dotsm i_k) )$ 
is defined 
as the horizontal concatenation placing 
the diagram of 
$p_{i_1\dotsm i_k}$ 
at the left of the one for $1_\und{j}$,   
\begin{equation*}
\tilde{e}'(p_{i_1\dotsm i_k},\und{j}) \ =\ \  
\labellist 
\pinlabel $\dotsc$ at   33 61
\pinlabel $\dotsc$ at  126 61
\pinlabel $\dotsc$ at  204 61   
\pinlabel $\dotsc$ at  282 61 
\pinlabel $\dotsc$ at  388 61   
\tiny \hair 2pt
\pinlabel $i_1$   at   4  10
\pinlabel $n$     at  59   9
\pinlabel $i_2$   at  98  10
\pinlabel $n$     at 152   9
\pinlabel $i_{k}$ at 253 10
\pinlabel $n$     at 310 10
\pinlabel $j_{1}$ at 346 10
\pinlabel $j_{\ell}$ at 429 10
\endlabellist
\figins{-28}{0.8}{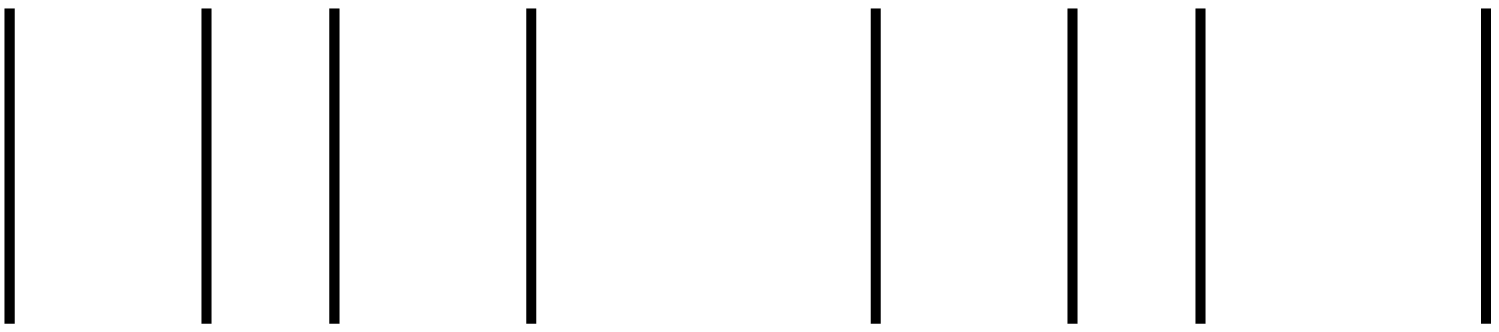} 
\end{equation*} 
In the case $1\leq i_1\leq \dotsm \leq i_k\leq n+1$   
we write $e'(p_{i_1\dotsm i_k},\und{j})$ instead of 
$\tilde{e}'(p_{i_1\dotsm i_k},\und{j})$. 

\smallskip

We next introduce the notion of factoring a diagram through a family of idempotents $e_i$. 
\begin{defn}
We say that $Y\in R_{n+1}(\nu+\alpha_n)$ factors through the family  
of idempotents $\{e_j\}_{j\in J}$ 
if it can be written as a sum $\sum_{j\in J}c_jX_j$ with all $c_j\in\Bbbk$ nonzero 
and where each $X_j$ is in $R_{n+1}(\nu+\alpha_n) e_j R_{n+1}(\nu+\alpha_n)$    
with $j$ minimal. 
\end{defn}

\begin{prop}\label{prop:onestrandn}
Every $1_{\und{j}}$ in $R_{n+1}(\nu+\alpha_n)$ factors through 
the family $\{e'(p_j,\und{j})\}_{j\in J}$ for some indexing set $J$. 
\end{prop}

\begin{proof}
The idempotent $1_{\und{j}}$ consists of $\vert\und{j}\vert$ parallel vertical strands labeled 
$j_1,\dotsc ,j_{\vert\und{j}\vert}$ in that order from left to right. 
We give an algorithm to obtain the factorization as claimed.\\ 
Step 1: Take the single strand labeled $n$ and start pushing it to the left by application of 
the move in~\eqref{eq:R2} 
until we find a strand labeled $n-1$. 
\begin{equation*}
1_{\und{j}}\ = 
\mspace{40mu}
\labellist
\pinlabel $\dotsc$ at  -30 48   
\pinlabel $\dotsc$ at  132 48   
\pinlabel $\dotsc$ at  262 48   
\tiny \hair 2pt
\pinlabel $j_{r-2}$ at   0 -10
\pinlabel $n-1$    at  44 -11
\pinlabel $j_r$    at  88 -10
\pinlabel $j_{r'}$  at 176 -10
\pinlabel $n$      at 218 -10
\endlabellist
\figins{-19}{0.6}{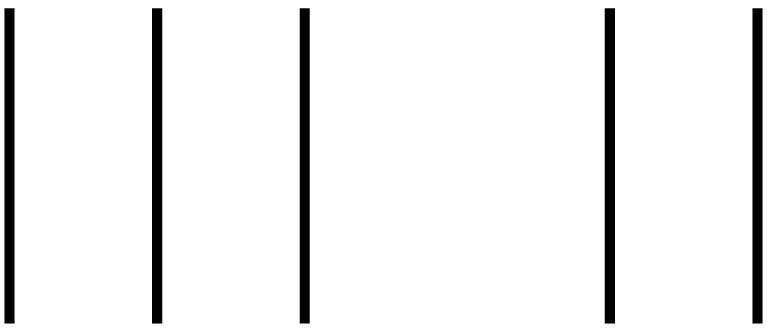} 
\mspace{50mu}\longmapsto\mspace{50mu}
\labellist
\pinlabel $\dotsc$ at  -30 48   
\pinlabel $\dotsc$ at  132 48   
\pinlabel $\dotsc$ at  262 48   
\tiny \hair 2pt
\pinlabel $j_{r-2}$ at   0 -10
\pinlabel $n-1$    at  44 -11
\pinlabel $j_r$    at  88 -10
\pinlabel $j_{r'}$  at 176 -10
\pinlabel $n$      at 218 -10
\endlabellist
\figins{-19}{0.6}{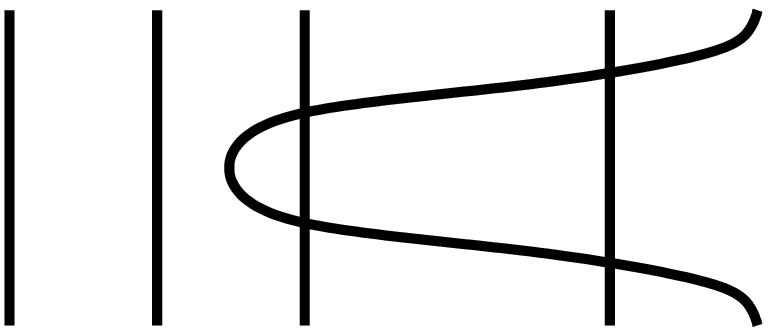}\mspace{30mu}\vspace*{2ex}
\end{equation*}
Step 2: Pass to the strand labeled $n-1$. There are three cases to consider: 
we can have $(i)$ $j_{r-2}= n-1$, $(ii)$ $j_{r-2}=n-2$ or $(iii)$ $j_{r-2}\neq n-2,n-1$.\\
$(i)$ If $j_{r-2}= n-1$ we use the identity 
\begin{equation}\label{eq:invR3}
\labellist
\tiny \hair 2pt
\pinlabel $n-1$ at   0 -10
\pinlabel $n-1$ at  60 -10
\pinlabel $n$   at 118 -10
\endlabellist
\figins{-23}{0.7}{id3}\ \ \  
=\ \ \ 
\labellist
\tiny \hair 2pt
\pinlabel $n-1$ at   0 -10
\pinlabel $n-1$ at  60 -10
\pinlabel $n$   at 118 -10
\endlabellist
\figins{-23}{0.7}{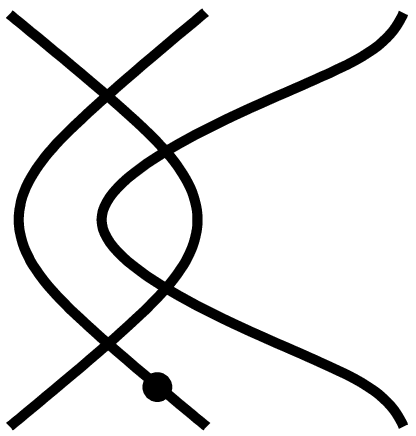}\ \ \
-\ \ \ 
 \labellist
\tiny \hair 2pt
\pinlabel $n-1$ at   0 -10
\pinlabel $n-1$ at  60 -10
\pinlabel $n$   at 118 -10
\endlabellist
\figins{-23}{0.7}{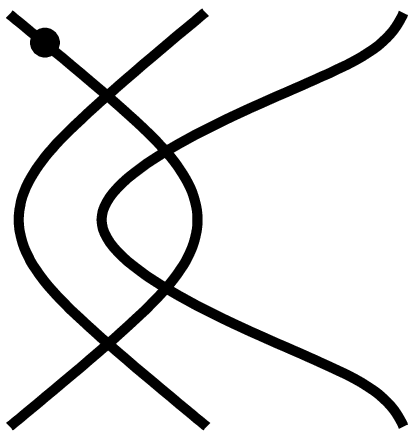}
\vspace*{2ex}
\end{equation}
which follows easily from~\eqref{eq:R2} and~\eqref{eq:dotslide}.  
We see that $1_{\dotsm n-1,n-1,n\dotsm}$ factors through $1_{\dotsm n-1,n,n-1\dotsm}$. 
This reduces the number of strands on the left of the strand labeled $n$. 
We then apply Step 1 to the block formed by strand labeled $n$ and the one labeled $n-1$ immediately on its left. \\  
$(ii)$ If $j_{r-2}=n-2$ we apply Step 1 to the block formed by the strands labeled $n-2$, $n-1$, $n$. 
We can proceed until we find a strand labeled $n-3$, $n-2$ or $n-1$.  
If we find a strand labeled $n-3$ we repeat Step 1 to the block formed by the strands labeled 
$n-3$, $n-2$, $n-1$ and $n$. 
Case we find a strand labeled $n-2$ we are in the situation of $(i)$ with $n$ replaced by $n-1$. 
Case we find a strand labeled $n-1$ we use~\eqref{eq:R3} to obtain that $1_{\dotsm n-1,n-2,n-1,n \dotsm}$ 
factors through $1_{\dotsm n-1,n-1,n-2,n \dotsm}$  and through $1_{\dotsm n-2,n-1,n-1,n \dotsm}$.  
In the first case we can apply~\eqref{eq:R2} to obtain that $1_{\dotsm n-1,n-1,n-2,n \dotsm}$ factors through 
 $1_{\dotsm n-1,n-1,n\dotsm}$ which is the case $(i)$. In the second case we apply $(i)$ to the strands labeled 
$n-1$, $n-1$, $n$. 
Either way we reduce the number of strands on the left of the one labeled $n$.\\ 
$(iii)$ If $j_{r-2}\neq n-2,n-1$ we apply Step 1 to the block formed by the strands  
labeled $n-1$, $n$ until we find a strand labeled $n-2$ or $n-1$. We then proceed like in $(ii)$.   
We then proceed recursively: each time we get a diagram factoring through $1_{\dotsm s,s+1,\dotsm, n-1,n \dotsm}$ 
we apply Step 1 to the entire block formed by the strands labeled $s$ to $n$ until
we find a strand labeled $j$ for $s-1\leq j\leq n-1$. We then apply the move in~\eqref{eq:R2} to pull this strand 
to the left of the one labeled $j-1$ obtaining the configuration below.
\begin{equation*}
y =
\mspace{30mu}
\labellist
\pinlabel $\dotsc$ at  -20 48   
\pinlabel $\dotsc$ at   92 48   
\pinlabel $\dotsc$ at  258 48
\pinlabel $\dotsc$ at  328 48
\tiny \hair 2pt
\pinlabel $j$   at   0 -10
\pinlabel $s$   at  44 -11
\pinlabel $j-2$ at  128 -10
\pinlabel $j-1$ at 176 -10
\pinlabel $j$   at 218 -10
\pinlabel $n$   at 292 -10
\endlabellist
\figins{-19}{0.6}{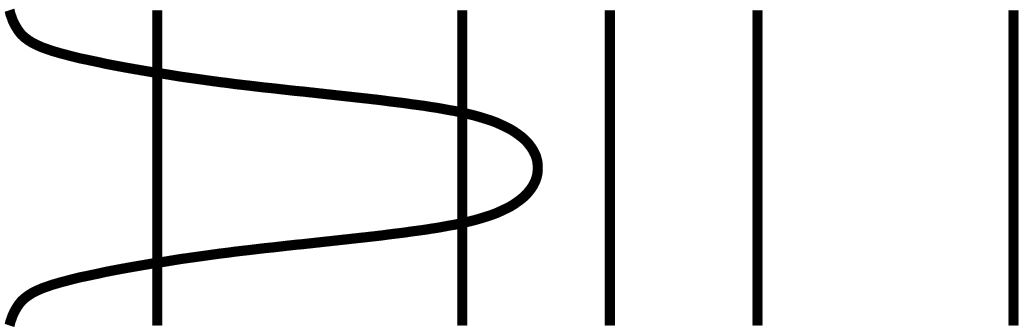}\mspace{30mu}\vspace*{2ex} 
\end{equation*}
Using~\eqref{eq:R3} in the region factoring through $1_{\dotsm j,j-1,j\dotsm}$ one obtains that $y$ 
factors through the idempotent $1_{\dotsm j,j,j-1,j+1\dotsm}$ and through $1_{\dotsm j-1,j,j,j+1\dotsm}$. 
Using~\eqref{eq:R2} on the first term we can slide the strand labeled $j-1$ to the right of the one 
labeled $n$ and then apply the procedure described above in $(i)$ to $1_{\dotsm j,j,j+1\dotsm}$.  
To the second term we apply the procedure of $(i)$ to $1_{\dotsm j,j,j+1\dotsm}$. 
Again, in either case we obtain a linear combination of terms each having less strands on the left 
of the one labeled $n$. 
The procedure ends when we obtain a linear combination of diagrams, each one factoring through  
and idempotent of the form $e'(p_i,\und{i})$ as claimed.  
\end{proof}

We now take care of the case of $k>1$.

\begin{lem}\label{lem:ordseq}
Suppose $k=2$. Then $\tilde{e}'(p_{i_1i_2},\und{i})$ factors through a family 
$\{e'(j_ij_2,\und{j})\}_{j\in J}$ for some indexing set $J$. 
\end{lem}
\begin{proof} 
Suppose we have 
\begin{equation*}
\tilde{e}'(p_{ji},\und{i})\ =\ \  
\labellist
\pinlabel $\dotsc$ at   94 48   
\pinlabel $\dotsc$ at  242 48  
\pinlabel $\dotsc$ at  432 48 
\pinlabel $\dotsc$ at  512 48 
\tiny \hair 2pt
\pinlabel $j$      at   0 -10
\pinlabel $j+1$    at  44 -11
\pinlabel $n$      at 132 -10
\pinlabel $i$      at 176 -10
\pinlabel $j-1$    at 300 -10
\pinlabel $j$      at 346 -10
\pinlabel $j+1$    at 390 -10
\pinlabel $n$      at 474 -10
\endlabellist
\figins{-19}{0.6}{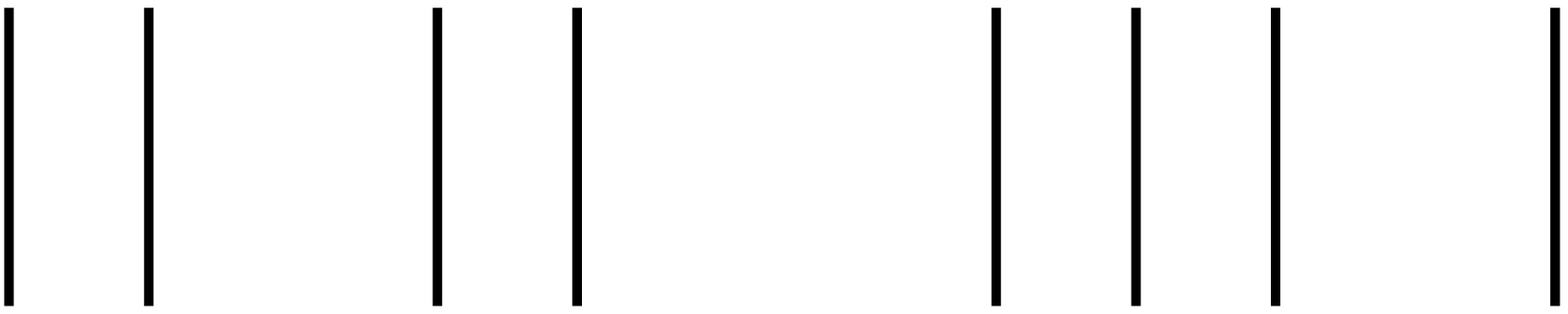} 
\vspace*{2ex} 
\end{equation*}
Using~\eqref{eq:R2} we slide the first strand labeled $n$ from the left to the right until it encounters 
a strand labeled $n-1$ to obtain a factorization through $1_{j\dotsm n-1,i\dotsm n-2, n,n-1,n,*}$: 
\begin{equation}\label{eq:bigslider}
\labellist
\pinlabel $\dotsc$ at   94 48   
\pinlabel $\dotsc$ at  282 48  
\pinlabel $\dotsc$ at  478 48 
\pinlabel $\dotsc$ at  636 48 
\tiny \hair 2pt
\pinlabel $j$      at   0 -10 
\pinlabel $j+1$    at  44 -11 
\pinlabel $n-1$    at 132 -10 
\pinlabel $n$      at 176 -11 
\pinlabel $i$      at 218 -10 
\pinlabel $j-1$    at 342 -10 
\pinlabel $j$      at 386 -10 
\pinlabel $j+1$    at 430 -10 
\pinlabel $n-2$    at 512 -10 
\pinlabel $n-1$    at 560 -10 
\pinlabel $n$      at 602 -10
\endlabellist
\figins{-19}{0.6}{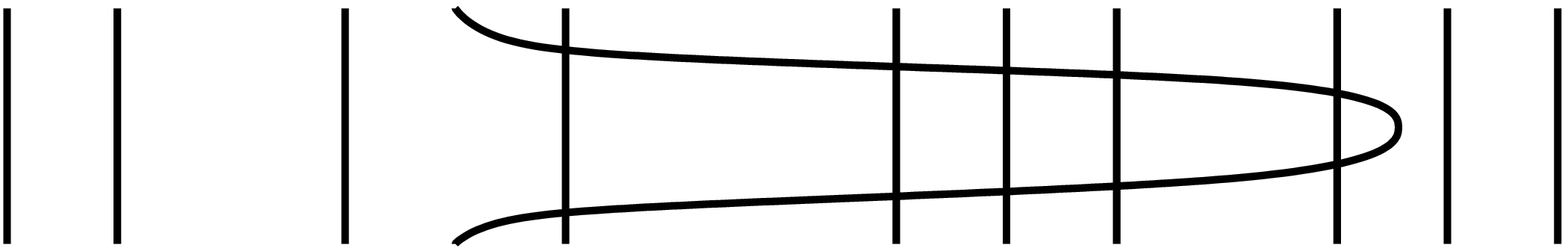} 
\vspace*{2ex} 
\end{equation}
We now use~\eqref{eq:R3} in the $(n)-(n-1)-(n)$ part on the right of the diagram to obtain 
\begin{equation*}
\labellist
\pinlabel $\dotsc$ at   33 48   
\pinlabel $\dotsc$ at  190 48  
\pinlabel $\dotsc$ at  398 48 
\tiny \hair 2pt
\pinlabel $j$      at   0 -10 
\pinlabel $n-1$    at  58 -10 
\pinlabel $n$      at 100 -11 
\pinlabel $i$      at 144 -10 
\pinlabel $j$      at 228 -10 
\pinlabel $n-2$    at 280 -10 
\pinlabel $n-1$    at 330 -10 
\pinlabel $n$      at 378 -10
\endlabellist
\figins{-19}{0.6}{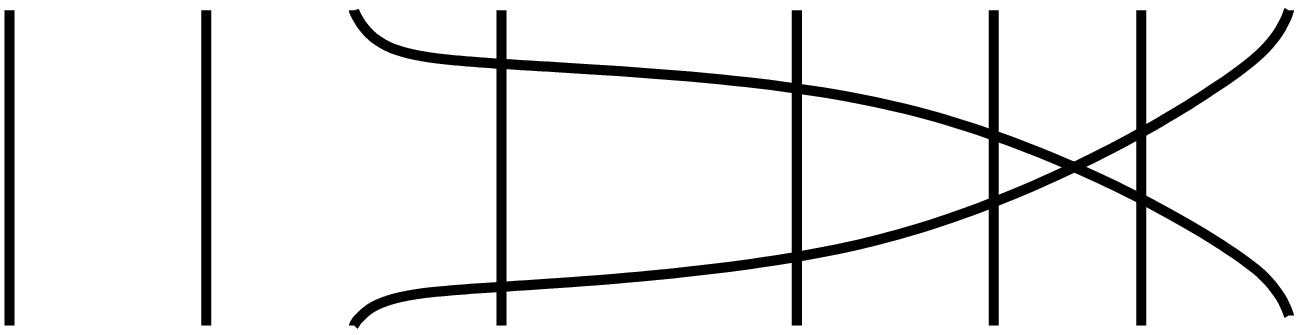}\mspace{40mu} -\mspace{15mu}    
\labellist
\pinlabel $\dotsc$ at   33 48   
\pinlabel $\dotsc$ at  190 48  
\pinlabel $\dotsc$ at  398 48 
\tiny \hair 2pt
\pinlabel $j$      at   0 -10 
\pinlabel $n-1$    at  58 -10 
\pinlabel $n$      at 100 -11 
\pinlabel $i$      at 144 -10 
\pinlabel $j$      at 228 -10 
\pinlabel $n-2$    at 280 -10 
\pinlabel $n-1$    at 330 -10 
\pinlabel $n$      at 378 -10
\endlabellist
\figins{-19}{0.6}{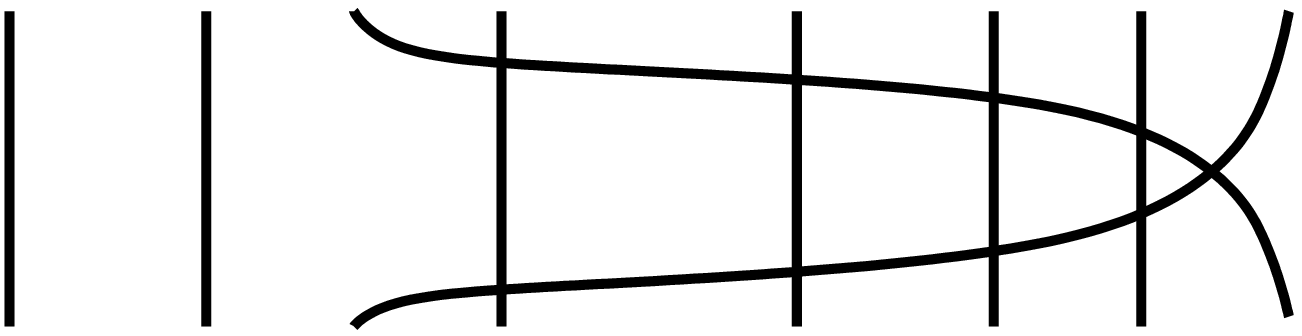}
\vspace*{2ex} 
\end{equation*}

The first term  
factors through the idempotent $1_{j\dotsm n-1,i\dotsm j,n-2,n,n,n-1*}$ and the second through  
$1_{j\dotsm n-1,i\dotsm j,n-2,n-1,n,n*}$. 
For the first term it is easy to see that we can do the same as in~\eqref{eq:bigslider} to slide the entire block 
formed by the strands labeled $1\dotsm j,n-2$ to the right of the two strands labeled $n$ to obtain a 
factorization through $1_{j\dotsm n-1,n,n, i\dotsm j,n-2,n-1*}$ which is of the form $e'(p_{jn},\und{j})$ 
for some $\und{j}$, as wanted. 
For the second term we slide 
the first strand labeled $n-1$ from the left to the right until it finds the strand labeled $n-2$:
\begin{equation*}
\labellist
\pinlabel $\dotsc$ at   94 48   
\pinlabel $\dotsc$ at  282 48  
\pinlabel $\dotsc$ at  478 48 
\pinlabel $\dotsc$ at  726 48 
\tiny \hair 2pt
\pinlabel $j$      at   0 -10 
\pinlabel $j+1$    at  44 -11 
\pinlabel $n-2$    at 129 -10 
\pinlabel $n-1$    at 177 -10 
\pinlabel $i$      at 218 -10 
\pinlabel $j-1$    at 342 -10 
\pinlabel $j$      at 386 -10 
\pinlabel $j+1$    at 430 -10 
\pinlabel $n-3$    at 510 -10 
\pinlabel $n-2$    at 559 -10 
\pinlabel $n-1$    at 604 -10
\pinlabel $n$      at 644 -10
\pinlabel $n$      at 686 -10
\endlabellist
\figins{-19}{0.6}{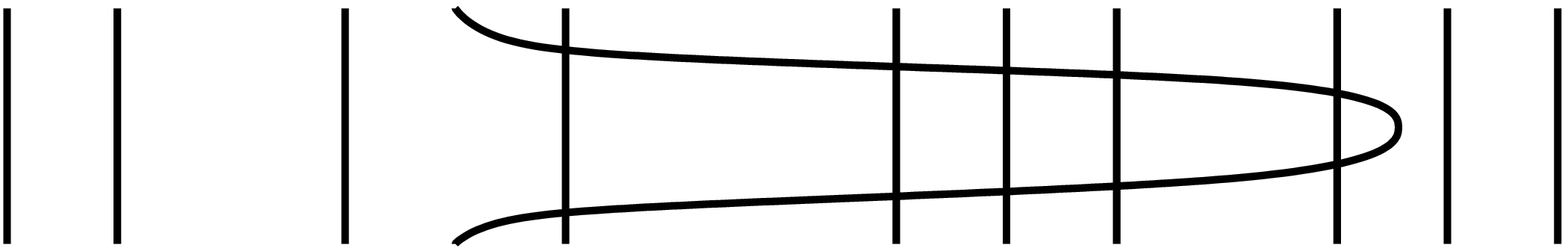} 
\vspace*{2ex} 
\end{equation*}
Applying~\eqref{eq:R3} to the part containing $(n-1)-(n-2)-(n-1)$ we see that it factors 
through 
\begin{equation}\label{eq:factoriz}
1_{j\dotsm n-2,i\dotsm n-3,n-1,n-1,n-2,n,n*}
\mspace{20mu}\text{and}\mspace{20mu}  
1_{j\dotsm n-2,i\dotsm n-3,n-2,n-1,n-1,n,n*}.
\end{equation} 
Using~\eqref{eq:R2} we can put the first term in the form 
\begin{equation*}
\labellist
\pinlabel $\dotsc$ at   47 65   
\pinlabel $\dotsc$ at  496 65 
\pinlabel $\dotsc$ at  650 65 
\tiny \hair 2pt
\pinlabel $j$      at   0 -10 
\pinlabel $n-2$    at  87 -11 
\pinlabel $i$      at 130 -10 
\pinlabel $n-3$    at 212 -10 
\pinlabel $n-1$    at 270 -10 
\pinlabel $n-1$    at 316 -10 
\pinlabel $n-2$    at 362 -10 
\pinlabel $n$      at 404 -10
\pinlabel $n$      at 446 -10
\endlabellist
\figins{-24}{0.75}{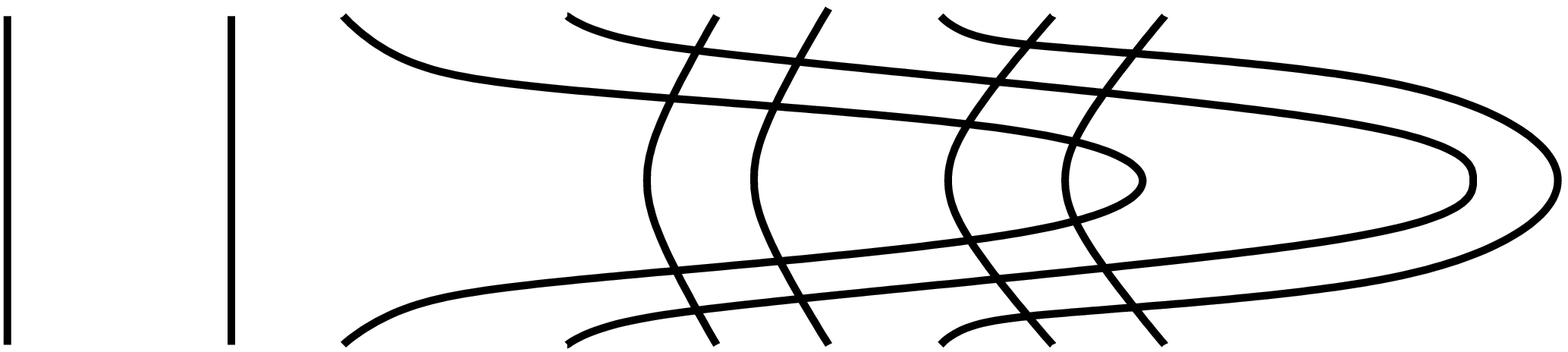} 
\vspace*{2ex} 
\end{equation*}
and we see that it factors through $1_{j\dotsm n-2,n-1,n-1,n,n* }$ 
which in turn factors  (twice) through 
an idempotent of the form $e'(p_{j,n-1},\und{j}')$ as wanted
(this uses the identity~\eqref{eq:invR3}). 
For the second term in~\eqref{eq:factoriz} 
we start by sliding 
the first strand labeled $n-2$ from the left  
and repeat the procedure.  
After having slid all the strands at the left of the first one labeled $i$ 
to the right we end up with a factorization through 
the family $\{e'(p_{j,n-r},\und{r})\}_{r\in\{0,\dotsc, n-j\} }$
and through a term of the form
\begin{equation}\label{factorizz}
1_{i\dotsm j-1, j,j,j+1,j+1, \dotsc, n-1,n-1,n,n*}.
\end{equation}
Notice that after the first strand labeled $j-1$ all the strands are labeled 
in pairs each two strands with the same label appearing consecutively. 
Applying the identity~\eqref{eq:invR3}, 
to the part labeled $(j)-(j)-(j+1)$ we get a factorization 
through $1_{i\dotsm j-1, j,j+1,j,j+1, \dotsc, n-1,n-1,n,n*}$.
Doing the same to the part containing $(j+1)-(j+2)-(j+2)$ we
get a factorization through 
$1_{i\dotsm j-1, j,j+1,j,j+2, j+1,j+2, \dotsc, n-1,n-1,n,n*}$ 
which factors through 
$$1_{i\dotsm j-1, j,j+1,j+2,j, j+1,j+2, \dotsc, n-1,n-1,n,n*}$$ 
(this uses~\eqref{eq:R2} between the two consecutive strands labeled $j$ and $j+2$).  
It is clear that we will end up  with a factorization through $e'(p_{ij},\und{\ell})$ 
for some $\und{\ell}$. 
This way we see that $\tilde{e}'(p_{i_1i_2},\und{i})$ factors through the family 
$\{e'(p_{j,n-s},\und{s})\}_{s\in\{0,\dotsm, n-j \} }\cup \{e'(p_{ij},\und{\ell})\}$ 
for some idempotents $\{1_{\und{r}}\}_{r=0,\dotsm, n-j}$ and $1_{\und{\ell}}$.  
\end{proof}

\begin{prop}\label{prop:genfactoriz}
Every $1_{\und{j}}$ in  $R_{n+1}(\nu+k\alpha_n)$ factors through a family 
$\{e'(p_{i_1\dotsm i_k}, \und{i})\}_{i_1\dotsm i_k,\und{i}\in I}$ 
for some indexing set $I$.   
\end{prop}

\begin{proof}
We use a combination of Proposition~\ref{prop:onestrandn} and Lemma~\ref{lem:ordseq}. 
By application of the method described in the proof of Proposition~\ref{prop:onestrandn} 
to the leftmost strand of $1_{\und{j}}$ labeled $n$ we factorize it through some family  
the idempotents $\{e'(p_{a_1},\und{a})\}$
where each $1_{\und{a}}$ is in $R_{n}(\nu(a_1)+(k-1)\alpha_n)$. 
Repeating the procedure for the newly created $1_{\und{r}}$ we get that 
$1_{\und{j}}$ factors through  $\{\tilde{e}'(p_{b_1b_2}, \und{b})\}$
for $1_{\und{b}}$ in $R_{n}(\nu(b_1b_2)+(k-2)\alpha_n)$. 
But each of the $\tilde{e}'(p_{b_1b_2}, \und{b})$ factors through
a family $\{e'(p_{c_1c_2},\und{c})\}$ from Lemma~\ref{lem:ordseq} and therefore 
$1_{\und{j}}$ factors through a family $\{e'(p_{d_1d_2},\und{d})\}$. 
Then pass to the third leftmost strand labeled $n$.   
From Proposition~\ref{prop:onestrandn} each of the $e'(p_{d_1d_2},\und{d})$ 
factors through a family $\{\tilde{e}'(p_{d_1d_2f_3},\und{f})\}$ with 
$d_1\leq d_2$ as before and $1_{\und{f}}\in R(\nu(d_1d_2f_3)+(k-3)\alpha_n)$. 
For each of these terms we apply Lemma~\ref{lem:ordseq} again: if $d_2\leq f_3$ do nothing  
otherwise factorize $\tilde{e}'(p_{d_1d_2f_3},\und{f})$ 
through $\{\tilde{e}'(p_{d_1g_2g_3},\und{g}) \}$ where $g_2\leq g_3$. 
If $d_1\leq g_2$ do nothing otherwise
apply Lemma~\ref{lem:ordseq} to the $p_{d_1}p_{g_2}$ part to obtain a factorization 
through $\{e'(p_{h_1h_2},\und{h})\}$ with $h_1\leq h_2$. This procedure slides strands to the space between 
the second and third strands labeled $n$ and therefore we need to 
apply Proposition~\ref{prop:onestrandn} to the third strand labeled $k$ again and repeat the  
procedure described above. Notice that applying Proposition~\ref{prop:onestrandn} and  
Lemma~\ref{lem:ordseq} amounts of ``sliding'' locally some strands \emph{to the right} 
of a strand labeled $n$.  
This means that each time we apply each of these procedures to a strand labeled $n$ 
we decrease the number of strands on its left. This means that the process terminates 
with the factorization of $1_{\und{j}}$ through a family $\{e'(p_{m_1m_2m_3},\und{m})\}$ 
with each $1_{\und{m}}\in R_{n+1}(\nu(m_1m_2m_3)-(k-3)\alpha_n)$. 
We now repeat the whole procedure to the fourth strand labeled $n$. 
Finiteness of the number of strands on its left implies that 
application of Proposition~\ref{prop:onestrandn} and Lemma~\ref{lem:ordseq} as above allows factoring 
$1_{\und{j}}$ through $\{ e'(p_{n_1n_2n_3n_4},\und{n})\}$. 
Proceeding recursively with the remaining strands labeled $n$ we get that $1_{\und{j}}$ factors through 
a family $\{e'(p_{\ell_1\ell_2\ell_3\dotsm\ell_k},\und{\ell})\}_{\ell\in L}$ with each $1_{\ell}$ 
in $R_{n+2}(\nu(\ell_1\dotsm\ell_k))$, as claimed. 
\end{proof}


\subsection{Cyclotomic KLR-algebras}

Fix a partition 
$\lambda$ with $n+1$ parts  
for once and for all and 
let $I^{\lambda}$ the two-sided ideal genera\-ted by $x^{\llambda_{i_1}}_{1,\und{i}}$ 
for all sequences $\und{i}$. 

\begin{defn}\label{def:cycKLR}. 
The cyclotomic KLR algebra $R_{n+1}^\lambda$ is the quotient of 
$R_{n+1}$ by the two-sided ideal $I^{\lambda}$. 
\end{defn}
Differently from the standard convention in the literature we label cyclotomic KLR algebras 
by partitions instead of integral dominant weights. This convention will be useful later. 
In terms of diagrams we are taking the quotient of  $R_{n+1}$ by the two-sided ideal generated by all the diagrams 
of the form 
\begin{equation*}
\labellist
\pinlabel $\lambda$ at  -64 74  
\pinlabel $\dotsc$  at  134 63   
\tiny \hair 2pt
\pinlabel $\llambda_{j_1}$  at  -12 66
\pinlabel $j_1$  at   2 10  \pinlabel $j_2$   at  46 10
\pinlabel $j_3$  at  90 10  \pinlabel $j_k$   at 180 10
\endlabellist
\figins{-19}{0.70}{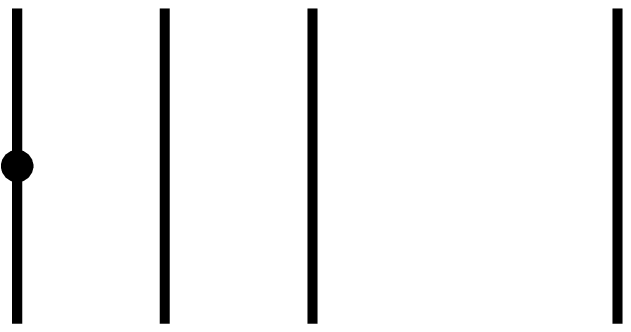} 
\end{equation*}
where the leftmost strand has $\llambda_{j_1}$ dots on it. 
We always label the leftmost region of a diagram with a partition $\lambda$ 
to indicate it is in $R_{n+1}^\lambda$. 
The following was proved in~\cite{W1}.
\begin{lem}\label{lem:frobKLR}
The cyclotomic KLR algebra $R_{n+1}^\llambda$ is Frobenius. 
\end{lem}

\smallskip

Projective modules over $R_{n+1}^\lambda$ are defined the same way as for $R_{n+1}$, 
we write ${}_eP^{\lambda}$ for $eR_{n+1}^{\lambda}$.  
Denote by $R^{\lambda}_{n+1}-\amod$ and by $R^{\lambda}_{n+1}-\prmod$ the categories of finitely generated graded 
$R^{\lambda}_{n+1}$-modules and finitely generated graded projective $R^{\lambda}_{n+1}$-modules
respectively.  
The module category structure in the $R^{\lambda}_{n+1}-\amod$ was studied in~\cite{KK, W1}. 
Here we describe the necessary to proceed through this paper.  
Let $\imath_i\colon R_{n+1}^{\lambda}(\nu)\to R_{n+1}^{\lambda}(\nu+\alpha_i)$ be the map obtained by adding 
a vertical strand labeled $i$ on the right of a diagram from $R^{\lambda}_{n+1}$. 
The categorical $\sll$-action on $R_{n+1}^{\lambda}$ is obtained by 
the pair of biadjoint exact functors defined by 
\begin{equation}\label{eq:Uaction}
\begin{split}
F_i^{\lambda} =  & 
\colon\mspace{22mu} R^{\lambda}_{n+1}(\nu)-\amod\mspace{16mu} 
\xra{\mspace{63mu}\Ind_i\mspace{63mu}}  R^{\lambda}_{n+1}(\nu+\alpha_i)-\amod 
\\
E_i^{\lambda} = 
&\colon R^{\lambda}_{n+1}(\nu+\alpha_i)-\amod \xra{\ \Res_i\{\alpha_i^\vee(\lambda-\nu)-1\}\ }  
\mspace{20mu}R^{\lambda}_{n+1} (\nu)-\amod  
\end{split}
\end{equation}

The \emph{Khovanov-Lauda cyclotomic conjecture}~\cite{KL1} was proved by Brundan and Kleshchev~\cite{BK}
based on Ariki's categorification theorem~\cite{ariki}:  
\begin{thm}\label{thm:BK}
There is an isomorphism of $\sll$-representations
\begin{equation*}
K_0\bigl( R_{n+1}^{\lambda}\bigr) \cong V_{\lambda}^{\sll}  .
\end{equation*}
\end{thm}

In Subsection~\ref{ssec:easyBK} we give an alternative, elementary proof  using the
categorical branching rule. 
Theorem~\ref{thm:BK} was subsequently extended to affine type $A$ by Brundan and Kleshchev~\cite{BK2} and to all types by 
Kang and Kashiwara~\cite{KK} and independently by Webster~\cite{W1}. 
Webster also proved that
\begin{equation*}
\gdim\Hom_{R^\lambda_{n+1}-\amod}(P,P') = \langle [P],[P'] \rangle ,
\end{equation*}
where $\langle~,~\rangle$ is the $q$-Shapovalov form.  

\smallskip


All the results in Subsection~\ref{ssec:factidemp} descend to the cyclotomic setting.
In particular they allow a presentation of the category 
$R_{n+1}^{\lambda}(\nu+k\alpha_n)-\prmod$ in terms 
of the collection of projectives
 $\{{}_{e'(p_{i_1\dotsm i_k},\und{i})}P^{\lambda}\}$ that turns out to be useful later.

\section{Categorical branching rules}\label{sec:catbranching}

\subsection{Categorical branching rules}\label{ssec:catbranching}

We have a direct sum decomposition of algebras
\begin{equation*}
R^{\lambda}_{n+1} \cong \bigoplus\limits_{k\geq 0} R^{\lambda}_{n+1}(k\alpha_n) , 
\end{equation*}
where $R^{\lambda}_{n+1}(k\alpha_n)\subseteq R^{\lambda}_{n+1}$ is the subalgebra generated by the diagrams in 
$R^{\lambda}_{n+1}$
containing exactly $k$ strands labeled $n$.
We also have
\begin{equation*}
R^{\lambda}_{n+1}-\amod \cong \bigoplus\limits_{k\geq 0} R^{\lambda}_{n+1}(k\alpha_n)-\amod .
\end{equation*}
Clearly $R^{\lambda}_{n+1}(0)\cong R^{p_\Lambda(\lambda)}_{n}$, 
where $p_\Lambda(\lambda)\colon\Lambda_+^{n+1}\to\Lambda_+^n$ is the projection given 
in~\eqref{eq:projW}.
We want to identify each block of $R^{\lambda}_{n+1}-\amod$ with the categorification of 
 the $\sln$-representations in~\eqref{eq:brulsalg} 
in the sense that $R^{\lambda}_{n+1}(k\alpha_n)-\amod$ will give the 
$\sln$ irreducibles obtained by removing exactly $k$ boxes from the 
Young diagram for $V^\lambda_{\sll}$. 

\medskip

We start by defining a special class of idempotents in $R_{n+1}^\lambda$.

\begin{defn}
The idempotent 
$e'(p_{i_1\dotsm i_k},\und{j})\in R_{n+1}^\lambda(\nu+k\alpha_n)$  
is said to be a \emph{special idempotent},
denoted $e(p_{i_1\dotsm i_k},\und{j})$, 
if $\xi_{i_1\dotsm i_k}$ is $\lambda$-dominant.  
\end{defn}
\n The property of $\lambda$-dominancy of $\xi_{i_1\dotsm i_k}$ implies that $\nu(i_1\dotsm i_k)$ is in $\Lambda_+^{n}$.  

\medskip

In the following we give the maps between some cyclotomic KLR algebras that are necessary to obtain the categorical branching rule.  
\begin{lem}\label{lem:cycinc}
For each $k\geq 0$ there is a surjection of algebras 
\begin{equation*}
R^\lambda_{n+1}(k\alpha_n)\xra{\ \pi_k\ } \bigoplus\limits_{\xi_{i_1 \dotsm i_k}\in\cD_\lambda^k}R_n^{\xi_{i_1\dotsm i_k}(\lambda)} .
\end{equation*} 
\end{lem}

\begin{proof}
We first prove that for each $\xi_{i_1 \dotsm i_k}\colon\Lambda_+^{n+1}\to\Lambda_+^n$ 
we have a surjection of algebras 
$$R_{n+1}^\lambda(k\alpha_n)\xra{\ \pi_{i_1\dotsm i_k}\ }R_n^{\xi_{i_1\dotsm i_k}(\lambda)}.$$
To this end it is enough to show that for each $\xi_{i_1\dotsm i_k}$ as above, the subalgebra 
\begin{equation*} 
A_{i_1\dotsm i_k}:= 
\bigoplus\limits_{\und{r},\und{s}\in\seq\{\alpha_1,\dotsc ,\alpha_n\}}
e(p_{i_1\dotsm i_k},\und{r})
\bigl(R^\lambda_{n+1}(k\alpha_{n-1})\bigr)
e(p_{i_1\dotsm i_k},\und{s})
\end{equation*} 
projects onto  $R_n^{\xi_{i_1\dotsm i_k}(\lambda)}$.
Let 
$\widetilde{A}_{i_1\dotsm i_k}\subset {A}_{i_1\dotsm i_k}$
be the subalgebra generated by 
all diagrams having a representative given by diagrams consisting of $k$ blocks of vertical strands 
on the left, where the $n-i_r+1$ strands which belong to the $r$th block from the left are labeled 
$i_r, i_r+1,\dotsc , n$ in that order, as below
\begin{equation*}
\labellist
\small
\pinlabel $R^\lambda_{n+1}(0.\alpha_n)$ at 429 90  
\pinlabel $\dotsc$ at  33 90   \pinlabel $\dotsc$ at 430 145 
\pinlabel $\dotsc$ at 118 90   \pinlabel $\dotsc$ at 430  35 
\pinlabel $\dotsc$ at 300 90 
\pinlabel $\dotsc$ at 185 90 
\pinlabel $\dotsc$ at 230 90  
\pinlabel $\lambda$ at -30 95 
\tiny \hair 2pt
\pinlabel $i_1$ at   5  5  \pinlabel $n$ at  59  5 
\pinlabel $i_2$ at  92  5  \pinlabel $n$ at 145  5
\pinlabel $i_k$ at 272  5  \pinlabel $n$ at 326  5
\endlabellist 
\figins{-20.5}{1.0}{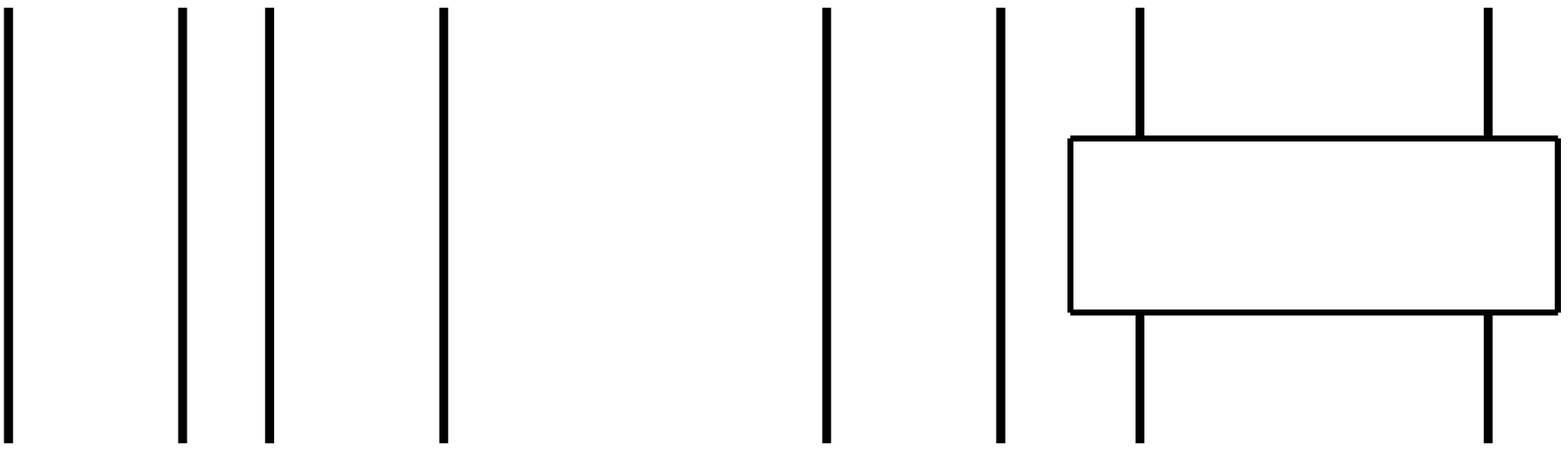}
\end{equation*}
and let $\widetilde{A}_{i_1\dotsm i_k}^\bot$ be its complement vector space. 
Let also $\widetilde{A}_{i_1\dotsm i_k}^\lambda$ be the quotient of $\widetilde{A}_{i_1\dotsm i_k}$ by the two sided ideal 
generated by all diagrams of the form 
\begin{equation*}
\labellist
\pinlabel $\dotsc$ at 445 90 
\small
\pinlabel $\dotsc$ at  33 90   
\pinlabel $\dotsc$ at 118 90   
\pinlabel $\dotsc$ at 300 90 
\pinlabel $\dotsc$ at 185 90 
\pinlabel $\dotsc$ at 230 90 
\pinlabel $\lambda$ at -30 95
\tiny \hair 2pt
\pinlabel $i_1$ at   5  5  \pinlabel $n$ at  59  5 
\pinlabel $i_2$ at  92  5  \pinlabel $n$ at 145  5
\pinlabel $i_k$ at 272  5  \pinlabel $n$ at 326  5
\pinlabel $n$ at 326  5  
\pinlabel $j_1$ at 386  5 
\pinlabel $j_{\ell}$ at 500  5 
\pinlabel $\zeta_{j_1}$ at 364  96  
\endlabellist 
\figins{-20.5}{1.0}{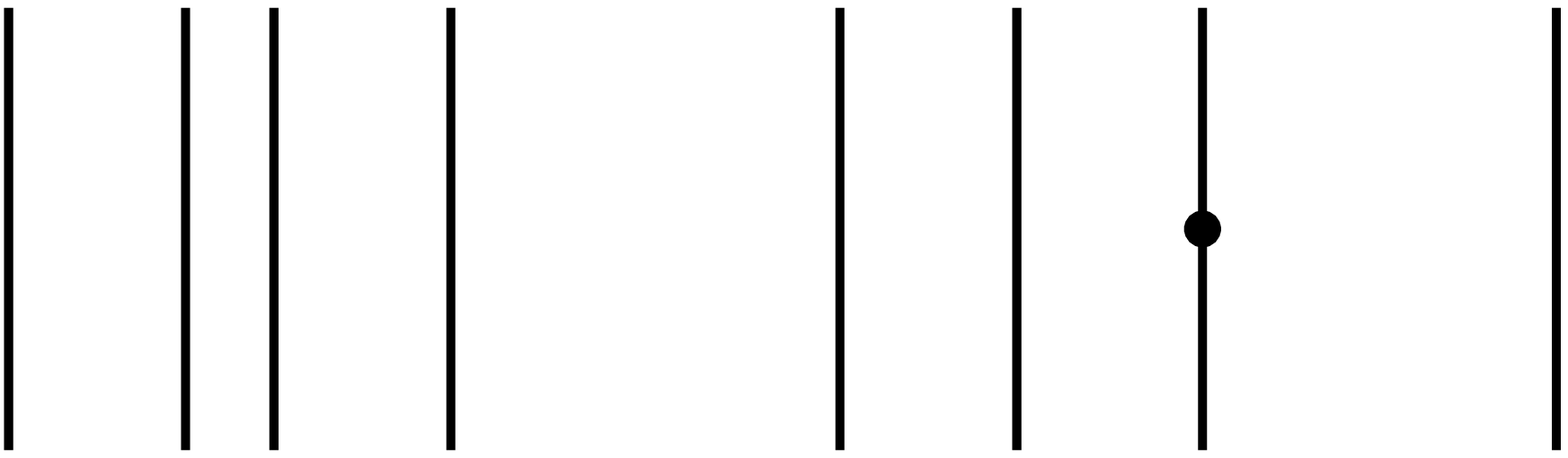}
\end{equation*}
where $\zeta=\bar{\xi_{i_1\dotsm i_k}(\lambda)}$. 
The algebras $\widetilde{A}_{i_1\dotsm i_k}^\lambda$ and $R_n^{\xi_{i_1\dotsm i_k}(\lambda)}$ 
are isomorphic.

We start with the case $k=1$ and use it to prove the general case 
by recursion.   
To this end we compute
\begin{equation*}
X_i(j,r_j) = \mspace{14mu}
\labellist
\pinlabel $\dotsc$ at   95 63    
\pinlabel $\dotsc$ at  262 63 
\pinlabel $\lambda$ at -15 95 
\tiny \hair 2pt
\pinlabel $r_j$    at  -9   65
\pinlabel $i$      at  28    1 
\pinlabel $i+1$    at  59    0 
\pinlabel $n$      at 128    0
\pinlabel $j$      at 158   -1
\endlabellist
\figins{-30}{0.85}{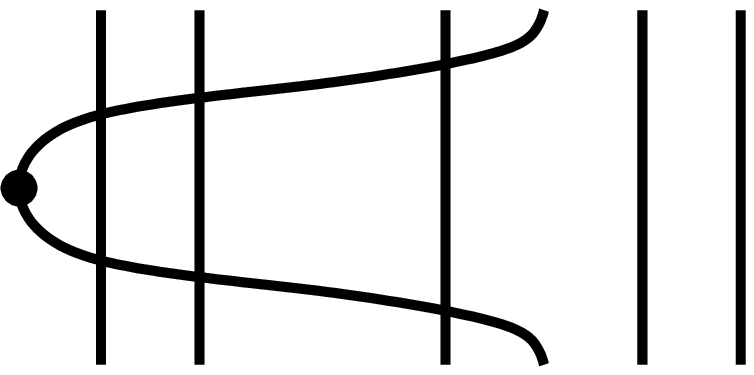}
\mspace{24mu}
\end{equation*}
We have several cases to consider, 
$j<i-1$, 
$j=i-1$, 
$j=i$ and 
$i<j<n$.  

For $j<i-1$ we have 
\begin{equation*}
X_i(j,r_j) =\mspace{22mu}   
\labellist
\pinlabel $\dotsc$ at   85 63    
\pinlabel $\dotsc$ at  252 63    
\pinlabel $\lambda$ at -10 75
\tiny \hair 2pt
\pinlabel $r_j$    at 130   65
\pinlabel $i$      at  15    1 
\pinlabel $i+1$    at  46    0 
\pinlabel $n$      at 115    0
\pinlabel $j$      at 145   -1
\endlabellist
\figins{-30}{0.85}{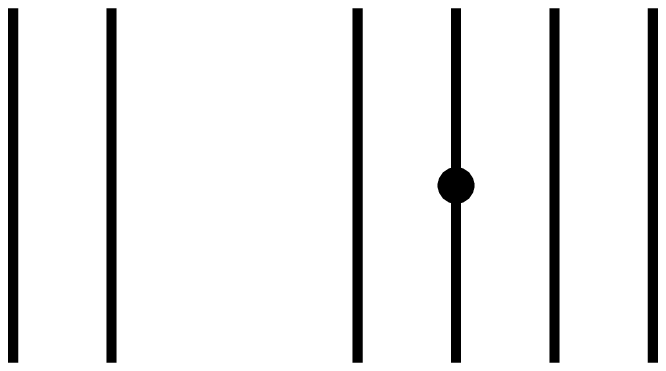}
\mspace{24mu}
\end{equation*}
which follow easily from relations~\eqref{eq:R2} and~\eqref{eq:dotslide}. 

For $j=i-1$ we have 
\begin{align*}
X_i(i-1,r_{i-1}) &= \mspace{10mu}
\labellist
\pinlabel $\dotsc$ at   95 63    
\pinlabel $\dotsc$ at  242 63 
\pinlabel $\lambda$ at -10 85
\tiny \hair 2pt
\pinlabel $r_{i-1}$ at 150  36 
\pinlabel $i$      at  24   1 
\pinlabel $i+1$    at  55   0 
\pinlabel $n$      at 124   0
\pinlabel $i-1$    at 156   0
\endlabellist
\figins{-30}{0.85}{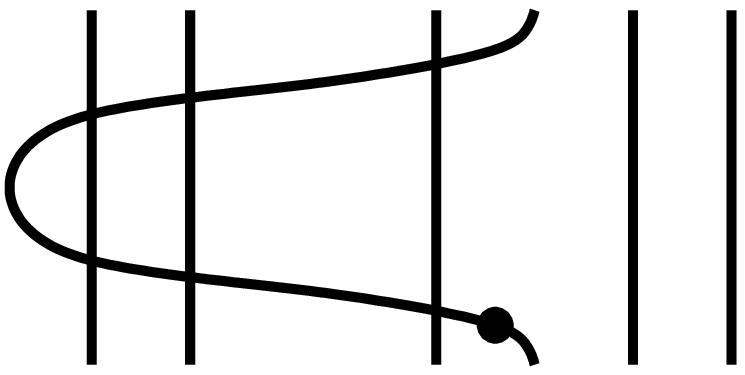}
\mspace{38mu} 
\displaybreak[0]\\[1ex]
&= \mspace{16mu}
\labellist
\pinlabel $\dotsc$ at   95 63    
\pinlabel $\dotsc$ at  232 63    
\pinlabel $\lambda$ at -10 85
\tiny \hair 2pt
\pinlabel $r_{i-1}$ at 145  36 
\pinlabel $i$      at  15   1 
\pinlabel $i+1$    at  46   0 
\pinlabel $n$      at 115   0
\pinlabel $i-1$    at 147   0
\endlabellist 
\figins{-30}{0.85}{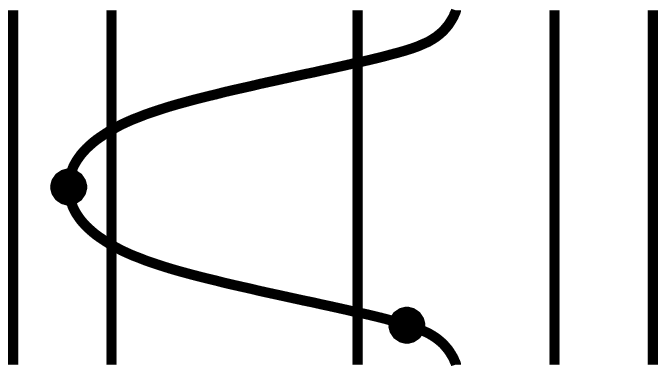} 
\mspace{34mu}
+
\mspace{15mu}
\labellist
\pinlabel $\dotsc$ at   95 63    
\pinlabel $\dotsc$ at  232 63 
\pinlabel $\lambda$ at -10 85 
\tiny \hair 2pt
\pinlabel $r_{i-1}$ at 145  36 
\pinlabel $i$      at  15   1 
\pinlabel $i+1$    at  46   0 
\pinlabel $n$      at 115   0
\pinlabel $i-1$    at 147   0
\endlabellist 
\figins{-30}{0.85}{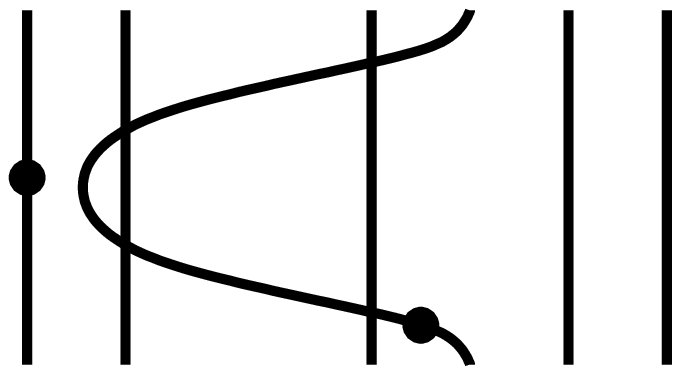} 
\mspace{24mu}
\end{align*}
where we used relation~\eqref{eq:R2}. 
Sliding the newly created dot close to the original $r_{i-1}$ dots on the first term 
and applying relation~\eqref{eq:R2} to both terms  
we get
\begin{equation}\label{eq:cyciminus}
X_i(i-1,r_{i-1}) =\mspace{14mu}   
\labellist
\pinlabel $\dotsc$ at   85 63    
\pinlabel $\dotsc$ at  242 63 
 \pinlabel $\lambda$ at -10 85
\tiny \hair 2pt
\pinlabel $\text{\rotatebox{82}{$r_{i-1}+1$}}$ at 160 85
\pinlabel $i$      at  15    1 
\pinlabel $i+1$    at  46    0 
\pinlabel $n$      at 115    0
\pinlabel $i-1$    at 147    0
\endlabellist
\figins{-30}{0.85}{cycR2j}
\mspace{44mu}
+
\mspace{14mu}   
\labellist
\pinlabel $\dotsc$ at   85 63    
\pinlabel $\dotsc$ at  242 63 
\pinlabel $\lambda$ at -10 85
\tiny \hair 2pt
\pinlabel $\text{\rotatebox{82}{$r_{i-1}$}}$ at 160 85
\pinlabel $i$      at  15    1 
\pinlabel $i+1$    at  46    0 
\pinlabel $n$      at 115    0
\pinlabel $i-1$    at 147    0
\endlabellist
\figins{-30}{0.85}{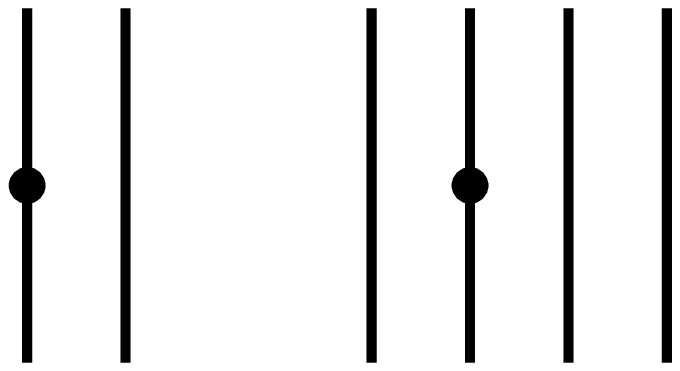}
\mspace{24mu}
\end{equation}
which consists of a term in $\widetilde{A}_i$ and a term in $\widetilde{A}_i^\bot$.  

For $j=i$ we compute
\begin{equation*}
X_i(i,r_{i}) = \mspace{22mu}  
\labellist
\pinlabel $\dotsc$ at   85 63    
\pinlabel $\dotsc$ at  252 63    
\pinlabel $\lambda$ at -15 95
\tiny \hair 2pt
\pinlabel $r_i$    at -10   65
\pinlabel $i$      at  28    1 
\pinlabel $i+1$    at  59    0 
\pinlabel $n$      at 128    -1
\pinlabel $i$      at 157    0
\endlabellist
\figins{-30}{0.85}{cycR2}
\mspace{38mu} 
= 
\sum_{\ell_1+\ell_2 = r_i-1}
\labellist
\pinlabel $\dotsc$ at   95 63    
\pinlabel $\dotsc$ at  232 63   
\pinlabel $\lambda$ at -10 75 
\tiny \hair 2pt
\pinlabel $i$      at  15   1 
\pinlabel $i+1$    at  46   0 
\pinlabel $n$      at 115   0
\pinlabel $i$      at 147   0
\pinlabel $\ell_1$ at 32 24 \pinlabel $\ell_2$ at 44 61
\endlabellist 
\figins{-30}{0.85}{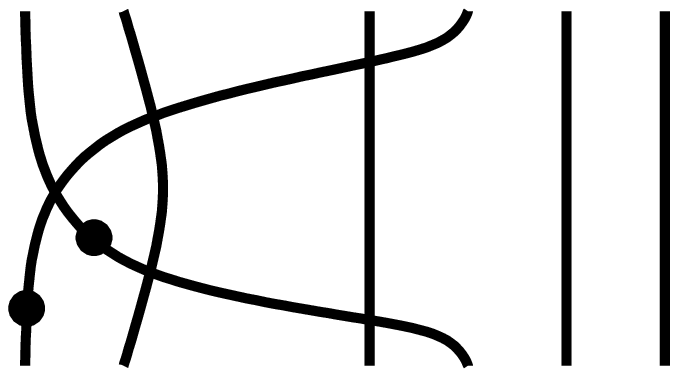} 
\mspace{24mu}
\end{equation*}
where we used relation~\eqref{eq:dotslides} followed by~\eqref{eq:R2}.  
Using~\eqref{eq:dotslides} to slide the $\ell_2$ dots to to upper part of the first strand from the left gives 

\begin{equation*}
X_i(i,r_{i}) =  
\sum_{\ell_1+\ell_2 = r_i-1}
\left(
\labellist
\pinlabel $\dotsc$ at   95 63    
\pinlabel $\dotsc$ at  232 63 
\pinlabel $\lambda$ at 0 75
\tiny \hair 2pt
\pinlabel $i$      at  15   1 
\pinlabel $i+1$    at  46   0 
\pinlabel $n$      at 115   -1
\pinlabel $i$      at 147   0
\pinlabel $\ell_1$ at 5 34 \pinlabel $\ell_2$ at 5 101
\endlabellist 
\figins{-30}{0.85}{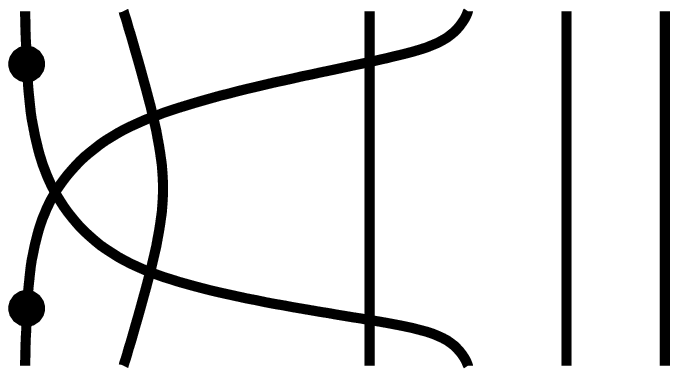} 
\mspace{36mu} +
\sum_{\kappa_1+\kappa_2=\ell_2-1}
\labellist
\pinlabel $\dotsc$ at   85 63    
\pinlabel $\dotsc$ at  232 63    
\pinlabel $\lambda$ at -10 75
\tiny \hair 2pt
\pinlabel $\ell_1+\kappa_1$ at -15 98 \pinlabel $\kappa_2$ at 33 42 
\pinlabel $i$       at  15   1 
\pinlabel $i+1$     at  46   0 
\pinlabel $n$       at 115   -1
\pinlabel $i$       at 147   0
\endlabellist 
\figins{-30}{0.85}{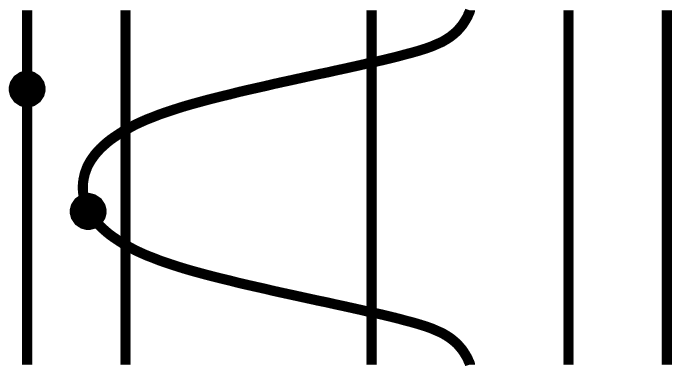} 
\mspace{32mu}
\right)
\end{equation*}
For $\bar{\lambda}_i=1,0$ we have $X_i(i,r_i)=0$ for all $r_i\geq 0$ and so it is enough to consider the 
case $r_i>1$ here. 
For $r_i>1$ the first term is in $\widetilde{A}_i^\bot$ and the second term is in $\widetilde{A}_i^\bot$ unless 
$\ell_1=\kappa_1=0$. This results in 
\begin{align*}
X_i(i,r_{i}) 
&=  \mspace{26mu}
\labellist
\pinlabel $\dotsc$ at   85 63    
\pinlabel $\dotsc$ at  232 63   
\pinlabel $\lambda$ at -10 75
\tiny \hair 2pt
\pinlabel $\text{\rotatebox{82}{$r_i-2$}}$ at 27 82 
\pinlabel $i$       at  15   1 
\pinlabel $i+1$     at  46   0 
\pinlabel $n$       at 115   -1
\pinlabel $i$       at 147   0
\endlabellist 
\figins{-30}{0.85}{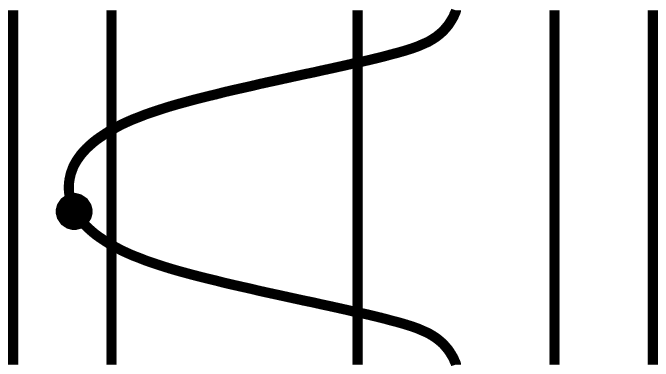} 
\mspace{38mu}
+\text{ terms in }\widetilde{A}_i^\bot 
\displaybreak[0]\\[1ex]
&\stackrel{\eqref{eq:cyciminus}}{=}
\mspace{20mu}
\labellist
\pinlabel $\dotsc$ at   85 63    
\pinlabel $\dotsc$ at  242 63 
\pinlabel $\lambda$ at -10 75
\tiny \hair 2pt
\pinlabel $\text{\rotatebox{82}{$r_{i}-1$}}$ at 160 85
\pinlabel $i$      at  15    1 
\pinlabel $i+1$    at  46    0 
\pinlabel $n$      at 115    -1
\pinlabel $i$      at 147    0
\endlabellist
\figins{-30}{0.85}{cycR2j}
\mspace{38mu}
+\text{ terms in }\widetilde{A}_i^\bot 
\end{align*}

Finally for $i<j<n$ we have 
\begin{align*}
X_i(j,r_{j}) &= \mspace{28mu}  
\labellist
\pinlabel $\dotsc$ at   85 63    
\pinlabel $\dotsc$ at  252 63    
\pinlabel $\lambda$ at -25 90 
\tiny \hair 2pt
\pinlabel $r_j$    at -10   65
\pinlabel $i$      at  28    1 
\pinlabel $i+1$    at  59    0 
\pinlabel $n$      at 128    -1
\pinlabel $j$      at 157    -1
\endlabellist
\figins{-30}{0.85}{cycR2}
\mspace{38mu} = \mspace{29mu}
\labellist
\pinlabel $\dotsc$ at   25 63   
\pinlabel $\dotsc$ at  115 63    
\pinlabel $\dotsc$ at  262 63    
\pinlabel $\lambda$ at -20 85 
\tiny \hair 2pt
\pinlabel $r_j$    at  40   40
\pinlabel $i$      at   2   1 
\pinlabel $j-1$    at  56   0 
\pinlabel $j$      at  88   0 
\pinlabel $n$      at 145   0
\pinlabel $j$      at 175   0
\endlabellist 
\figins{-30}{0.85}{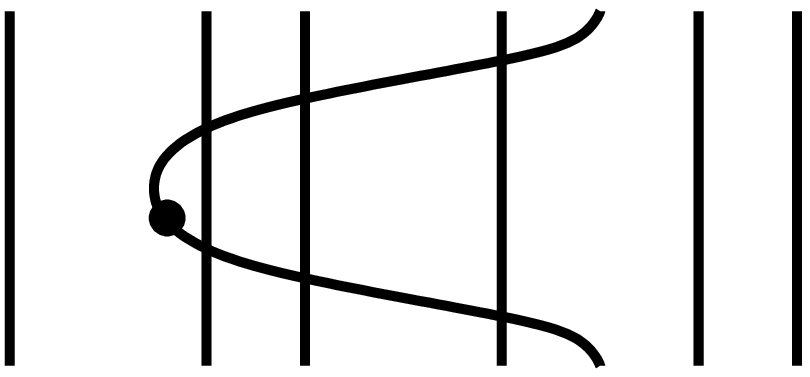} 
\mspace{24mu}
\displaybreak[0]\\[1ex]
&= \mspace{25mu}
\labellist
\pinlabel $\dotsc$ at   30 63   
\pinlabel $\dotsc$ at  115 63    
\pinlabel $\dotsc$ at  262 63 
\pinlabel $\lambda$ at -20 85 
\tiny \hair 2pt
\pinlabel $r_j$    at  75   40
\pinlabel $i$      at   2   1 
\pinlabel $j-1$    at  56   0 
\pinlabel $j$      at  88   0 
\pinlabel $n$      at 145   0
\pinlabel $j$      at 175   0
\endlabellist 
\figins{-30}{0.85}{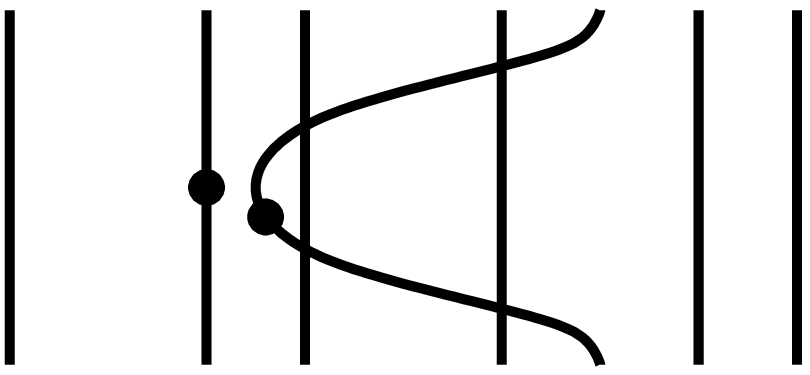} 
\mspace{36mu} + \mspace{22mu}
\labellist
\pinlabel $\dotsc$ at   30 63   
\pinlabel $\dotsc$ at  115 63    
\pinlabel $\dotsc$ at  262 63   
\pinlabel $\lambda$ at -20 85 
\tiny \hair 2pt
\pinlabel $\text{\rotatebox{75}{$r_j+1$}}$  at  72 95
\pinlabel $i$      at   2   1 
\pinlabel $j-1$    at  56   0 
\pinlabel $j$      at  88   0 
\pinlabel $n$      at 145   0
\pinlabel $j$      at 175   0
\endlabellist 
\figins{-30}{0.85}{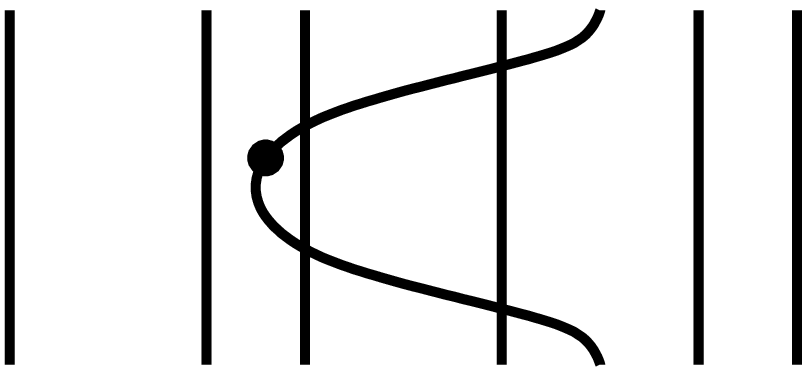} 
\mspace{24mu}
\end{align*}
The first term is in $\widetilde{A}_i^\bot$. For the second we use the result of the case of $i=j$, done above, 
to obtain
\begin{equation*}
X_i(j,r_j) =\mspace{22mu}   
\labellist
\pinlabel $\dotsc$ at   85 63    
\pinlabel $\dotsc$ at  252 63    
\pinlabel $\lambda$ at -10 75
\tiny \hair 2pt
\pinlabel $r_j$    at 130   65
\pinlabel $i$      at  15    1 
\pinlabel $i+1$    at  46    0 
\pinlabel $n$      at 115    0
\pinlabel $j$      at 145   -1
\endlabellist
\figins{-30}{0.85}{cycR2j}
\mspace{52mu}
+\text{ terms in }\widetilde{A}_i^\bot 
\end{equation*}

Taking $r_i=\bar{\lambda}_i$ we see that $X_i(j,\bar{\lambda}_j)$ consists 
of a sum of a term in $R_n^{\xi_i(\lambda)}$ 
with terms in $\widetilde{A}_i^\bot$, which shows that 
$R_{n+1}^\lambda(\alpha_n)$ projects onto $R_n^{\xi_i(\lambda)}$. 
We call this projection $\pi_i$.  The kernel of $\pi_i$ is the two-sided ideal 
generated by the elements in $\widetilde{A}_i^\bot$ involved above.  
Proceeding recursively one gets that $\pi_{i_1\dotsm i_k}$ is a surjection of algebras. 
The lemma now follows from the observation that $R_{n+1}^\lambda(k\alpha_n)$ projects canonically onto
 $\bigoplus\limits_{\xi_{i_1\dotsm i_k}\in\cD_\lambda^k}{A}_{i_1\dotsm i_k}$.  
\end{proof}

Summing over $k$ in Lemma~\ref{lem:cycinc} we have the following.
\begin{cor}
We have a surjection of algebras
\begin{equation*} 
 R^\lambda_{n+1} 
\xra{\ \ \pi^\lambda \ \ }
\bigoplus\limits_{\xi_{i_1\dotsm i_k}\in\cD_\lambda}R^{\xi_{i_1\dotsm i_k}(\lambda)}_n  .
\end{equation*}
\end{cor}

\medskip

Fix a $k\geq 1$ and 
let 
\begin{align*}
\Pi_k^\lambda:= \ext_k^\lambda
\colon
R^\lambda_{n+1}(k\alpha_n)-\amod
&\to
\bigoplus\limits_{\xi_{i_1 \dotsm i_k}\in\cD_\lambda^k}R^{\xi_{i_1\dotsm i_k}(\lambda)}_n-\amod 
\\[1ex]
 M &\mapsto M \otimes_{R^\lambda_{n+1}(k\alpha_n)}
\Bigl(\oplus_{\xi_{i_1 \dotsm i_k}}R^{\xi_{i_1\dotsm i_k}(\lambda)}_n\Bigr)
\end{align*}
and $\res_k^\lambda\colon\bigl(\oplus_{\xi_{i_1 \dotsm i_k}}R^{\xi_{i_1\dotsm i_k}(\lambda)}_n\bigr)-\amod 
\to 
R^\lambda_{n+1}(k\alpha_n)-\amod$ 
be respectively the functors of extension of scalars and restriction of scalars 
by the map $\pi_k$ from Lemma~\ref{lem:cycinc}.

Using the surjections 
$\pi_{i_1\dotsm i_k}\colon R_{n+1}^{\lambda}(k\alpha_n)\to R_n^{\xi_{i_1\dotsm i_k}(\lambda)}$
for each $\xi_{i_1\dotsm i_k}$ that are inherited from the map $\pi_k$ from Lemma~\ref{lem:cycinc} we call 
\begin{align*}
\Pi_{i_1\dotsm i_k}^\lambda\colon R_{n+1}^{\lambda}(k\alpha_n)-\amod &\to R_n^{\xi_{i_1\dotsm i_k}(\lambda)}-\amod 
\intertext{and} 
\res_{i_1\dotsm i_k}^\lambda\colon R_n^{\xi_{i_1\dotsm i_k}(\lambda)}-\amod &\to R_{n+1}^{\lambda}(k\alpha_n)-\amod
\end{align*} 
the \emph{components} of $\Pi_k^\lambda$ and $\res_k^\lambda$.
Functors  $\Pi_{i_1\dotsm i_k}^\lambda$ take projectives to projectives. 

\begin{lem}\label{lem:BR-biadj}
The functors $\Pi_k^\lambda$ and $\res_k^\lambda$ are biadjoint.
\end{lem}
\begin{proof}
The functor $\Pi_k^\lambda$ is left adjoint to $\res_k^\lambda$ by definition.  
We have to show that it is also its right adjoint. In other words, we have to show that 
it  coincides with the functor of \emph{co-extension of scalars} by $\pi_k$: 
\begin{align*}
\Cxt_k^\lambda
\colon
R^\lambda_{n+1}(k\alpha_n)-\amod
&\to
\bigoplus\limits_{\xi_{i_1 \dotsm i_k}\in\cD_\lambda^k}R^{\xi_{i_1\dotsm i_k}(\lambda)}_n-\amod 
\\[1ex]
 M &\mapsto \Hom_ {R^\lambda_{n+1}(k\alpha_n)}
\Bigl(\oplus_{\xi_{i_1 \dotsm i_k}}R^{\xi_{i_1\dotsm i_k}(\lambda)}_n, M\Bigr)
\end{align*} 
Since every object in these categories have a presentation
by projectives 
it is enough to show that they coincide as functors
from $R^\lambda_{n+1}(k\alpha_n)-\prmod$
to
$\bigl(\oplus_{\xi_{i_1 \dotsm i_k\in\cD_\lambda^k}}R^{\xi_{i_1\dotsm i_k}(\lambda)}_n\bigr)-\prmod$. 

For a projective ${}_{\und{i}}P$ in $R^\lambda_{n+1}(k\alpha_n)-\prmod$  
the projective $\Pi_k^\lambda({}_{\und{i}}P)$ has a basis given by all the diagrams that 
start in a sequence determined by $e(p_{i_1\dotsm i_k},\und{i}')$
and end in the sequence $\und{i}$, where the strands corresponding to $p_{i_1\dotsm i_k}$ 
do not carry any dots and are not allowed to cross among themselves. 
The algebra $\oplus_{\xi_{i_1 \dotsm i_k\in\cD_\lambda^k}}R^{\xi_{i_1\dotsm i_k}(\lambda)}_n$ acts by 
inclusion in $\oplus_{\xi_{i_1 \dotsm i_k\in\cD_\lambda^k}}\widetilde{A}^{\xi_{i_1\dotsm i_k}(\lambda)}$ followed 
by composition on the bottom of a diagram from $\Pi_k^\lambda({}_{\und{i}}P)$. 
This coincides with the definition of the functor of co-extension by $\pi_k$ on objects.
The same argument works in the check that both functors coincide on morphisms as well.  
\end{proof}

\begin{lem}\label{lem:BRfull}
The functor $\Pi_k^\lambda$ is full and essentially surjective.
\end{lem}

\begin{proof}
Fullness of $\Pi_k^\lambda$ is a consequence of the surjectivity of $\pi_k$
and the definition of $\Pi_k^\lambda$. 
The same argument can be used to prove that each component 
$\Pi_{i_1\dotsm i_k}^\lambda$ is essentially surjective, 
since every projective ${}_{\und{i}}P$ in $R^{\xi_{i_1\dotsm i_k}(\lambda)}_n-\prmod$ 
can be obtained as 
$\Pi_{i_1\dotsm i_k}^\lambda({}_{e(p_{i_1\dotsm i_k},\und{i})}P)$ for ${}_{e(p_{i_1\dotsm i_k},\und{i})}P$ 
in $R_{n+1}^\lambda(k\alpha_n)-\prmod$.
\end{proof}

\medskip

\begin{lem}\label{lem:branch-k}
Each functor $\Pi^\lambda_{i_1\dotsm i_k}$  
intertwines the categorical $\sln$-action. 
\end{lem}
\begin{proof}
It is clear that the projection map $\pi_{i_1\dotsm i_k}\colon R_{n+1}^\lambda(k\alpha_n)\to R_n^{\xi_{i_1\dotsm i_k}(\lambda)}$ 
commutes with the 
map $\phi_j\colon R_{n+1}^\lambda(k\alpha_n)\to R_{n+1}^\lambda(k\alpha_n)$ that adds a vertical strand 
labeled $j\in\{1,\dotsc ,n-1\}$ on the right of a diagram from $R_{n+1}^\lambda(k\alpha_n)$. 
This induces a natural isomorphism of functors 
$\Pi_{i_1\dotsm i_k}F_j^{\lambda}\xra{ \cong }F_j^{\xi_{i_1\dotsm i_k}(\lambda)}\Pi_{i_1\dotsm i_k}$. 
Now consider $E_j^\lambda$ and $E_j^{\xi_{i_1\dotsm i_k}(\lambda)}$. 
Recall that for a projective ${}_{\und{i}}P$ in $R_{n+1}^\lambda(k\alpha_n)$ with $\und{i}\in\seq(\beta)$ 
the projective $E_j^\lambda({}_{\und{i}}P)$ has a basis given by all diagrams starting in a sequence 
$\und{i}'j\in\seq(\beta)$ for fixed $j$ and ending up in the sequence $\und{i}$ and that 
$\Pi_{i_1\dotsm i_k}({}_{\und{i}}P)$ has an analogous description. 
The isomorphism between $\Pi_{i_1\dotsm i_k}E_j^\lambda$ and 
$E_j^{\xi_{i_1\dotsm i_k}(\lambda)}\Pi_{i_1\dotsm i_k}$ 
now follows from comparison between the vector spaces  $\Pi_{i_1\dotsm i_k}E_j^\lambda({}_{\und{i}}P)$ and 
$E_j^{\xi_{i_1\dotsm i_k}(\lambda)}\Pi_{i_1\dotsm i_k}({}_{\und{i}}P)$. 
\end{proof}

\begin{prop}\label{prop:branchrules}
\n Each functor $\Pi^\lambda_{i_1\dotsm i_k}$ descends to  
a surjection 
\begin{equation*}
K_0(\Pi^\lambda_{i_1\dotsm i_k})\colon K_0(R_{n+1}^\lambda(k\alpha_n))\to K_0(R_n^{\xi_{i_1\dotsm i_k}(\lambda)})
\end{equation*} 
of $\sln$-representations. 
\end{prop}

\medskip 

Finally define the functor 
\begin{equation}\label{eq:brfunctor}
\Pi^{\lambda}:=\bigoplus_{k\geq 0}\Pi^\lambda_{k} 
\colon 
 R^\lambda_{n+1}-\amod
\to
\bigoplus\limits_{\xi_{i_1\dotsm i_k}\in\cD_\lambda}R^{\xi_{i_1\dotsm i_k}(\lambda)}_n-\amod .
\end{equation} 
Functor $\Pi^\lambda$ is full, essentially surjective 
and  intertwines 
the  $\sln$-action by Lemmas~\ref{lem:BR-biadj} to~\ref{lem:branch-k}.

Combining Proposition~\ref{prop:branchrules} with Theorem~\ref{thm:BK} we have the main result of this section, 
which follows easily by counting dimensions.
\begin{thm}\label{thm:branchrules}
Functor $\Pi^{\lambda}$ descends to 
an isomorphism of $\sln$-representations
\begin{align*} 
K_0(\Pi^{\lambda} )
\colon V^{\sll}_{\lambda} \cong K_0 ( R^\lambda_{n+1} )
\xra{\ \ \cong\ \ }
K_0\biggl(\ \bigoplus\limits_{\xi_{i_1\dotsm i_k}\in\cD_\lambda}R^{\xi_{i_1\dotsm i_k}(\lambda)}_n\biggr)
\cong\bigoplus\limits_{\mu\in\tau(\lambda)} V_{\mu}^{\sln} .
\end{align*}
\end{thm}

\begin{cor}\label{cor:BRinj}
The functor $\Pi^{\lambda}$ is injective on objects. 
\end{cor}

\begin{proof}
From the results of Subsection~\ref{ssec:factidemp} translated into the cyclotomic setting 
we see that 
it is enough to prove that  
 $\Pi^{\lambda}$ is injective on the collection of objects 
which can be expressed as direct sums of projectives over $R_{n+1}^\lambda$ of the form
${}_{e'(p_{i_1\dotsm i_k})}P$.
We proceed by induction on the reverse order of the lexicographic order on the $p_{i_1\dotsm i_k}$s 
(this is induced by the lexicographic order on the $k$-tuples  $(i_1, \dotsc , i_k)\in\bZ_{> 0}^k$). 
The base case $(i_1, \dotsc , i_k)=(n,\dotsc , n)$ being trivial we proceed to the general case.  
The fact that $\Pi^{\lambda}{}_{e'(p_{i_1\dotsm i_k})}P =0$ implies that 
$$\Hom_{R_{n+1}^\lambda-\amod}({}_{e'(p_{i_1\dotsm i_k})}P, {}_{e(p_{j_1\dotsm j_k},\und{j})}P)=0$$
for all $p_{j_1\dotsm j_k}$ greater that $p_{i_1\dotsm i_k}$ in the lexicographic order. 
By the induction hypothesis this implies that
\begin{equation}\label{eq:homzero}
\Hom_{R_{n+1}^\lambda-\amod}({}_{e'(p_{i_1\dotsm i_k})}P, {}_{e'(p_{m_1\dotsm m_k},\und{m})}P)=0
\end{equation}
for all $p_{m_1\dotsm m_k}$ greater that $p_{i_1\dotsm i_k}$ in the lexicographic order,  
since all but the special projectives in $\{ {}_{e'(p_{m_1\dotsm m_k},\und{m})}P \}_{p_{m_1\dotsm m_k} > p_{i_1\dotsm i_k} }$ 
are zero. 
Since every object in $R_{n+1}^\lambda$ is isomorphic to a direct summand in 
$\bigoplus_{i_1\dotsm i_k,\und{i}}P^{\lambda}\{ s_{i_1\dotsm i_k,\und{i}}\}^{m_{i_1\dotsm i_k,\und{i}} }$ 
this implies that
$$\End_{R_{n+1}^\lambda-\amod}({}_{e'(p_{i_1\dotsm i_k})}P) = {e'(p_{i_1\dotsm i_k})}R_{n+1}^\lambda{e'(p_{i_1\dotsm i_k})}$$
is one-dimensional, for any diagram other than 
the one consisting only of vertical strands without dots factors through 
elements in a sum of spaces of the form~\eqref{eq:homzero}, as explained in Subsection~\ref{ssec:factidemp}.  
In particular this implies that ${}_{e'(p_{i_1\dotsm i_k})}P$ 
is indecomposable, which contradicts Theorem~\ref{thm:branchrules} for otherwise 
$K_0(\Pi^\lambda)({}_{e'(p_{i_1\dotsm i_k})}P)$ would be nonzero since
$K_0({}_{e'(p_{i_1\dotsm i_k})}P)$ is nonzero and $K_0(\Pi^\lambda)$ is an isomorphism.
\end{proof}

Lemma~\ref{lem:BRfull} and Corollary~\ref{cor:BRinj} altogether imply that 
the category $\bigoplus\limits_{\xi_{i_1\dotsm i_k}\in\cD_\lambda}R^{\xi_{i_1\dotsm i_k}(\lambda)}_n-\prmod$
contains all the objects of $R^\lambda_{n+1}-\prmod$.

\subsection{Categorical projection}\label{ssec:catproj}

Although $\Pi_k^\lambda$ has a nice behavior on projectives this is not the case 
for $\res_k^\lambda$. Functor $\res_k^\lambda$ takes a projective ${}_{\und{j}}P$ 
in $R^{\xi_{i_1\dotsm i_k}(\lambda)}_n-\prmod$ to the 
$R_{n+1}^\lambda(k\alpha_n)$-module
${}_{i_1\dotsm i_k,\und{j}}L$, 
which is a quotient of the projective ${}_{i_1\dotsm i_k,\und{j}}P$ over $R_{n+1}^\lambda(k\alpha_n)$. 
Module ${}_{i_1\dotsm i_k,\und{j}}L$ has a presentation by the span of the subset of the set of all diagrams 
from ${}_{i_1\dotsm i_k,\und{j}}P$ whose strands can be regarded as belonging to two groups, 
one consisting of the usual KLR-strands from ${}_{\und{j}}P$ satisfying the KLR relations, 
and other consisting of $\sum_{s=1}^k(n-i_s +1)$ strands, 
with labels that at any height 
are ordered from left two right according to $p_{i_1\dotsm i_k}$.  
Strands from the second group cannot intersect among themselves nor 
carry dots but they can intersect strands from the first group. 
We can also regard a diagram in ${}_{i_1\dotsm i_k,\und{j}}L$ as the overlap of a 
diagram from ${}_{\und{j}}P$ and $\sum_{s=1}^k(n-i_s +1)$ strands that run 
parallel to each other, do not carry dots, have labels determined by  
$p_{i_1\dotsm i_k}$, and end up at the left of all strands from ${}_{\und{j}}P$, as for example in 
\begin{equation*}
\labellist
\small
\pinlabel ${}_{\und{j}}P$ at 220 116  
\pinlabel $\lambda$ at -40 140
\tiny \hair 2pt
\pinlabel $\dotsc$ at  35 190 
\pinlabel $\dotsc$ at  94 190 
\pinlabel $\dotsc$ at 220 190 
\pinlabel $\dotsc$ at  80  40
\pinlabel $\dotsc$ at 225  40
\pinlabel $i_1$ at   4 220
\pinlabel $i_r$ at  65 219
\pinlabel $n$   at 119 219
\pinlabel $j_1$ at 164 220
\pinlabel $j_m$ at 278 220
\endlabellist 
\figins{-20.5}{1.1}{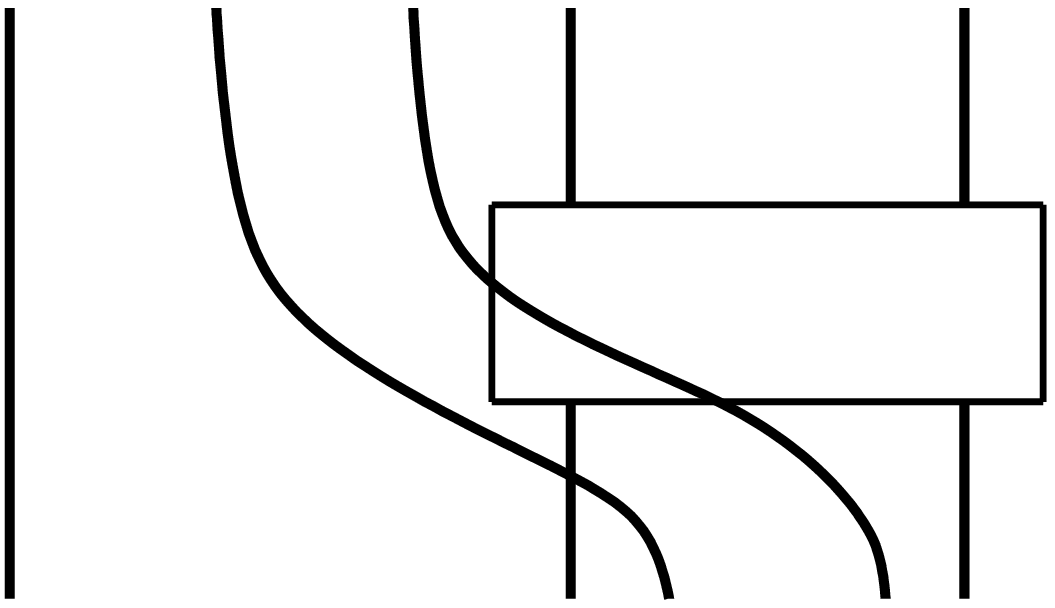}
\end{equation*}

We have that $\Pi^\lambda_{\ell_1,\dotsc ,\ell_r}({}_{i_1\dotsm i_k,\und{j}}L)$  
is isomorphic to ${}_\und{j}P$ if $(\ell_1,\dotsc ,\ell_r)=(i_1\dotsm i_k)$ and 
is zero otherwise. 
We also have that  
the functor $\bigoplus_{\xi_{i_1,\dotsc ,i_k}\in\cD_{\lambda}}\Pi^\lambda_{i_1,\dotsc ,i_k}\res^\lambda_{i_1,\dotsc ,i_k} $ 
is isomorphic to the identity functor of $\bigoplus_{\xi_{i_1\dotsm i_k}\in\cD_\lambda}R^{\xi_{i_1\dotsm i_k}(\lambda)}_n-\amod$. 
We can think of the set of the composite functors 
$\{\Pi^\lambda_{i_1,\dotsc ,i_k}\res^\lambda_{i_1,\dotsc ,i_k}\}_{\xi_{i_1,\dotsc ,i_k}\in\cD_{\lambda}}$ 
as a categorical version of a complete system of orthogonal idempotents. 
This is another way of seeing that $R^\lambda_{n+1}-\amod$ contains  
$\bigoplus_{\xi_{i_1\dotsm i_k}\in\cD_\lambda}R^{\xi_{i_1\dotsm i_k}(\lambda)}_n-\amod$.

\medskip

The collection of diagrams described above can be given the structure of associative $\Bbbk$-algebra if 
we do not force the labels on the top to be ordered according to $e(p_{i_1\dotsm i_k},\und{j})$  
(of course, the labels of the strands from the second group have to end up in the order determined by $p_{i_1\dotsm i_k}$) 
and impose the relations inherited from the KLR relations~\eqref{eq:R2}-\eqref{eq:dotslide}.  
Denote this algebra $\widetilde{R}^\lambda_{n+1}(k\alpha_n;p_{i_1\dotsm i_k},\und{j})$ 
and define 
\begin{defn} 
$\widetilde{R}^\lambda_{n+1}(k\alpha_n;p_{i_1\dotsm i_k}) =  \bigoplus_{\und{j}}\widetilde{R}^\lambda_{n+1}(k\alpha_n;p_{i_1\dotsm i_k},\und{j})$. 
\end{defn}
Each element of $\widetilde{R}^\lambda_{n+1}(k\alpha_n;p_{i_1\dotsm i_k})$ can be thought of as an overlap of two diagrams, one from 
$R^\lambda_{n+1}(\nu(i_1\dotsm i_k))$ for some $\nu\in\Lambda_+^{n+1}$ 
and another one from the 1-dimensional algebra consisting of the single 
element given by $\sum_{s=1}^k(n-i_s+1)$ vertical strands without dots and labeled 
in the order determined by $p_{i_1\dotsm i_k}$  
(see Subsection~\ref{ssec:factidemp} for the definition of $\nu(i_1\dotsm 1_k)$). 
\begin{defn}
We define the algebras $\widetilde{R}^\lambda_{n+1}(k\alpha_n)$ and $\widetilde{R}^\lambda_{n+1}$ by 
\begin{align*}
\widetilde{R}^\lambda_{n+1}(k\alpha_n) &= \bigoplus_{\xi_{i_1\dotsm i_k}\in\cD_\lambda^k} \widetilde{R}^\lambda_{n+1}(k\alpha_n;p_{i_1\dotsm i_k})  
\intertext{and}
\widetilde{R}^\lambda_{n+1}      &= \bigoplus_{k\geq 0 } \widetilde{R}^\lambda_{n+1}(k\alpha_n)  .   
\end{align*} 
\end{defn}

We now describe $\widetilde{R}^\lambda_{n+1}(k\alpha_n)$ more intrinsically. 
The kernel $K$ of the action of $R_{n+1}^{\lambda}(\nu)$ on ${}_{i_1\dotsm i_k,\und{j}}L$ contains all diagrams 
in $\widetilde{R}^\lambda_{n+1}(k\alpha_n)$ that have either a crossing between strands belonging to the second group 
or a dot on one of its strands. It is not hard to see that this collection of diagrams exhaust $K$. 
Let $J_{i_1\dotsm i_k,\und{j}}\subset R_{n+1}^{\lambda}(k\alpha_n)$  be the two-sided ideal generated by $K$  
and define the ideal $J_{i_1\dotsm i_k}=\sum_{\und{j}}J_{i_1\dotsm i_k,\und{j}}$ 
where $\und{j}$ runs over all sequences of simple roots in $\Lambda_+^n$.   
%
All the above adds up to the following. 
\begin{lem}
We have an isomorphism of algebras 
$\widetilde{R}^\lambda_{n+1}(k\alpha_n;p_{i_1\dotsm i_k})\cong R^\lambda_{n+1}(k\alpha_n)/ J_{i_1\dotsm i_k}$.
\end{lem}

\n Moreover, we also have
\begin{lem}
Module ${}_{i_1\dotsm i_k,\und{j}}L$ is projective as a module over $\widetilde{R}^\lambda_{n+1}(k\alpha_n)$.
\end{lem}
\begin{proof}
We have that the element 
$e(p_{i_1\dotsm i_k},\und{j})$ is an idempotent in $\widetilde{R}^\lambda_{n+1}(k\alpha_n)$  
and the module 
${}_{i_1\dotsm i_k,\und{j}}L = e(p_{i_1\dotsm i_k},\und{j})\widetilde{R}^\lambda_{n+1}(k\alpha_n)$.  
\end{proof}

\medskip 

Functor $\res_k^\lambda$ can be regarded as a functor from 
$\bigoplus_{\xi_{i_1\dotsm i_k}\in\,\cD_k^\lambda}R_n^{\xi_{i_1\dotsm i_k}(\lambda)}-\amod$
to 
$\widetilde{R}^\lambda_{n+1}(k\alpha_n)-\amod$, 
the latter category seen as the quotient of $R^\lambda_{n+1}(k\alpha_n)-\amod$ by $\sum_{\xi_{i_1\dotsm i_k}\in\cD^\lambda_k}J_{i_1\dotsm i_k}$.   
This functor takes projectives to projectives. 
With this in mind we see that the quotient functor
\begin{equation}\label{eq:quotQ}
\cQ_k\colon R^\lambda_{n+1}(k\alpha_n)-\amod \to \widetilde{R}^\lambda_{n+1}(k\alpha_n)-\amod
\end{equation} 
is isomorphic to the functor $\res_k^\lambda\Pi_k^\lambda$. 
Moreover, the functor $\Pi^\lambda$ descends to a functor 
$$\widetilde{\Pi}^\lambda \colon 
\widetilde{R}^\lambda_{n+1}-\amod \to \bigoplus\limits_{\xi_{i_1\dotsm i_k}\in\,\cD^\lambda}R_n^{\xi_{i_1\dotsm i_k}(\lambda)}-\amod$$ 
which is full and essentially surjective by Lemma~\ref{lem:BRfull}. 
It is also faithful by the definition of the projections $\pi_k$ . 
The categories $\widetilde{R}^\lambda_{n+1}-\amod$ and 
$\bigoplus_{\xi_{i_1\dotsm i_k}\in\,\cD^\lambda}R_n^{\xi_{i_1\dotsm i_k}(\lambda)}-\amod$ 
are therefore equivalent.  
Denote by 
\begin{align*}
\widetilde{\Pi}_k^\lambda &\colon 
\widetilde{R}^\lambda_{n+1}(k\alpha_n)-\amod \to \bigoplus\limits_{\xi_{i_1\dotsm i_k}\in\,\cD_k^\lambda}R_n^{\xi_{i_1\dotsm i_k}(\lambda)}-\amod
\intertext{and}
\widetilde{\res}_k^\lambda &\colon
 \bigoplus\limits_{\xi_{i_1\dotsm i_k}\in\,\cD_k^\lambda}R_n^{\xi_{i_1\dotsm i_k}(\lambda)}-\amod \to
\widetilde{R}^\lambda_{n+1}(k\alpha_n)-\amod 
\end{align*}
The functors induced by $\Pi_k^\lambda$ and $\res_k^\lambda$ from Subsection~\ref{ssec:catbranching}.  
It is not hard to see that the composite functors  
$\widetilde{\res}_k^\lambda\widetilde{\Pi}_k^\lambda$ and $\widetilde{\Pi}_k^\lambda\widetilde{\res}_k^\lambda$ 
are isomorphic to the identity functors on  $\widetilde{R}^\lambda_{n+1}-\amod$
and 
$\bigoplus_{\xi_{i_1\dotsm i_k}\in\,\cD^\lambda}R_n^{\xi_{i_1\dotsm i_k}(\lambda)}-\amod$.

\section{Categorifying the Gelfand-Tsetlin basis}\label{sec:GelfTse}    

\subsection{Recovering the categorical $\sll$-action}\label{sec:recsll}%

A key step in constructing a categorical $\sll$-action on the categorified 
Gelfand-Tsetlin basis consists in recovering the categorical 
$\sll$-action on $\Pi^\lambda(R_{n+1}^\lambda-\amod)$
from its categorical action on $R_{n+1}^\lambda-\amod$.
This amounts to understand the interplay between the functors
$E_{n}^\lambda$, $F_{n}^\lambda$,  $\Pi^\lambda$ and $\res_k^\lambda$. 
For a module $M$ in $R_{n+1}^\lambda-\amod$ 
we want to see how the functors $E_{n}^\lambda$, $F_{n}^\lambda$ 
allow to move between the components $\Pi_{i_1\dotsm i_k}(M)$
of the different categories
$R^{\xi_{i_1\dotsm i_k}(\lambda)}_n-\amod$.

Contrary to the case of $\Pi^\lambda_{i_1\dotsm i_k}$ 
the functors $\res_{i_1\dotsm i_k}^\lambda$ do not intertwine the categorical $\sln$-action.  
Nevertheless we can define functors 
\begin{equation*}
E_j^{\xi_{i_1\dotsm i_k}^{j_1\dotsm j_r}(\lambda)}, F_j^{\xi_{i_1\dotsm i_k}^{j_1\dotsm j_r}(\lambda)}
\colon
R_{n}^{\xi_{i_1\dotsm i_k}(\lambda)}-\amod\to R_{n}^{\xi_{j_1\dotsm j_r}(\lambda)}-\amod
\end{equation*}
for $j=1,\dotsc , n-1$, by
\begin{align*}
E_j^{\xi_{i_1\dotsm i_k}^{j_1\dotsm j_r}(\lambda)} &= \Pi^\lambda_{j_1\dotsm j_r}E_j^\lambda\res^\lambda_{\xi_{i_1\dotsm i_k}}
\intertext{and} 
F_j^{\xi_{i_1\dotsm i_k}^{j_1\dotsm j_r}(\lambda)} &= \Pi^\lambda_{j_1\dotsm j_r}F_j^\lambda\res^\lambda_{\xi_{i_1\dotsm i_k}}
\end{align*}

\begin{lem}
For $j\in\{1,\dotsc ,n-1\}$ the functors  $E_j^{\xi_{i_1\dotsm i_k}^{j_1\dotsm j_r}(\lambda)}$ 
and $F_j^{\xi_{i_1\dotsm i_k}^{j_1\dotsm j_r}(\lambda)}$ are zero unless $(i_1,\dotsc , i_k)=(j_1,\dotsc, j_r)$.  
In this case they coincide with the functors $E_j^{\xi_{i_1 \dotsm i_k}(\lambda)}$ and 
$F_j^{\xi_{i_1 \dotsm i_k}(\lambda)}$ inherit from the structure of categorical $\sll$-module 
on $R_{n+1}^\lambda-\amod$. 
\end{lem}

\begin{proof}
From Lemma~\ref{lem:branch-k} we have, for $j\in\{1,\dotsc ,n-1\}$,  
\begin{align*}
E_j^{\xi_{i_1\dotsm i_k}^{j_1\dotsm j_r}(\lambda)} &= \Pi^\lambda_{j_1\dotsm j_r}E_j^\lambda\res^\lambda_{\xi_{i_1\dotsm i_k}}
\cong E_j^{\xi_{j_i\dotsm j_r}(\lambda)}\Pi^\lambda_{j_1\dotsm j_r} \res^\lambda_{\xi_{i_1\dotsm i_k}} .
\end{align*} 
The claim follows from the fact that the functor 
$\Pi^\lambda_{j_1\dotsm j_r} \res^\lambda_{\xi_{i_1\dotsm i_k}}$
is the identity functor acting on $R_n^{\xi_{i_1\dotsm i_k}(\lambda)}$ if the sequences $(i_1,\dotsc , i_k)$ and $(j_1,\dotsc, j_r)$  
are equal
or the zero functor if they are different.  
The same reasoning proves the case of 
$E_j^{\xi_{i_1\dotsm i_k}^{j_1\dotsm j_r}(\lambda)}$. 
\end{proof}

\begin{defn}
For $i\in\{1,\dotsc ,n-1\}$ we define the functors
\begin{align*}
F_j^\xi := \bigoplus\limits_{\xi_{i_1\dotsm i_k}\in\cD_\lambda} F_j^{\xi_{i_1\dotsm i_k}^{i_1\dotsm i_k}(\lambda)}  
\mspace{25mu}\text{and}\mspace{25mu} 
E_j^\xi := \bigoplus\limits_{\xi_{i_1\dotsm i_k}\in\cD_\lambda} E_j^{\xi_{i_1\dotsm i_k}^{i_1\dotsm i_k}(\lambda)} 
\end{align*}
with the obvious source and target categories.
\end{defn}

Let us now treat the case of the functors
 $F_n^\lambda$ and $E_n^\lambda$. 
Each object ${}_{i_1\dotsm i_k,\und{j}}L$ in  $\widetilde{R}^\lambda_{n+1}-\prmod$ is also an object in $R^\lambda_{n+1}-\amod$ 
which is not projective in general. 
It is not hard to see that 
the projective cover of ${}_{i_1\dotsm i_k,\und{j}}L$ in $R^\lambda_{n+1}-\amod$ is 
${}_{e(p_{i_1\dotsm i_k},\und{j})}P$.  
Lemmas~\ref{lem:frobKLR} and~\ref{lem:BRfull} together with Corollary~\ref{cor:BRinj} imply 
that every object in $R^\lambda_{n+1}-\prmod$ arises this way.  
For an endofunctor $G$ acting on $R^\lambda_{n+1}-\prmod$ we define a functor $\widetilde{G}$ on 
$\widetilde{R}^\lambda_{n+1}-\prmod$ as follows. 
For an object $M$ in $\widetilde{R}^\lambda_{n+1}-\prmod$ 
we define $\widetilde{G}(M)$ as $QG(P(M))$
where $P(M)$ is the projective cover of $M$ in  
${R}^\lambda_{n+1}-\amod$ and $\cQ=\oplus_{k\geq 0}\cQ_k$ is the quotient functor from Eq.~\eqref{eq:quotQ}. 
The action of $\widetilde{G}$  on a morphism $f$ in $\Hom_{\widetilde{R}^\lambda_{n+1}-\prmod}(M, M')$ 
is defined in an analogous way. This operations are well defined, because the composite 
$QG'P(QGP(M))$ is isomorphic to $QG'GP(M)):=\widetilde{G'G}(M)$ for $G'$ an endofunctor on ${R}^\lambda_{n+1}-\prmod$.   
For morphisms $f$, $f'$ we observe that $P(QGP(f))$ equals $GP(f)$  yielding 
$\widetilde{G}(f'f)=\widetilde{G}(f')\widetilde{G}(f)$. 
\begin{lem}
The pair of (biadjoint) endofunctors $\{\widetilde{F}_n^\lambda,\widetilde{E}_n^\lambda \}$ 
take projectives to projectives and 
define a categorical $\mathfrak{sl}_2$-action on  $\widetilde{R}^\lambda_{n+1}-\prmod$. 
\end{lem}
This action extends canonically to a categorical action on $\widetilde{R}^\lambda_{n+1}-\amod$. 
We now use this result to construct a $\mathfrak{sl}_2$-pair of functors $\{F_n^\xi, E_n^\xi\}$ acting 
on the category 
$\bigoplus_{\xi_{i_1\dotsm i_k}\in\cD_\lambda}R^{\xi_{i_1\dotsm i_k}(\lambda)}_n-\amod$. 
We first define the functors
\begin{equation*}
{F}_n^{\xi_{i_1\dotsm i_k}^{j_1\dotsm j_r}(\lambda)}, {E}_n^{\xi_{i_1\dotsm i_k}^{j_1\dotsm j_r}(\lambda)}
\colon
R_{n}^{\xi_{i_1\dotsm i_k}(\lambda)}-\amod\to R_{n}^{\xi_{j_1\dotsm j_r}(\lambda)}-\amod
\end{equation*}
by
\begin{align*}
{F}_n^{\xi_{i_1\dotsm i_k}^{j_1\dotsm j_r}(\lambda)} &= \widetilde{\Pi}^\lambda_{j_1\dotsm j_r}\widetilde{F}_n^\lambda\widetilde{\res}^\lambda_{\xi_{i_1\dotsm i_k}}
\intertext{and} 
{E}_n^{\xi_{i_1\dotsm i_k}^{j_1\dotsm j_r}(\lambda)} &= \widetilde{\Pi}^\lambda_{j_1\dotsm j_r}\widetilde{E}_n^\lambda\widetilde{\res}^\lambda_{\xi_{i_1\dotsm i_k}}
\end{align*}
Both functors are zero if the sequences $(i_1,\dotsc , i_k)$ and $(j_1,\dotsc, j_r)$ 
have the same length, since this would mean that the source and target categories 
would correspond to weights of the form $\nu+k\alpha_n$ with $\nu=\nu_1\alpha_1+\dotsm + \nu_{n-1}\alpha_{n-1}$ 
with common $k$ (all diagrams in $\widetilde{R}^\lambda_{n+1}-\prmod$ and ${R}^\lambda_{n+1}-\prmod$  
would contain the same number of strands labeled $n$) 
and both functors $E_n^\lambda$ and $F_n^\lambda$ 
(and $\widetilde{E}_n^\lambda$, $\widetilde{F}_n^\lambda$)
change the number of strands labeled $n$.   
Summing over all ${i_1\dotsm i_k}$ and all $k$ we get the functors
\begin{equation*}
{F}_n^{\xi},\ {E}_n^{\xi} 
\colon
\bigoplus\limits_{\xi_{i_1\dotsm i_k}\in\cD_\lambda}R_{n}^{\xi_{i_1\dotsm i_k}(\lambda)}-\amod
\ \to\
\bigoplus\limits_{\xi_{i_1\dotsm i_k}\in\cD_\lambda}R_{n}^{\xi_{i_1\dotsm i_k}(\lambda)}-\amod
\end{equation*}
given by
\begin{equation*}
{F}_n^{\xi} := \bigoplus\limits_{\substack{\xi_{i_1\dotsm i_k}\in\cD_\lambda\\[0.5ex] \xi_{j_1\dotsm j_r}\in\cD_\lambda}}
{F}_n^{\xi_{i_1\dotsm i_k}^{j_1\dotsm j_r}(\lambda)} 
\mspace{45mu}\text{and}\mspace{45mu}
{E}_n^{\xi} := \bigoplus\limits_{\substack{\xi_{i_1\dotsm i_k}\in\cD_\lambda\\[0.5ex] \xi_{j_1\dotsm j_r}\in\cD_\lambda}}
{E}_n^{\xi_{i_1\dotsm i_k}^{j_1\dotsm j_r}(\lambda)} . 
\end{equation*}

\begin{prop}
The functors ${F}_n^\xi$ and ${E}_n^\xi$ take projectives to projectives and 
define a categorical $\mathfrak{sl}_2$-action on 
$$\bigoplus\limits_{\xi_{i_1\dotsm i_k}\in\cD_\lambda}R^{\xi_{i_1\dotsm i_k}(\lambda)}_n-\amod .$$  
\end{prop}
\begin{proof}
The first claim is a consequence of the definition of 
functors $\{ \widetilde{F}_n^\lambda, \widetilde{E}_n^\lambda\}$.
Biadjointness is a consequence of biadjointness of the pair $\{ F_n, E_n \}$ and the definition of $\{F_n^\xi, E_n^\xi \}$. 
Since $\bigoplus_{\xi_{i_1\dotsm i_k}\in\cD_\lambda^k}\widetilde{\res}^\lambda_{i_1\dotsm i_k}\widetilde{\Pi}_{i_i\dotsm i_k}^\lambda$ is the 
identity functor on $\widetilde{R}_{n+1}^\lambda(k\alpha_n)-\amod$ we have
\begin{align*}
{F}_n^{\xi}
{E}_n^{\xi}
&= \bigoplus\limits_{\substack{\xi_{i_1\dotsm i_k}\in\cD_\lambda\\[0.5ex] \xi_{j_1\dotsm j_r}\in\cD_\lambda\\[0.5ex]
\xi_{\ell_1\dotsm \ell_s}\in\cD_\lambda\\[0.5ex] \xi_{m_1\dotsm m_t}\in\cD_\lambda}}
\widetilde{\Pi}^\lambda_{j_1\dotsm j_r}\widetilde{F}_n^\lambda\widetilde{\res}^\lambda_{\xi_{j_1\dotsm j_r}}
\widetilde{\Pi}^\lambda_{\ell_1\dotsm \ell_s}\widetilde{E}_n^\lambda\widetilde{\res}^\lambda_{\xi_{m_1\dotsm m_t}}
\\[1ex]
&\cong
\bigoplus\limits_{\substack{\xi_{i_1\dotsm i_k}\in\cD_\lambda\\[0.5ex] \xi_{m_1\dotsm m_t}\in\cD_\lambda}}
\widetilde{\Pi}^\lambda_{j_1\dotsm j_r}\widetilde{F}_n^\lambda 
\widetilde{E}_n^\lambda\widetilde{\res}^\lambda_{\xi_{m_1\dotsm m_t}} 
\intertext{and} 
{E}_n^{\xi}
{F}_n^{\xi}
&\cong
\bigoplus\limits_{\substack{\xi_{i_1\dotsm i_k}\in\cD_\lambda\\[0.5ex] \xi_{j_1\dotsm j_r}\in\cD_\lambda}}
\widetilde{\Pi}^\lambda_{i_1\dotsm i_k}\widetilde{E}_n^\lambda 
\widetilde{F}_n^\lambda\widetilde{\res}^\lambda_{\xi_{j_1\dotsm j_r}}  
\end{align*}
and the claim follows. 
\end{proof}

\begin{cor}
The functors $\{F_j^\xi, E_j^\xi\}_{\in \{1,\dotsc ,n\}}$ define a categorical $\sll$-action on 
$$\bigoplus\limits_{\xi_{i_1\dotsm i_k}\in\cD_\lambda}R^{\xi_{i_1\dotsm i_k}(\lambda)}_n-\amod .$$  
\end{cor}

\begin{cor}
With the action $E_n^\xi$ and $F_n^\xi$ as above the surjection $K_0(\Pi^\lambda)$ in  
Proposition~\ref{prop:branchrules} is a surjection of $\sll$-representations. 
\end{cor}

\begin{proof}
It is enough to show that the functor $\Pi^\lambda$ intertwines the categorical 
$\mathfrak{sl}_2$-action defined by $\{F_n^\xi,E_n^\xi\}$. 
For an object $M$ in $R_{n+1}^\lambda-\prmod$ we have 
\begin{align*}
F_n^\xi\Pi^\lambda(M)
&=\widetilde{\Pi}^\lambda\widetilde{F}_n^\lambda\widetilde{res}^\lambda\Pi^\lambda(M)
\\
&=\widetilde{\Pi}^\lambda \cQ{F}_n^\lambda P(\widetilde{res}^\lambda\Pi^\lambda(M))
\\
&=\widetilde{\Pi}^\lambda \widetilde{\res}^\lambda\Pi^\lambda{F}_n^\lambda P(\widetilde{\res}^\lambda\Pi^\lambda(M))
\\
&=\Pi^\lambda{F}_n^\lambda P(\widetilde{\res}^\lambda\Pi^\lambda(M))
\end{align*} 
The claim now follows from a
 comparison between the vector spaces $P(\widetilde{\res}^\lambda\Pi^\lambda(M))$ and $M$. 
\end{proof}

\subsection{The cyclotomic quotient conjecture revisited}%
\label{ssec:easyBK}

We can now give an elementary proof of the Khovanov-Lauda cyclotomic conjecture in Type $A$. 
Recall that from Proposition~\ref{prop:branchrules} we have a surjection of  $\sll$-representations 
\begin{align*} 
K_0(\Pi^{\lambda} )
\colon 
K_0 ( R^\lambda_{n+1} )
\to 
\bigoplus\limits_{\xi_{i_1\dotsm i_k}\in\cD_\lambda}K_0(R^{\xi_{i_1\dotsm i_k}(\lambda)}_n)  
\end{align*}
and so, if we know that 
$K_0(R^{\xi_{i_1\dotsm i_k}(\lambda)}_n) \cong  V^{\xi_{i_1\dotsm i_k}(\lambda)}_{\sln}$
we are done. 
The cyclotomic conjecture for $\sll$ follows from the cyclotomic conjecture for $\mathfrak{sl}_2$  
by recursion, which in turn is a consequence of the fact that in $R^{\mu}_2$ we have $1_{\bar{\mu}+1}=0$, where 
$1_{\bar{\mu}+1}=0$ is the diagram consisting of ${\bar{\mu}+1}$ vertical parallel strands.

\subsection{Classes of special indecomposables and the Gelfand-Tsetlin basis}%

Applying the procedure described in Section~\ref{sec:catbranching} 
recursively we end up with a direct sum 
of $d_\lambda:=\dim(V_\lambda^{\sll})$ one-dimensional $\Bbbk$-vector spaces. 
We  now reverse this procedure. 

\begin{defn}
For each Gelfand-Tsetlin pattern $s\in\cS(\lambda)$ we define a functor
\begin{equation*}
\res^s := \res^{\lambda}_{\mu^n}\res^{\mu^n}_{\mu^{n-1}}\dotsm\res^{\mu^1} 
\colon 
\Bbbk-\amod \to R_{n+1}^\lambda-\amod.
\end{equation*}
\end{defn}
Here each of the functors $\res^{\mu^j}_{\mu^{j-1}}$ is the restriction functor 
corresponding to the surjection $\pi_{\mu^j,\mu^{j-1}}\colon R_j^{\mu^{j}}\to R_{j-1}^{\mu^{j-1}}$, 
as in Lemma~\ref{lem:cycinc} which uniquely determines a sequence  
$(i_{1_j}\dotsm i_{k_j})$.
Functor $\res^s$ takes the one-dimensional $\Bbbk$-module $\Bbbk$ to the module 
${}_{e(p_{i_{1_1}\dotsm i_{k_1}},p_{i_{1_2}\dotsm i_{k_2}}, \dotsc ,p_{i_{1_n}\dotsm i_{k_n}})}L$ which is the $R_{n+1}^\lambda$-module 
consisting of $n$ sets of non-intersecting strands, labeled by the 
order given by the $p_{i_{1_r}\dotsm i_{k_r}}$, carrying no dots, and ending in the sequence 
determined by the idempotent $e(p_{i_{1_1}\dotsm i_{k_1}},p_{i_{1_2}\dotsm i_{k_2}},\dotsc ,p_{i_{1_n}\dotsm i_{k_n}})$.  
To keep the notation simple from now on we write 
$\pi_s$ instead of $\pi_{i_{1_n}\dotsm i_{k_n}}\dotsm\pi_{i_{1_n}\dotsm i_{k_2}}\pi_{i_{1_1}\dotsm i_{k_1}}$ 
and $e(s)$ instead of 
$e(p_{i_{1_1}\dotsm i_{k_1}},p_{i_{1_2}\dotsm i_{k_2}}, \dotsc ,p_{i_{1_n}\dotsm i_{k_n}})$. 

\begin{defn}\label{def:quotGT}
The algebra $\widecheck{R}_{n+1}^\lambda$ is defined as the quotient
\begin{equation}
\widecheck{R}_{n+1}^\lambda = 
\bigoplus\limits_{s\in\cS(\lambda)} R_{n+1}^\lambda/\ker(\pi_s). 
\end{equation}
\end{defn}

As in the case of the algebras $\widetilde{R}_{n+1}^\lambda$ of Subsection~\ref{ssec:catproj}
each algebra $R_{n+1}^\lambda/\ker(\pi_s)$ admit a presentation by diagrams consisting of groups of strands 
labeled by the entries $s(j)$ of the string $s$. 
Strands within the same group cannot cross among themselves and are labeled in the order given by the idempotent 
$e(p_{i_{1_j}\dotsm i_{k_j}})$ determined by $s(j)$.

\begin{lem}
We have $\Hom_{\widecheck{R}_{n+1}^\lambda-\amod}({}_{e(s)}L,{}_{e(s')}L)=0$ if $s\neq s'$. 
\end{lem}

\begin{lem}
Module ${}_{e(s)}L$ is projective indecomposable 
as a module over $\widecheck{R}_{n+1}^\lambda$. 
\end{lem}

Denote by $\widecheck{\psi}: R_{n+1}^\lambda\to \widecheck{R}_{n+1}^\lambda $ the projection map. 
The action of functors $\{F_i^\lambda, E_i^\lambda\}_{i\in\{1,\dotsc n\}}$ on 
the collection of all the modules ${}_{e(s)}L$ in 
$\widecheck{R}_{n+1}^\lambda-\amod$ inherited from the one on 
$R_{n+1}^\lambda-\amod$ 
does not commute with the quotient functor  
$\widecheck{\Psi}\colon R_{n+1}^\lambda-\amod\to\widecheck{R}_{n+1}^\lambda-\amod$ 
induced by $\widecheck{\psi}$. 
Nevertheless 
if we change the action of 
the $F_i^\lambda$s and of the $E_i^\lambda$s on 
$\widecheck{R}_{n+1}^\lambda-\amod$
we obtain a commutative diagram 
\begin{equation}\label{eq:cdiagGT}
\raisebox{9.5mm}{
\xymatrix@C=18mm{
R_{n+1}^\lambda-\amod \ar[d]^{\widecheck{\Psi}}\ar[r]^{F_i^{\lambda} , E_i^{\lambda}} 
&  
R_{n+1}^\lambda-\amod \ar[d]^{\widecheck{\Psi}} 
\\
\widecheck{R}_{n+1}^\lambda-\amod \ar[r]^{\widecheck{F}_i^{\lambda} , \widecheck{E}_i^{\lambda}} 
& 
\widecheck{R}_{n+1}^\lambda-\amod 
} }
\end{equation}
To this end we define 
\begin{equation*}
\widecheck{F}_j^{\lambda}, \widecheck{E}_j^{\lambda}
\colon
\widecheck{R}_{n+1}^{\lambda}-\amod\to \widecheck{R}_{n+1}^{\lambda}-\amod
\end{equation*}
as the composite functors
\begin{equation*}
\widecheck{F}_j^{\lambda} = 
\bigoplus\limits_{s_1,s_2,s_3\in\,\cS(\lambda)}
\widetilde{\res}^\lambda_{s_3}\widetilde{\Pi}^\lambda_{s_2}F_j^\lambda\widetilde{\res}^\lambda_{s_1}\Pi^\lambda_{s_1}
\mspace{55mu}
\widecheck{E}_j^{\lambda} = 
\bigoplus\limits_{s_1,s_2,s_3\in\,\cS(\lambda)}
\widetilde{\res}^\lambda_{s_3}\widetilde{\Pi}^\lambda_{s_2}E_j^\lambda\widetilde{\res}^\lambda_{s_1}\Pi^\lambda_{s_1} .
\end{equation*}

\begin{lem}
Functors $\widecheck{F}_j^{\lambda}$ and $\widecheck{E}_j^{\lambda}$ are biadjoint and take projectives to projectives.
\end{lem}

\begin{prop}
The collection of endofunctors $\{\widetilde{F}_i^\lambda, \widetilde{E}_i^\lambda\}_{i\in\{1,\dotsc n\}}$ 
defines a categorical $\sll$-action on $\widecheck{R}_{n+1}^\lambda-\amod$.
The functor $\widecheck{\Psi}$ intertwines the categorical $\sll$-action. 
\end{prop}

\begin{thm}\label{thm:GTcat}
There is an isomorphism of $\sll$ representations 
\begin{equation*}
K_0(\widecheck{R}_{n+1}^\lambda)\xra{\ \ \cong\ \ } V_\lambda^{\sll} 
\end{equation*} 
taking the projective ${}_{e(s)}L$ to the Gelfand-Tsetlin basis element $\ket{s}$. 
\end{thm}
\begin{proof}
The surjection $\widecheck{\Psi}$ of algebras induces a surjective map between the Grothendieck groups
\begin{equation*}
K_0(\widecheck{\psi})\colon 
K_0(R_{n+1}^\lambda)\xra{\ \ \cong\ \ } K_0(\widecheck{R}_{n+1}^\lambda)
\end{equation*} 
intertwining the action of $\sll$ 
which is an isomorphism if $K_0(\widecheck{R}_{n+1}^\lambda)$ is not zero, by Schur's lemma. 
To prove it is not zero we use the categorical branching rule to reduce  the size of the category
$\widecheck{R}_{n+1}^\lambda-\prmod$ 
recursively until we get something with nonzero $K_0$. 
Chose a string $s\in\cS(\lambda)$. 
Each surjection 
$\pi_{i_1\dotsm i_k}\colon R_{n+1}^\lambda(k\alpha_n)\to R_{n}^{\xi_{i_1\dotsm i_k}(\lambda)}$ 
from Lemma~\ref{lem:cycinc} induces a surjection 
$\widecheck{\pi}_{i_1\dotsm i_k}\colon\widecheck{R}_{n+1}^\lambda(k\alpha_n)\to\widecheck{R}_{n}^{\xi_{i_1\dotsm i_k}(\lambda)}$ 
which in turn results in a map 
$K_0(\widecheck{\pi}_{i_1\dotsm i_k})\colon K_0(\widecheck{R}_{n+1}^\lambda(k\alpha_n))\to K_0(\widecheck{R}_{n}^{\xi_{i_1\dotsm i_k}(\lambda)})$ 
which is surjective.  
Continuing recursively we end up with a chain of surjections  
\begin{equation*}
K_0(\widecheck{R}_{n+1}^\lambda(k\alpha_n))\to \dotsm \to K_0(\widecheck{R}_{1}^{\mu^1})=K_0(\Bbbk)\neq 0  
\end{equation*}
which implies that 
$K_0(\widecheck{R}_{n+1}^\lambda)=\bigoplus_{k\geq 0}K_0(\widecheck{R}_{n+1}^\lambda(k\alpha_n))$ 
is nonzero. 

The second claim follows from the fact that every indecomposable  
in $R_{n+1}^\lambda-\prmod$ splits under $\widecheck{\Psi}$ into a direct sum of 
indecomposables in $\widecheck{R}_{n+1}^\lambda-\prmod$, each one labeled by an element of  $\cS(\lambda)$ 
together with the fact that the number of projective indecomposables is the same in both categories 
and the already established result that the map $K_0(\widecheck{\psi})$  
is an isomorphism. 
\end{proof}

\medskip

The results above allow us to give a  presentation of the category 
$R_{n+1}^\lambda-\prmod$ in terms of the Gelfand-Tsetlin basis using  
the idempotents $e(s)$.

\begin{prop}
Every object in $R_{n+1}^\lambda-\prmod$ is isomorphic to a direct summand of 
some ${}_{e(s)}P\{ \varepsilon_s \}$, for some $s\in\cS(\lambda)$ and some shift $\varepsilon_s$. 
\end{prop}
\begin{proof}
Every object in $\widecheck{R}_{n+1}^\lambda-\prmod$ is a quotient of an object in 
${R}_{n+1}^\lambda-\prmod$. 
An inductive argument, starting with the modules ${}_{i_1\dotsm i_k,\und{j}}L$ of Subsection~\ref{ssec:catproj}, 
shows that each object ${}_{e(s)}L$ in $\widecheck{R}_{n+1}^\lambda-\prmod$ has a projective cover in 
${R}_{n+1}^\lambda-\prmod$ which coincides with ${}_{e(s)}P$. 
The claim follows from this observation together with 
Lemmas~\ref{lem:frobKLR},~\ref{lem:BRfull} and Corollary~\ref{cor:BRinj}. 
\end{proof}

\begin{thm}
The isomorphism $K_0(R_{n+1}^\lambda)\to V_\lambda^{\sll} $ of Theorem~\ref{thm:BK}  sends the projective 
${}_{e(s)}P$ to the Gelfand-Tsetlin basis element $\ket{s}$. 
\end{thm}

This basis is not orthogonal with respect to the $q$-Shapovalov form, but 
it can be use to redefining another bilinear form $(\ \,,\ )$ on $V_\lambda^{\sll}$ as 
\begin{equation*}
( [P],[P'] ) := \gdim\Hom_{\widecheck{R}_{n+1}^\lambda-\amod}( P , P' )
\end{equation*}
for $P$, $P'$ objects in ${R}_{n+1}^\lambda-\prmod$,  
clearly giving $( {}_{e(s)}P, {}_{e(s')}P )=0$ if $s\neq s'$.

\subsection{A functorial realisation of the Gelfand-Tsetlin basis}%

For each $s\in\cS(\lambda)$ we also have functors
\begin{align*}
\Pi^{s} := \Pi^{\mu^1}\dotsm\Pi^{\mu^n}\Pi^\lambda 
&\colon 
R_{n+1}^\lambda-\amod \to \Bbbk-\amod 
\\[1ex]
\widetilde{\Pi}^{s} := \widetilde{\Pi}^{\mu^1}\dotsm\widetilde{\Pi}^{\mu^n}\widetilde{\Pi}^\lambda 
&\colon 
\widetilde{R}_{n+1}^\lambda-\amod \to \Bbbk-\amod 
\intertext{and}
\widetilde{\res}^{s} := \widetilde{\res}^{\lambda}\widetilde{\res}^{\mu^n}\dotsm\widetilde{\res}^{\mu^1} 
&\colon 
\Bbbk-\amod \to \widetilde{R}_{n+1}^\lambda-\amod 
\end{align*}
with the obvious definition of the categories $\widetilde{R}_{j}^{\mu^j}-\amod$.

\begin{lem}
Functors $\Pi^s$ have orthogonal hom-spaces, in the sense that 
for an $R_{n+1}^\lambda$-module $M$ we have that $\Hom_{\oplus_{d_\lambda}\Bbbk-\amod}(\Pi^s(M),\Pi^{s'}(M))=0$ 
if $s\neq s'$. 
\end{lem}

Let $\cG\cT(\lambda)$ denote the category of functors 
$$\Fun \colon R_{n+1}^\lambda-\amod \to \Bbbk-\amod$$ 

There are endofunctors acting on $\cG\cT(\lambda)$ defined by 
\begin{align*}
F_i^{\cG\cT}\phi(M) := 
\bigoplus\limits_{r,s\in\cS(\mu^{i-1})} \widetilde{\Pi}^{r}F_i^{\mu^{i-1}} \widetilde{\res}^{s}\phi(M)
\intertext{and}
E_i^{\cG\cT}\phi(M) := 
\bigoplus\limits_{r,s\in\cS(\mu^{i-1})} \widetilde{\Pi}^{r}E_i^{\mu^{i-1}} \widetilde{\res}^{s}\phi(M)
\end{align*}
for $\phi$ a functor $R_{n+1}^\lambda-\amod \to \Bbbk-\amod$, 
$M$ an $R_{n+1}^\lambda$-module and  $i=1,\dotsc ,n$.

\begin{lem}
Each pair of functors $F_i^{\cG\cT},E_i^{\cG\cT}$ is biadjoint. 
\end{lem}

\begin{prop}
The collection of functors $\{ F_i^{\cG\cT},E_i^{\cG\cT}\}_{i=1,\dotsc ,n}$ defines a categorical 
$\sll$-action on $\cG\cT(\lambda)$.
\end{prop}

\begin{conj}
We have an isomorphism of $\sll$-modules $K_0(\cG\cT(\lambda))\cong V_\lambda^{\sll}$
that takes $\Pi^s$ to the Gelfand-Tsetlin basis element $\ket{s}$.
\end{conj}

\section{Cyclotomic KLR algebras categorify Weyl modules}\label{sec:appl}

\subsection{The $q$-Schur categorification}\label{ssec:qshur}

In~\cite{MSV} a diagrammatic categorification of the $q$-Schur 
algebra was constructed using a quotient of Khovanov and Lauda's 
categorified quantum groups from~\cite{KL3, KL:err}.
Khovanov and Lauda's categorified quantum $\mathfrak{sl}_n$ 
consists of a 2-category ${\cU}(\mathfrak{sl}_n)$ defined from the following data. 
The objects are \emph{weights} $\lambda\in\bZ^{n-1}$. 
The 1-morphisms are products of symbols 
$\lambda'\mathcal{F}_{i}\lambda$ (with 
$\lambda'_j = \lambda_j+1$ if $j=i\pm 1$, $\lambda'_j = \lambda_j-2$ if $j=i$, 
and $\lambda'_j = \lambda_j$ otherwise) 
and 
$\lambda'\mathcal{E}_{i}\lambda$ (with 
$\lambda'_j = \lambda_j-1$ if $j=i\pm 1$, $\lambda'_j = \lambda_j+2$ if $j=i$, 
and $\lambda'_j = \lambda_j$ otherwise) 
with the convention that says that
$\lambda'\mathcal{F}_{i}\mu\nu\mathcal{F}_{i}\lambda$
and 
$\lambda'\mathcal{E}_{i}\mu\nu\mathcal{E}_{i}\lambda$ 
are zero unless $\mu=\nu$.
The 2-morphism of ${\cU}(\mathfrak{sl}_n)$ are given by planar diagrams in a strip generated 
by oriented arcs that can intersect transversely and can be decorated with dots 
(closed oriented 1 manifolds are allowed). 
This graphical calculus is a generalization of the KLR algebras to a calculus where the strands 
can travel in all directions in the sense that 
it gives the KLR diagrammatics when we restrict strands to travel only downwards.  
The boundary of each arc is decorated with a 1-morphism. 
These 2-morphisms are subject to a set of relations which we do not give here 
(see~\cite{KL3, MSV} for details).

In~\cite{MSV} Khovanov and Lauda's categorified 
quantum $\mathfrak{sl}_n$ was upgraded to a categorification ${\cU}(\mathfrak{gl}_n)$ of quantum $\mathfrak{gl}_n$
(taking Khovanov and Lauda's diagrams and relations of ${\cU}(\mathfrak{sl}_n)$ with $\mathfrak{gl}_n$-weights)
and define the categorification of $S_q(n,d)$ as the quotient of ${\cU}(\mathfrak{gl}_n)$
by 2-morphisms factoring through a weight not in $\Lambda(n,d)$.

\begin{defn}
The category $\cS(n,d)$ is the quotient of $\cU(\mathfrak{gl}_n)$ by the ideal generated by all 2-morphisms 
containing a region with a label not in $\Lambda(n,d)$. 
\end{defn}

The main result of~\cite{MSV} is that $\cS(n,d)$ categorifies the $q$-Schur algebra 
from Subsection~\ref{ssec:schur}. 

\begin{thm}[\cite{MSV}]\label{thm:K0schur}
There is an isomorphism of $\bQ(q)$-algebras  
\begin{equation*}
\gamma\colon \dot{\bf{S}}(n,d) 
\xra{\ \ \cong  \ \ }  
K_0\bigl(\Kar(\cS(n,d))\bigr) .
\end{equation*}
\end{thm}

\subsection{Categorical Weyl modules}\label{ssec:weyl}

Recall that 
$$W_{\lambda}\cong 1_{\lambda}\dot{\bf{S}}(n,d)/[\mu>\lambda]$$ 
where ``$>$'' is the lexicographic order,  
is an irreducible for $\dot{\bf{S}}(n,d)$ and that 
all irreducibles can be obtained this way.
It was conjectured in~\cite{MSV} that it is easy to categorify the irreducible 
representations $W_{\lambda}$, for $\lambda\in\Lambda^+(n,d)$, using the 
category $\cS(n,d)$.

\begin{defn} 
For any $\lambda\in\Lambda^+(n,d)$, let 
$1_{\lambda}\cS(n,d)$ be the category whose objects are the $1$-morphisms 
in $\cS(n,d)$ of the form $1_{\lambda}x$ and whose morphisms are the 
$2$-morphisms in $\cS(n,d)$ between such 1-morphisms. 
Note that 
$1_{\lambda}\cS(n,d)$ does not have a monoidal structure, because 
two $1$-morphisms $1_{\lambda}x$ and $1_{\lambda}y$ cannot be composed in general. 
Alternatively one can see $1_{\lambda}\cS(n,d)$ as a graded ring, whose 
elements are the morphisms. 
\end{defn}
\begin{defn}
Let ${\cV}_{\lambda}$ be the quotient of $1_{\lambda}\cS(n,d)$ by 
the ideal generated by all diagrams which contain a region labeled by 
$\mu>\lambda$. 
\end{defn} 

There is a natural categorical action of $\cS(n,d)$, and therefore 
of $\cU(\mathfrak{sl}_n)$, on ${\cV}_{\lambda}$, 
defined by putting a diagram in $\cS(n,d)$ on the right-hand side of a 
diagram in ${\cV}_{\lambda}$. This action descends to an action of 
$\dot{\bf{S}}(n,d)\cong K_0\bigl(\Kar{\cS(n,d)}\bigr)$ on 
$K_0(\Kar({\cV}_{\lambda}))$. 
The map  
$\gamma$ from Theorem~\ref{thm:K0schur}  
induces a well-defined linear map 
$\gamma_{\lambda}\colon W_{\lambda}\to K_0(\Kar({\mathcal V}_{\lambda}))$, which 
intertwines the $\dot{\mathbf{S}}(n,d)$-actions.  
It was proved in~\cite{MSV} that 
$\gamma_{\lambda}$ is surjective and 
it was conjectured that it is an isomorphism.  
Since $W_{\lambda}$ 
is irreducible, we have $K_0(\Kar({\cV}_{\lambda}))\cong V_{\lambda}$ or 
$K_0(\Kar({\cV}_{\lambda}))=0$. 
So it suffices to show that 
$K_0(\Kar({\cV}_{\lambda}))\ne 0$.

\medskip

From now on we regard $R^{\lambda}_{n+1}$ as the category whose objects 
are sequences of simple roots and morphisms are KLR diagrams. 
Let ${\mathcal N}_\lambda$ be the two-sided ideal generated by diagrams of 
$R^{\lambda}_{n+1}$ 
containing a bubble of positive degree in its left-most region. 

\begin{defn}
The category $\widecheck{{\cV}}_{\lambda}$ is the quotient 
of  ${\cV}_{\lambda}$ by ${\mathcal N}_\lambda$.
\end{defn}

The ideal ${\mathcal N}_\lambda$ is virtually-nilpotent and therefore 
$\widecheck{{\cV}}_{\lambda}$ has the same Grothendieck group as 
${\cV}_{\lambda}$ (see~\cite[Sec. 7]{MSV}) where it was also explained that this quotient satisfies the cyclotomic 
condition from Definition~\ref{def:cycKLR}. 
In~\cite{MSV} there was defined a functor from $R^{\lambda}_{n+1}$ to $\widecheck{{\cV}}_{\lambda}$ 
which is the identity on objects and morphisms where the strands in the diagrams 
of $R^{\lambda}_{n+1}$ are seen as secretely oriented downwards. 
This functor is clearly full and essentially surjective and it was conjectured to be faithful. 
We denote this functor $\Phi_\lambda$. 
The main result of this section is the following.

\begin{thm}\label{thm:weyl} 
The functor $\Phi_\lambda$ is faithful and therefore an equivalence of categories. 
\end{thm}

\begin{proof}
We can decorate the regions of the diagrams 
of $R_{n+1}^{\lambda}$ 
with $\mathfrak{gl}_{n+1}$-weights, 
starting with a $\lambda$ in the leftmost region and subtracting $\varepsilon_j-\varepsilon_{i+1}$ 
any time we cross a strand labeled $j$. 
In other words, if the region on the left of strand labelled $j$ is decorated with the weight $\lambda'$
then the label of the region immediately at its right is
$\lambda' -\varepsilon_j+\varepsilon_{j+1} =
(\lambda'_1,\dotsc ,\lambda'_j-1,\lambda'_{j+1}+1,\dotsc ,\lambda'_{n+1}).$

We first prove that 
if $X\in R_{n+1}^{\lambda}(\beta)$ contains a region labeled by $\mu\notin\Lambda^{\mathfrak{gl}_{n+1}}$  
then $X=0$.  
It is enough to assume that 
$\mu$ is the label of its rightmost region. 
Moreover we can assume that $\mu_{n+1}<0$. 
For suppose $\mu_j<0$ and $\mu_i\geq 0$ for $i>k$.
Then 
we can use the decomposition in~\eqref{eq:brfunctor}  
and the fact that $\Pi^{\lambda}$ is injective on objects
to obtain an array of diagrams, each one in a distinct $R^{\xi_{i_1\dotsm i_k}(\lambda)}_{n}$, 
but all having the weight $(\mu_1,\dotsc ,\mu_{n})$ in its rightmost region. 
A recursive application of this procedure yields therefore an array of diagrams 
in a direct sum of cyclotomic KLR algebras $\oplus_{\zeta}R^{\zeta}_j$, 
all of them with the rightmost region decorated with   
$(\mu_1,\dotsc ,\mu_{j})$. 
We can assume further that $X$ 
is of the form $1_{\und{r}}$ for some sequence $r$ of simple roots. 

Assume that one of the components $\Pi^{\xi_{i_1\dotsm i_k}(\lambda)}1_{\und{r}}$ is nonzero. 
Then we have a nonzero diagram 
in $R_{m}^{\lambda}$ connecting the special idempotent $e(p_{i_1\dotsm i_k},\und{r}')$ to $1_{\und{r}}$,  
as in 
\begin{equation*}
\labellist
\small 
\pinlabel $\lambda$ at -40 140 
\pinlabel $\mu$ at 420 136
\pinlabel $\widetilde{\mu}$ at 150 46
\tiny \hair 2pt
\pinlabel $\dotsc$ at  35 190 
\pinlabel $\dotsc$ at  330 190 
\pinlabel $\dotsc$ at 195 190 
\pinlabel $\dotsc$ at  35  40
\pinlabel $\dotsc$ at 330  40
\pinlabel $r_1$ at   4 220
\pinlabel $n$   at 244 219
\pinlabel $r_{\vert\beta\vert}$ at 372 215
\pinlabel $i_1$ at   4 15 
\pinlabel $n$ at 104 15
\endlabellist 
\figins{-20.5}{1.1}{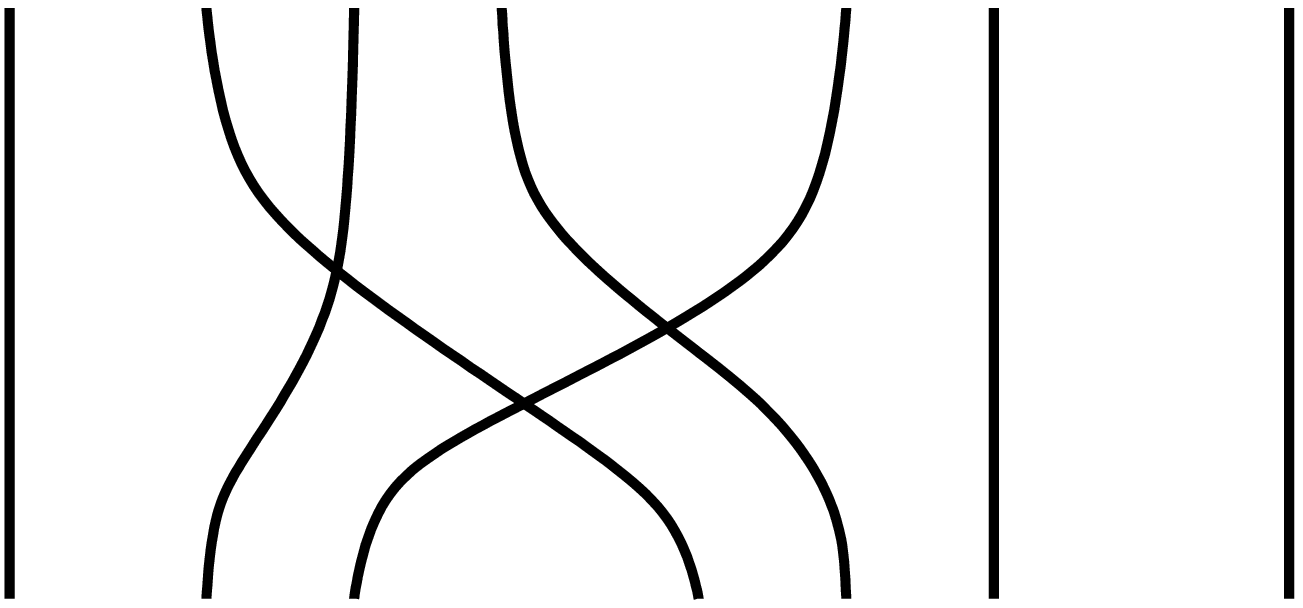}
\end{equation*}
Recall that the strands ending in $p_{i_1\dotsm i_k}$ in the bottom 
do not cross each other nor carry any dots.  
All strands labelled $n$ must end at the bottom 
among the ones corresponding to $p_{i_1\dotsm i_k}$.   
We label $\widetilde{\mu}$ the region close to the bottom of $1_{\und{r}}$ and
immediately at the right of the 
last strand labelled $n$, counted from the left.

Let $\vert\alpha_{n-1}\vert_\beta$ and $\vert\alpha_{n}\vert_\beta$ 
be the number of strands labelled $n-1$ and $n$ in $R^\lambda_m(\beta)$ respectively. 
Since $\mu_n < 0$ we must have $\vert\alpha_{n-1}\vert_\beta < \vert\alpha_{n}\vert_\beta$ 
which means that $\widetilde{\mu}_n\leq \mu_n <0$. 
This implies that $e(p_{\und{s}},\und{r}')$ is not a special idempotent, which is a contradiction. 
This forces the component
$\Pi^{\xi_{i_i\dotsm i_k}(\lambda)}1_{\und{r}}$ to be the zero diagram. 
The reasoning above applies to all components $\Pi^{\xi_{i_i\dotsm i_k}(\lambda)}1_{\und{r}}$ 
and altogether, it implies that $\Pi^{\lambda}1_{\und{r}}=0$. 
Since $\Pi^{\lambda}$ 
is injective we conclude that $1_{\und{r}}=0$ in $R_{n+1}^{\lambda}(\beta)$.  
\end{proof}

\begin{cor} 
We have an isomorphism of $\dot{S}(n,d)$ representations
\begin{equation*}
K_0(\Kar({\cV}_\lambda))\cong W_{\lambda} .
\end{equation*}
\end{cor}





\end{document}